\documentclass[11pt]{amsart}
\usepackage{amssymb,amsmath,amsthm,verbatim,enumitem}
\usepackage{mathtools}
\usepackage{fullpage}
\usepackage[dvipsnames]{xcolor}
\usepackage{url}
\usepackage{array} 
\usepackage{mathrsfs}
\usepackage[all]{xy}
  \SelectTips{cm}{10}
  \everyxy={<2.5em,0em>:}
\usepackage{tikz}
\usepackage{picinpar} 

\usepackage[new]{old-arrows}

\usepackage{ifpdf}
  \ifpdf
    \usepackage[hypertexnames=false]{hyperref} 
  \else
    \newcommand{\texorpdfstring}[2]{#1}
    \newcommand{\href}[2]{#2}
  \fi

\newcommand{\AC}[1]{{\color{ForestGreen} AC: #1}}
  \DeclareFontFamily{OT1}{pzc}{}
  \DeclareFontShape{OT1}{pzc}{m}{it}{<-> s * [1.100] pzcmi7t}{}
  \DeclareMathAlphabet{\mathpzc}{OT1}{pzc}{m}{it}

\usepackage[OT1]{fontenc}

\usepackage[english]{babel}

\usepackage[utf8]{inputenc}

\usepackage{amsmath,amssymb,amsfonts,mathrsfs}

\usepackage{graphicx}

\usepackage{pdfpages}

\usepackage{varioref}

\usepackage{datetime}

\usepackage{mathtools}

\usepackage[h]{esvect}

\usepackage{array}

\usepackage{listings}
\usepackage{microtype}

\usepackage{booktabs}

\usepackage{tikz-cd}
\usetikzlibrary{decorations.pathreplacing}
\usetikzlibrary{calc,tikzmark}

\theoremstyle{plain}
  \newtheorem{theorem}{Theorem}[section]
  \newtheorem{lemma}[theorem]{Lemma}
  \newtheorem{proposition}[theorem]{Proposition}
  \newtheorem{corollary}[theorem]{Corollary}

    \newtheorem{conjecture}[theorem]{Conjecture}

\theoremstyle{definition}
  \newtheorem{definition}[theorem]{Definition}
  
  \newtheorem{example}[theorem]{Example}
  \newtheorem{remark}[theorem]{Remark}
  
\numberwithin{figure}{section}
\usepackage[all]{hypcap}

\renewcommand{\thesection}{\arabic{section}}
\renewcommand{\theequation}{\thesection.\arabic{equation}}

 \DeclareFontFamily{U}{manual}{}
 \DeclareFontShape{U}{manual}{m}{n}{ <->  manfnt }{}
 \newcommand{\manfntsymbol}[1]{%
    {\fontencoding{U}\fontfamily{manual}\selectfont\symbol{#1}}}

\makeatletter
   \@addtoreset{section}{part}
   \@addtoreset{equation}{section}
   \@addtoreset{footnote}{section}

    {\hspace*{\fill}$\lrcorner$\endgraf\endgroup\end{trivlist}}
\makeatother

\def\f{\varphi}

\def\O{\mathcal{O}}
\def\a{\alpha}

\def\lam{\lambda}

\def\ZZ{\mathbb{Z}}

\def\QQ{\mathbb{Q}}

\def\PP{\mathbb{P}}

\def\L{\mathcal{L}}
\def\T{\mathsf{T}}
\def\sI{\mathsf{I}}

\def\e{\varepsilon}
\def\O{\mathcal{O}}
\renewcommand{\O}{\mathcal{O}}

\def\CC{\mathbb{C}}

\def\|{\Vert}

\newcommand\equa[1]{\begin{equation}\begin{split}#1\end{split}\nonumber\end{equation}}
\newcommand\equanum[1]{\begin{equation}\begin{split}#1\end{split}\end{equation}}

\newcommand\scr[1]{\mathscr{#1}}
\def\seq{\subseteq}

\def\tr{\operatorname{Tr}}

\def\hom{\operatorname{Hom}}
\def\End{\operatorname{End}}
\def\NO{\operatorname{NO}}

\def\shom{\scr{H}om}

\def\td{\operatorname{td}}
\def\quot{\operatorname{Quot}}
\def\rk{\operatorname{rank}}

\def\tr{\operatorname{tr}}

\def\gl{\operatorname{SL}}

\def\vir{\textnormal{vir}}
\def\Exp{\operatorname{Exp}}

\def\hilb{\operatorname{Hilb}}
\def\Span{\operatorname{span}}
\def\gl{\operatorname{GL}}

\def\pt{\text{pt}}
\def\redu{\text{red}}
\def\dt4{\text{DT}_4}
\def\ch{\operatorname{ch}}
\def\sym{\operatorname{Sym}}
\def\ext{\operatorname{Ext}}

\def\loc{\text{loc}}

\def\->{\rightarrow}
\def\=>{\Rightarrow}

\def\<{\langle}
\def\>{\rangle}

\title{Equivariant Segre and Verlinde invariants for Quot schemes}

\author{Arkadij Bojko}
\address{ETH Zürich \\
Department of Mathematics \\
Rämistrasse 101\\
8092 Zürich}
\email{arkadij.bojko@math.ethz.ch}
\author{Jiahui Huang}
\address{ETH Zürich \\
Department of Mathematics \\
Rämistrasse 101\\
8092 Zürich}
\email{huangjia@student.ethz.ch}

\begin{document}


\begin{abstract}
The problem of studying the two seemingly unrelated sets of invariants forming the Segre and the Verlinde series has gone through multiple different adaptations including a version for the virtual geometries of Quot schemes on surfaces and Calabi-Yau fourfolds. 
Our work is the first one to address the equivariant setting for both $\CC^2$ and $\CC^4$ by examining higher degree contributions which have no compact analogue. 
\begin{enumerate}
    \item For $\CC^2$, we work mostly with virtual geometries of Quot schemes. After connecting the equivariant series in degree zero to the existing results of the first author for compact surfaces, we extend the Segre-Verlinde correspondence to all degrees and to the reduced virtual classes. Apart from it, we conjecture an equivariant symmetry between two different Segre series building again on previous work.
    \item For $\CC^4$, we give further motivation for the definition of the Verlinde series. Based on empirical data and additional structural results, we conjecture the equivariant Segre-Verlinde correspondence and the Segre symmetry analogous to the one for $\CC^2$.
\end{enumerate}
\end{abstract}

\maketitle
\tableofcontents

\section{Introduction}

\subsection{Definitions of Segre and Verlinde invariants}

Let $Y$ be a smooth quasi-projective variety. For a torsion-free sheaf $E$, the Quot scheme $\quot_Y(E,n)$ parameterizes quotients
\[E\twoheadrightarrow F\]
such that
\[\rk(F)=0,\quad c_1(F)=0,\quad \chi(F)=n.\]
When $Y=\CC^2$ is a toric surface or $Y=\CC^4$ is a toric Calabi-Yau 4-fold, we shall define equivariant Segre and Verlinde invariants on these Quot schemes and find relations between them. In the non-equivariant case where $Y$ is a smooth projective curve, surface, or a Calabi-Yau 4-fold, such relations are called Segre-Verlinde correspondences and are studied in \cite{KHilb,boj,Bojko2,GM,G_ttsche_2022,johnson,Marian_2021}.

The Segre and Verlinde invariants are defined by taking integrals or Euler characteristics of various insertions constructed using \emph{tautological bundles} 
 \[V^{[n]}=p_*(\mathcal{F}\otimes q^*V)\]
 on $\quot_Y(E,n)$ constructed for some vector bundle $V$ on $Y$. Here $$\quot_Y(E,n)\xleftarrow{p}\quot_Y(E,n)\times Y\xrightarrow{q}Y$$ are projections, and $\mathcal{F}$ is the universal quotient sheaf. This definition extends to the Grothendieck groups and associates to each $\alpha\in K^0(Y)$ the \textit{tautological class} $\alpha^{[n]}\in K^0(\quot_Y(E,n))$.  

\subsubsection{Equivariant invariants on Hilbert schemes of surfaces}

Let $S$ be a toric quasi-projective surface with an action of $\T=(\CC^*)^2$. Before introducing the equivariant invariants, we give a brief summary of equivariant cohomology and equivariant integration with a further background in Section \ref{sec:equivar-coho}. Given a $\T$-representation $V$, define its \emph{equivariant characteristic classes} by considering the associated bundle 
\[E\T\times_\T V\-> E\T\times_\T\{\pt\}=B\T\]
and taking its characteristic classes in $H^*(B\T)=H^*_\T(\pt)$. Denote $c^\T,s^\T,e_\T,\ch_\T,\td_\T$ the equivariant Chern class, Segre class, Euler class, Chern character, and Todd class respectively. From here forward, we will always use equivariant classes for toric varieties, and we will omit the torus $\T$ from the notations when it is clear from the context. In our equivariant setting, the integration $\int_{\hilb^n(S)}$ denotes the equivariant push-forward, and $\chi$ refers to the equivariant Euler characteristic.

Consider $E=\O_S$, so $\quot_S(\O_S,n)$ is the smooth Hilbert scheme $\hilb^n(S)$ parametrizing ideal sheaves of 0-dimensional subschemes of $S$ of length $n$. The action of $\T$ on $S$ lifts to an action of $\T$ on $\hilb^n(S)$, giving an equivariant structure to $\a^{[n]}$ for any equivariant K-theory class $\a\in K_\T(S)$. The \emph{Segre and Chern series}, first defined by \cite{Tyurin_1994} in the study of Donaldson invariants, are respectively
\equanum{\label{defn:chern-series}I^{\mathcal{S}}(\alpha;q):=&\sum_{n=0}^\infty q^n\int_{\hilb^n(S)}s(\alpha^{[n]}),\\
I^{\mathcal{C}}(\alpha;q):=&\sum_{n=0}^\infty q^n\int_{\hilb^n(S)}c(\alpha^{[n]}).}
These series are related by $c(\alpha^{[n]})=s(-\alpha^{[n]})$. The \emph{Verlinde series} is originally defined for moduli spaces of bundles on curves but it has since then been extended even to moduli spaces of torsion-free sheaves on surfaces in \cite{G_ttsche_2022}. For Hilbert schemes it was defined originally in \cite{EGL} in a slightly different form than we use here:
\equa{I^{\mathcal{V}}(\a;q):=&\sum_{n=0}^\infty q^n\chi(\hilb^n(S),\det(\a^{[n]})).
}

\subsubsection{Virtual invariants on Quot schemes}
In general, Quot schemes are not smooth, in which case we do not have a deformation invariant fundamental class to integrate against. One way to resolve this issue is to work with a virtual fundamental class $[\quot_Y(E,n)]^\vir$.

For a smooth projective surface $S$ and a vector bundle $E$, a perfect obstruction theory for $\quot_Y(E,n)$ of virtual dimension $nN$ was constructed in \cite[Lemma 1]{MOP1} and \cite[Proposition 5]{stark}. When $E$ is a torsion-free sheaf, the same was remarked in \cite{Bojko2}. To this perfect obstruction theory, one may associate a virtual fundamental class $[\quot_S(E,n)]^\vir$ \cite{BF,LiTian} and a virtual structure sheaf $\O^\vir$ \cite{k-local-vir}.  The virtual invariants are defined similarly as before, with the usual fundamental class replaced by the virtual class $[\quot_S(E,n)]^\vir$, and the Euler characteristic replaced by the virtual Euler characteristic $\chi^\vir(-):=\chi(-\otimes \O^\vir)$ which are both invariant under the deformation of the complex structure of $S$.
 
Unlike for surfaces and Fano 3-folds, the usual obstruction theory for a Quot scheme of a Calabi-Yau 4-fold $X$ is not perfect, so the previous method does not induce a virtual fundamental class. However, using the method of \cite{RICOLFI} and tools from \cite{Oh:2020rnj}, a virtual class $[\quot_X(E,n)]^\vir_{o(\mathcal{L})}\in H_{2nN}(\quot_X(E,n),\ZZ)$ was constructed in \cite[§2.1]{Bojko2} for $X$ a \emph{strict} Calabi-Yau 4-fold and $E$ a simple rigid locally-free sheaf. Here $\mathcal{L}$ denotes the determinant line bundle $\det\mathbf{R}\shom_{q}(\mathcal{I},\mathcal{I})$ on $\quot_X(E,n)$ where $\mathcal{I}$ is the universal subsheaf. As indicated by the subscript, this class is dependent on some choice of orientation $o(\mathcal{L})$, that is a choice of a square root of the isomorphism
\[Q:\mathcal{L}\otimes\mathcal{L}\->\O_{\quot_X(E,n)}\]
induced by Serre duality. There is also no canonical virtual structure sheaf, and instead, we have a ``twisted" virtual structure sheaf $\hat{\O}^\vir$. The motivation for this notation comes from its relation to the twisted virtual structure sheaf $\hat{\O}^\vir_{\NO}$ of Nekrasov--Okounkov \cite{NO} which was made precise in \cite[(84), §8]{Oh:2020rnj}. Without going into details, $\hat{\O}^\vir_{\NO}$ is defined from the usual virtual structure sheaf $\O^\vir$ on a moduli space with obstruction theory by twisting with $\text{det}^{\frac{1}{2}}\big(K^\vir\big)$ where $K^\vir$ is the virtual cotangent bundle. Square roots of determinant line bundles appear also in the construction of $\hat{\O}^\vir$, but this time they are used to make the result independent of the choices made. Another distinguishing feature is that 
$$
\hat{\O}^\vir\in K_0\big(\quot_Y(E,n), \ZZ\big[2^{-1}\big]\big)
$$
is not an integer class precisely because of the square-root determinant line bundles. To obtain integer invariants, the first author introduced in \cite[§5.3]{boj} and \cite[§1.4]{Bojko2} the \textit{untwisted virtual structure sheaf}
$$
\O^\vir = \hat{\O}^\vir\otimes \mathsf{E}^{\frac{1}{2}}\,,\qquad \mathsf{E} = \det\big((E^*)^{[n]}\big)\,.
$$
In Section \ref{sec:CY4ktheory}, we motivate this definition further by showing in the case  of $X=\CC^4$ that $\mathsf{E}^{-\frac{1}{2}}$ naturally appears in the construction of $\hat{\O}^\vir$ as the only term that is not an integer class. As a consequence, we prove
\begin{proposition}[Proposition \ref{prop:integral}]
The untwisted virtual structure sheaf is integral:
$$
\O^\vir\in K_0\big(\quot_X(E,n), \ZZ\big)\,.
$$
\end{proposition}

Before defining the equivariant virtual invariants on $Y=\CC^d$, $d=2,4$, we describe the torus action used on $\quot_Y(E,n)$. Let \[\T_0=(\CC^*)^d/(\sim)=\{(t_1,\ldots,t_d):t_1,\ldots,t_d\neq 0\}/(\sim)\]
act on $Y$ naturally, where we quotient by the subgroup $\<\sim\>=\<t_1t_2t_3t_4\>$ when $d=4$. Let $
\T_1=(\CC^*)^N=\{(y_1,\dots,y_n):y_i\neq 0\}$ and $E=\oplus_{i=1}^N\O_Y\<y_i\>$ be the $\T_1$-equivariant bundle of rank $N$ with weights $y_1,\dots, y_N$. This induces a $\T_0\times \T_1$-action on $\quot_Y(E,n)$ by acting on the middle term of the sequence
\[0\->I\->E\->F\->0.\]
Let $\alpha\in K_{\T_0}(Y)$, then we can write
\[\alpha=[\oplus_{i=1}^r\O_Y\<v_i\>]-[\oplus_{i=r+1}^{r+s}\O_Y\<v_i\>]\]
where $v_1,\dots,v_{r+s}$ are its $\T_0$-weights. However, instead of thinking of $v_i$ as $\T_0$-weights, we would like to view them as generic parameters. Therefore we introduce an additional torus $\T_2=(\CC^*)^{r+s}$ acting on $\CC^r\times \CC^s$. Set $\T:=\T_0\times \T_1\times \T_2$ and denote
\equa{&K_\T(\pt)=\ZZ[t_1^{\pm1},\dots,t_d^{\pm1};y_1^{\pm1},\dots,y_N^{\pm1};v_1^{\pm1},\dots,v_{r+s}^{\pm1}]/(\sim),\\
&H^*_\T(\pt)=\CC[\lambda_1,\dots,\lambda_d;m_1,\dots,m_N;w_1,\dots,w_{r+s}]/(\sim).
}
where we quotient by the ideals $(\sim)=(t_1t_2t_3t_4-1)$ and $(\sim)=(\lambda_1+\lambda_2+\lambda_3+\lambda_4)$ respectively when $d=4$.

In the surface case, the equivariant virtual invariants is be defined using virtual equivariant localization of \cite{vir-loc,k-local-vir} in an analogous way as we did for the non-virtual case. For toric Calabi-Yau 4-folds, the equivariant Donaldson invariants were first introduced in Cao-Leung \cite[§8]{Cao:2014bca}; the K-theoretic equivariant invariants were predicted by Cao-Kool-Monavari \cite{CKM} and formalized by J. Oh and R. P. Thomas \cite[§7]{Oh:2020rnj} using the twisted virtual structure sheaf $\hat{\O}^\vir$ and their virtual equivariant localization. 
The \textit{equivariant virtual Segre, Chern} and \textit{Verlinde series} for $\a$ on $\quot_Y(E,n)$ are respectively
\equa{{\mathcal{S}}_Y(E,\a;q)&:=\sum_{n=0}^\infty q^n\int_{[\quot_Y(E,n)]^\vir}s(\alpha^{[n]}),\\
{\mathcal{C}}_Y(E,\a;q)&:=\sum_{n=0}^\infty q^n\int_{[\quot_Y(E,n)]^\vir}c(\alpha^{[n]}),\\
{\mathcal{V}}_Y(E,\a;q)&:=\sum_{n=0}^\infty q^n\chi^\vir(\quot_Y(E,n),\det(\a^{[n]})).
}
\begin{remark}
It is important to note that the above definitions, when $Y=X=\CC^4$, are dependent on a choice of orientation $o(\mathcal{L})$. This induces signs at each fixed point $Z\in\quot_X(E,n)^\T$, which are suppressed from the notation. We denote the sign at $Z$ to be $(-1)^{o(\mathcal{L})|_{Z}}$ and call $o(\mathcal{L})$ a \emph{choice of signs}.
\end{remark}

When $N=1$, the weight on $E=\O_Y\<y_1\>$ is not necessary as this extra action does not affect the fixed locus, so we sometimes ignore it by setting $y_1=1$. By definition, the coefficients of the Chern and Segre series are rational functions in the cohomological parameters $\lambda_1,\ldots,\lambda_d,m_1,\ldots,m_N,w_1,\ldots,w_{r+s}$. For the Verlinde series, they are in K-theoretic parameters $t_1,\ldots,t_d,y_1\ldots,y_N,v_1\ldots,v_{r+s}$. Using the identification of Remark \ref{rmk:K-H-identification}, they can be viewed as functions in the cohomological parameters as well. Define
\equa{\ast_{Y,i}(E,\a;q):=\ast_Y(E,\a;q)|_{\deg\vec\lambda,\vec m,\vec w=i}}
to be the part with total degree $i$ in those variables, for $\ast\in\{{\mathcal{S}},{\mathcal{C}},{\mathcal{V}}\}$. More precisely, by restricting a multi-variable function to a certain degree, we mean the following.
\begin{definition}\label{defn:degree}
Let $f(z_1,\dots,z_k)$ be a function in the ring of fractions of $\CC[\![z_1,\dots,z_k]\!]$. Consider the formal Laurent series expansion of $f(bz_1,\dots,bz_k)$ in the variable $b$:
\[f(bz_1,\dots,bz_k)=\sum_{i=-\infty}^\infty f_i(z_1,\dots z_k)b^i.\]
For $i\in\ZZ$, the part of $f$ with \emph{total degree $i$} is 
\[f(z_1,\dots,z_k)|_{\deg \vec{z}=i}:=f_i(z_1,\dots,z_k).\]
\end{definition}

\subsubsection{Reduced invariants on Quot schemes of surfaces}

When $S$ is a $K$-trivial surface, the obstruction on $\quot_S(E,n)$ contains a trivial summand making $e(T^\vir)$ vanish, as a result, the invariants all vanish. We may instead consider the reduced classes and invariants from Gromov-Witten theory and stable pair theory \cite{KT}, which has been employed to study the enumerative geometry of Hilbert schemes in for instance \cite{reduced}. In the case of $\quot_S(E,n)$, a reduced perfect obstruction theory can be obtained by removing one copy of $\O_{\quot_S(E,n)}$ from the usual obstruction for Quot schemes. The equivariant analogue of a $K$-trivial surface would be $S=\CC^2$ with the action of the 1-dimensional torus $\T_0=\{(t_1,t_2):t_1t_2=1\}$. Let $T^\redu$ be the virtual tangent bundle obtained from the reduced obstruction theory. For $E=\oplus_{i=1}^N\O_S\<y_i\>$ and $\a=[\oplus_{i=1}^r\O_S\<v_i\>]-[\oplus_{i=r+1}^s\O_S\<v_i\>]$, we define the \emph{reduced Segre, Chern, Verlinde series} to be respectively
\equa{{\mathcal{S}}^{\redu}(E,\alpha;q)&:=\sum_{n>0}^\infty q^n\int_{[\quot_S(E,n)]^\redu}s(\alpha^{[n]}),\\
{\mathcal{C}}^{\redu}(E,\alpha;q)&:=\sum_{n>0}^\infty q^n\int_{[\quot_S(E,n)]^\redu}c(\alpha^{[n]}),
\\
{\mathcal{V}}^{\redu}(E,\a;q)&:=\sum_{n>0}^\infty q^n\int_{[\quot_S(E,n)]^\redu}\td(T^\redu)\ch(\det(\alpha^{[n]})).
}

\subsection{Summary of results for surfaces}
\subsubsection{Computation of Chern series}

Consider the case $S=\CC^2$ with the $\T=(\CC^*)^2$-action. Using the tools from \cite{GM}, we are able to compute the virtual Chern series of line bundles for Hilbert schemes as follows in Section \ref{sec:chern-rank-2}. Note that in the non-equivariant setting, this can be retrieved from \cite[Corollary 15]{OP}.
\begin{corollary}
Let $S=\CC^2$, and $L=\O_S\<v_1\>$ a $\T$-equivariant line bundle over $S$. We have
\equa{{\mathcal{C}}_S(\O_S,L;q)=\left(\frac1{1-q}\right)^{\int_Sc(L)c_1(S)}.}
\end{corollary}
By extracting the part with the lowest total degree in $\lambda_1,\lambda_2,w_1$, we obtain the following 2-dimensional analogue to the Donaldson-Thomas partition function for $\CC^3$ \cite[Theorem 1]{MNOP2}, or Cao-Kool's formulation of Nekrasov's conjecture for $\CC^4$ \cite[Appendix B]{CK1}.
\begin{corollary}[Corollary \ref{cor:nek2d}]\label{cor:vanishing-intro}
For $S=\CC^2$, the following equality holds
\[\sum_{n=1}^\infty q^n\int_{[\hilb^n(S)]^\vir}1=e^{\frac{\lambda_1+\lambda_2}{\lambda_1\lambda_2}q}.\]
\end{corollary}

\subsubsection{Universal series expressions}

A common approach used to find closed formulas for the Segre and Verlinde series is by computing their universal series. In the non-virtual Hilbert scheme case for projective surfaces, it was proven in \cite{EGL} using methods of cobordism classes that the Segre and Verlinde invariants admit the following universal series expressions
\[I^{\mathcal{C}}(\alpha;q)=A_0(q)^{c_2(\alpha)}A_1(q)^{\chi(\det(\alpha))}A_2(q)^{\frac12\chi(\O_S)}A_3(q)^{c_1(\a)K_S-\frac12K_S^2}A_4(q)^{K_S^2},\]
\equanum{\label{eqn:univ-Verlinde}I^{\mathcal{V}}(\a,q)=B_1(q)^{\chi(\det(\a))}B_2(q)^{\frac12\chi(\O_S)}B_3(q)^{c_1(\a)K_S-\frac12K_S^2}B_4(q)^{K_S^2}}
where the products in the exponents are intersection products. The series $A_i(x),B_i(x)$ are universal in the sense that they only depend on $\alpha$ through the rank and are independent of the surface. Explicit formulas for these series were conjectured and computed in \cite{Lehn_1999,MOP,Marian_2021,EGL,GM}.
The Segre-Verlinde correspondence in this case concerns the relations between $A_i$ and $B_i$. It was first proposed by D. Johnson and Marian-Oprea-Pandharipande in relation to the study of Le Potier’s strange duality \cite{MOP,DJon} and recently proven by G\"ottsche-Mellit \cite{GM}.

For virtual invariants on Quot schemes of a smooth projective surface $S$ and torsion-free sheaf $E$, the universal series expressions are given by \cite[Theorem 1.2]{Bojko2}:
\equa{{\mathcal{S}}_S(E,\a;q)&=A_1^\vir(q)^{c_1(S)c_1(\a)}A_2^\vir(q)^{c_1(S)^2}A_3^\vir(q)^{c_1(S)c_1(E)},\\
{\mathcal{V}}_S(E,\a;q)&=B_1^\vir(q)^{c_1(S)c_1(\a)}B_2^\vir(q)^{c_1(S)^2}B_3^\vir(q)^{c_1(S)c_1(E)}.\\
}
The explicit formulas are computed in \cite[Theorem 17]{KHilb} for $A_1,A_2,B_1,B_2$ and in \cite[Theorem 1.2]{Bojko2} for $A_3,B_3$.

Unlike the compact case where the invariants are simply numbers, the equivariant invariants can contain terms of various degrees in $H_\T^*(\pt)_\loc$. This is reflected by the following theorem where the virtual equivariant Segre and Verlinde invariants are written as infinite products of series labeled by partitions. The notations for partitions are set in Section \ref{sec:partition}. For a partition $\mu$ and a K-theory class $\a$, denote \[c_\mu(\a):=\prod_{i=1}^{\ell(\mu)}c_{\mu_i}(\a).\]

\begin{theorem}[Theorem \ref{thm:SV-univ-sieres}]\label{thm:SV2d-virtual-intro}
Let $S=\CC^2$. For any $r\in\ZZ$, $N>0$, there exist universal power series $A_{\mu,\nu,\xi}(q),B_{\mu,\nu,\xi}(q)$, dependent on $N$ and $r$, such that for $E=\oplus_{i=1}^{N}\O_S\<y_i\>$ and $\a\in K_\T(S)$ of rank $r$, the equivariant virtual Segre and Verlinde series on $\quot_S(E,n)$ can be written as the following infinite products
\equa{{\mathcal{S}}_S(E,\alpha;q)=&\prod_{\mu,\nu,\xi\text{ partitions}}A_{\mu,\nu,\xi}(q)^{\int_{S}c_\mu(\a) c_\nu(S)  c_\xi(E)c_1(S)},\\
{\mathcal{V}}_S(E,\alpha;q)=&\prod_{\mu,\nu,\xi\text{ partitions}}B_{\mu,\nu,\xi}(q)^{\int_{S}c_\mu(\a) c_\nu(S)  c_\xi(E)c_1(S)},\\
{\mathcal{C}}_S(E,\alpha;q)=&\prod_{\mu,\nu,\xi\text{ partitions}}C_{\mu,\nu,\xi}(q)^{\int_{S}c_\mu(\a) c_\nu(S)  c_\xi(E)c_1(S)}.
}
\end{theorem}

The series in the above expressions are universal in the sense that they depend on the input $\a$ only by its rank $r$. Sometimes for clarity, we will add superscripts $N,r$ to indicate the ranks of $E$ and $\a$. Note that the series labeled by $\mu,\nu,\xi$ are exponentiated to homogeneous rational functions in 
\[H_\T^*(\pt)_\loc=\CC(\lambda_1,\lam_2;m_1,\ldots,m_N;w_1,\ldots,w_{r+s})\]
of degree $|\mu|+|\nu|+|\xi|-1$. The degree 0 terms occur when one of $\mu,\nu,\xi$ is the partition $(1)$ and the rest are the empty partition $(0)$. The argument of Section \ref{sec:proj-reduction} shows that the series with degree 0 exponents are necessarily equal to the series from in the projective case, that is
\equanum{\label{eqn:proj-equi}&A_{(1),(0),(0)}(q)=A_{1}^\vir(q),\indent A_{(0),(1),(0)}(q)=A_{2}^\vir(q),\indent A_{(0),(0),(1)}(q)=A_{3}^\vir(q),\\
&B_{(1),(0),(0)}(q)=B_{1}^\vir(q),\indent B_{(0),(1),(0)}(q)=B_{2}^\vir(q),\indent B_{(0),(0),(1)}(q)=B_{3}^\vir(q).}

The universal series expressions of the reduced invariants take a much simpler form; as opposed to having series exponentiated to some powers of cohomology classes, we have the following additive expressions.
\begin{theorem}[Theorem \ref{cor:reduced-series}]\label{thm:reduced-intro}
When $S=\CC^2$, the equivariant reduced Segre and Verlinde series for $E=\oplus_{i=1}^N\O_S\<y_i\>$ and $\a\in K_\T(S)$ are
\equa{
{\mathcal{S}}^{\redu}(E,\a;q)=&\sum_{\mu,\nu,\xi}\log\left(A_{\mu,\nu,\xi}(q)\right)\cdot \int_S c_\mu(\a)c_\nu(S)c_\xi(E),\\
{\mathcal{V}}^{\redu}(E,\a;q)=&\sum_{\mu,\nu,\xi}\log\left(B_{\mu,\nu,\xi}(q)\right)\cdot \int_S c_\mu(\a)c_\nu(S)c_\xi(E),\\
{\mathcal{C}}^{\redu}(E,\a;q)=&\sum_{\mu,\nu,\xi}\log\left(C_{\mu,\nu,\xi}(q)\right)\cdot \int_S c_\mu(\a)c_\nu(S)c_\xi(E)\\
}
where $A_{\mu,\nu,\xi}, B_{\mu,\nu,\xi}$ and $C_{\mu,\nu,\xi}$ are the same series from Theorem \ref{thm:SV2d-virtual-intro}.
\end{theorem}

\subsubsection{Virtual Segre-Verlinde correspondence}
When $Y$ is compact, the virtual Segre-Verlinde correspondence has been proven for compact surfaces and Calabi-Yau 4-folds \cite[Theorem 1.6]{Bojko2} for torsion-free sheaves $E$ to be
\[{\mathcal{S}}_Y(E,\alpha;q)={\mathcal{V}}_Y(E,\alpha;(-1)^Nq).\]

As a corollary of this result, Theorem \ref{thm:SV2d-virtual-intro} and the relations (\ref{eqn:proj-equi}), we first prove the following ``weak" equivariant Segre-Verlinde correspondence.

\begin{corollary}[Corollary \ref{cor:SV2d-virtual}]\label{cor:weak-SV2d-virtual-intro}
In the setting of Theorem \ref{thm:SV2d-virtual-intro}, we have the following correspondence
\equa{A_{\mu,\nu,\xi}(q)&=B_{\mu,\nu,\xi}((-1)^Nq).
}
whenever one of $\mu,\nu,\xi$ is $(1)$ and the other two are $(0)$. In particular, the degree 0 part satisfies
\[{\mathcal{S}}_{S,0}(E,\alpha;q)-{\mathcal{V}}_{S,0}(E,\alpha;(-1)^Nq)=\sum_{n=2}^\infty \frac{f_n}{(\lambda_1\lambda_2)^{n-2}}\cdot\left(\int_S c_1(S)\right)^2\cdot q^n\]
for some terms $f_n\in H_\T^{2n-2}(\pt)$ dependent on $\a$ through its rank and Chern classes.
\end{corollary}
This is weak in the sense that only the series whose powers are degree 0 satisfy the usual correspondence, while Computations for small values of $n$ show that the terms $f_n$ can be non-zero so the naive correspondence does not hold in all degrees. One might instead ask whether there are any other relations between the series in all cohomological degrees generalizing the above corollary.

We are able to give a complete answer relating $B_{\mu,\nu,\xi}$ and $C_{\mu,\nu,\xi}$ once the sizes of the partitions $|\mu|,|\xi|$ are fixed and $\nu = (0)$. This is represented in Figure \ref{fig:sv}, but before stating our result, some further notation is needed. Given any partition $\mu$, integer $a\in\ZZ$ and $n> 0$, the binomial coefficients for $\mu$ is
\[\binom{a}{\mu}:=\prod_{i=1}^{\ell(\mu)}\binom{a}{\mu_i},\] 
and the downward factorial of $a$ by $n$ is
\[(a)_{(n)}:=a\cdot(a-1)\cdots(a-n+1),\] 
\[(a)_{(-1)}:=\frac1{a+1}.\] 

\begin{theorem}[Theorem \ref{thm:SV2d-deg-pos}]\label{thm:SV2d-deg-pos-intro}
For rank $r\geq 0$ and integers $k_1,k_2\geq 0$ with $k:=k_1+k_2$, the universal series of Theorem \ref{thm:SV2d-virtual-intro} satisfy
\equa{k_1!k_2!\sum_{|\mu|=k_1}\sum_{|\xi|=k_2}\binom{r}{\mu}\binom{N}{\xi}\log A_{\mu,(0),\xi}(q)&=-r\sum_{n=1}^\infty\frac{(-n(r+N))_{(k-1)}}{n}\binom{-nr-1}{nN-1}q^n,\\
k_1!k_2!r^{k_1}\sum_{|\xi|=k_2}\binom{N}{\xi}\log B_{(1)_{k_1},(0),\xi}(q)&=-r^{k}\sum_{n=1}^\infty n^{k-2}\binom{-nr-1}{nN-1}\left((-1)^Nq\right)^n.}
Furthermore, we have
\equa{
&k_1!k_2!\sum_{|\mu|=k_1}\sum_{|\xi|=k_2}\binom{r}{\mu}\binom{N}{\xi}\log C_{\mu,(0),\xi}(q)=r\sum_{n=1}^\infty\frac{(n(r-N))_{(k-1)}}{n}\binom{nr-1}{nN-1}q^n,
}
which can be compared to the identities above by replacing $r$ with $-r$.
\end{theorem}
In particular, setting $k=1$, the first two equalities of this theorem give 
\[A_{(1),(0),(0)}(q)=\left(B_{(1),(0),(0)}((-1)^Nq)\right),\]
\[A_{(0),(0),(1)}(q)=\left(B_{(0),(0),(1)}((-1)^Nq)\right).\]
This is consistent with the Segre-Verlinde correspondence in degree 0 from Corollary \ref{cor:weak-SV2d-virtual-intro}. Going one degree higher by setting $k=2$, we get the following correspondence.
\begin{corollary}[Corollary \ref{cor:sv2d-deg1}]\label{cor:sv2d-deg1-intro}
For rank $r\geq 0$, the universal series of Theorem \ref{thm:SV2d-virtual-intro} satisfy the following correspondences
\equa{&\indent A_{(1,1),(0),(0)}(q)^{-r}A_{(2),(0),(0)}(q)^{\frac{-(r-1)}{2}}=B_{(1,1),(0),(0)}\left((-1)^Nq\right)^{r+N},\\
&\indent A_{(1),(0),(1)}(q)^{-r}=B_{(1),(0),(1)}\left((-1)^Nq\right)^{r+N},\text{ and }\\
&\indent A_{(0),(0),(1,1)}(q)^{-rN}A_{(0),(0),(2)}(q)^{\frac{-r(N-1)}{2}}\\
&=B_{(0),(0),(1,1)}\left((-1)^Nq\right)^{N(r+N)}B_{(0),(0),(2)}\left((-1)^Nq\right)^{\frac{(N-1)(r+N)}{2}}.
}
\end{corollary}

It is natural to expect that applying Lagrange inversion to the expressions 
in Theorem \ref{thm:SV2d-deg-pos-intro} will simplify the formulae. In Section \ref{sec:closed-form}, we use this to phrase the higher degree Segre and Verlinde series in terms of applying differential operators of comparable form to a single function $\f(t)=\log(1+t)$ with common variable change except for the usual sign $q\leadsto (-1)^Nq$.
\begin{theorem}[Theorem \ref{cor:SV2d-closed}]\label{cor:SV2d-closed-intro}
Let $\f(t)=\log(1+t)$ and $\psi(t)=Nt^{-1}+r(1+t)^{-1}$. Define the differential operator
\[ D_\psi=\frac{1}{\psi}\cdot \frac{d}{dt}.\]
Furthermore, use the notation
\[D_{\mathcal{S}}^{(k)}=(-(r+N)D_\psi)_{(k-1)},\indent D_{\mathcal{V}}^{(k)}=r^{k-1}D_\psi^{k-1}\]
for $k\geq 0$ where $D_\psi^{-1}(-)$ denotes integrating $\psi\cdot (-)$ assuming a constant term 0. In the setting of Theorem \ref{thm:SV2d-deg-pos-intro}, we have the following relations
\equa{
k_1!k_2!\sum_{|\mu|=k_1}\sum_{|\xi|=k_2}\binom{r}{\mu}\binom{N}{\xi}\log A_{\mu,(0),\xi}(q)&=-r\sum_{i=1}^N\left(D_{\mathcal{S}}^{(k)}\f\right)(H_i),\\
k_1!k_2!r^{k_1}\sum_{|\xi|=k_2}\binom{N}{\xi}\log B_{(1)_{k_1},(0),\xi}\left((-1)^Nq\right)&=-r\sum_{i=1}^N\left(D_{\mathcal{V}}^{(k)}\f\right)(H_i)}
where $H_i$ are the Newton-Puiseux solutions to $H_i^N=q(1+H_i)^{-r}$.
\end{theorem}

In this theorem, setting $k=1$ allows us to retrieve the degree 0 parts for both the Segre and Verlinde invariants using the function $\f$, hence the Segre-Verlinde correspondence of Corollary \ref{cor:weak-SV2d-virtual-intro}. When $k=2$, we get a relation in the degree 1 part which is equivalent to Corollary \ref{cor:sv2d-deg1-intro}. In general, the Segre and Verlinde relations are obtained by applying $D_{\mathcal{S}}^{(k)}$ and $D_{\mathcal{V}}^{(k)}$ to $\f$. Figure \ref{fig:sv} summarizes these observations.

\begin{remark}
The above results only dealt with the $\nu=(0)$ case. When $\nu=(1)$, similar techniques of Section \ref{sec:reduced-computation} can be applied to the term \cite[(4.24)]{Bojko2}, if an explicit expression for it is found. For $|\nu|>1$, the same techniques no longer work due to limitations of Lemma \ref{lem:differential}.
\end{remark}

 \begin{figure}
     \centering
\begin{tikzcd}
	& {\text{Segre}} && {\text{Verlinde}} \\
	{k=0}\arrow[dd,dash,
    start anchor={[xshift=8ex, yshift=4ex]},
    end anchor={[xshift=8ex,yshift=-32ex]}
    ] & {D_{\mathcal{S}}^{(0)}\varphi} && {D_{\mathcal{V}}^{(0)}\varphi} &\\
	{k=1} && \varphi && {\leadsto\text{Corollary \ref{cor:weak-SV2d-virtual-intro}}} \\
	{k=2} & {D_{\mathcal{S}}^{(2)}\varphi} && {D_{\mathcal{V}}^{(2)}\varphi} & {\leadsto\text{Corollary \ref{cor:sv2d-deg1-intro}}} \\
	{k=3} & {D_{\mathcal{S}}^{(3)}\varphi} && {D_{\mathcal{V}}^{(3)}\varphi}\tikzmark{bracebegin}  \\
	\vdots & \vdots && \vdots & {\leadsto\text{Theorem \ref{cor:SV2d-closed-intro}}} \\
	k & {D_{\mathcal{S}}^{(k)}\varphi} && {D_{\mathcal{V}}^{(k)}\varphi\tikzmark{braceend} }
	\arrow[from=3-3, to=2-2]
	\arrow[from=3-3, to=2-4]
	\arrow[from=3-3, to=4-2]
	\arrow[from=3-3, to=4-4]
	\arrow[from=3-3, to=5-2]
	\arrow[from=3-3, to=5-4]
	\arrow[from=3-3, to=7-2]
	\arrow[from=3-3, to=7-4]
\end{tikzcd}
\begin{tikzpicture}[overlay,remember picture]
        \draw[decorate,decoration={brace}]  ( $ (pic cs:bracebegin) +(6ex, 9pt)  $ ) -- ( $ (pic cs:braceend) -(-6ex, 4pt) $ );
    \end{tikzpicture}
     \caption{Relating Segre and Verlinde invariants by the common function $\f$.}
     \label{fig:sv}
 \end{figure}
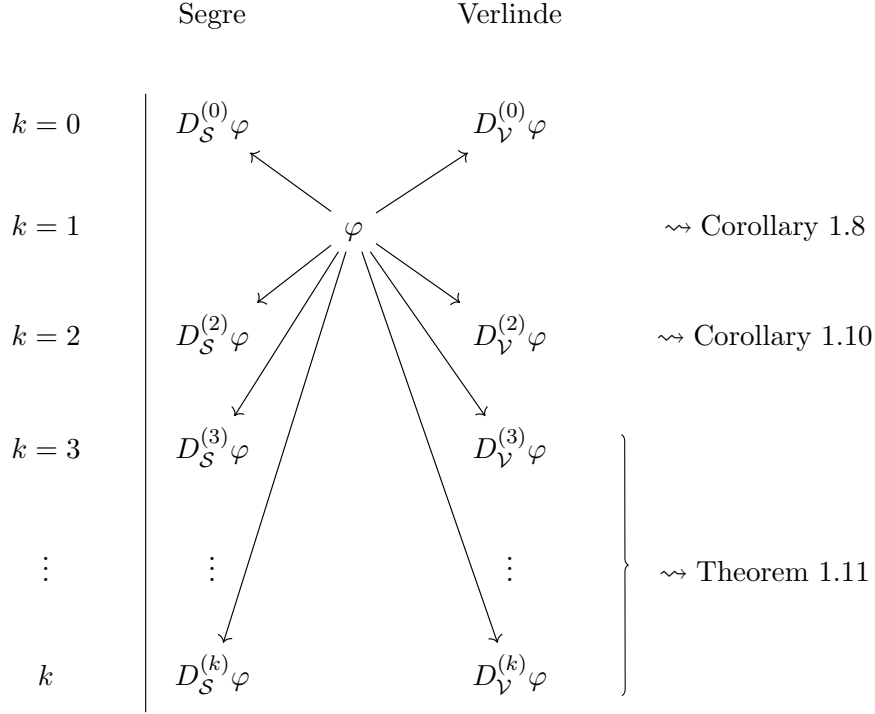

Also proven in \cite[Theorem 1.7]{Bojko2} is a symmetry for virtual Segre series, which states
\[{\mathcal{S}}_Y(E,V;(-1)^Nq)={\mathcal{S}}_Y(V,E;(-1)^rq)\]
for torsion free sheaves $E$ and $V$ of rank $N$ and $r$ respectively. Similar to Corollary \ref{cor:weak-SV2d-virtual-intro}, we have the following weak version of this symmetry.

\begin{corollary}\label{cor:sv-sym2d-intro}
In the setting of Theorem \ref{thm:SV2d-virtual-intro}, for $\a=V=\oplus_{i=1}^{r}\O_S\<v_i\>$, we have the following symmetry
\equa{A^{N,r}_{\mu,\nu,\xi}((-1)^Nq)&=A^{r,N}_{\xi,\nu,\mu}((-1)^rq)
}
whenever one of $\mu,\nu,\xi$ is $(1)$ and the other two are $(0)$. In degree 0, we have
\[{\mathcal{S}}_{S,0}(E,V;(-1)^Nq)-{\mathcal{S}}_{S,0}(V,E;(-1)^rq)=\sum_{n=1}^\infty \frac{g_n}{(\lambda_1\lambda_2)^{n-2}}\cdot\left(\int_S c_1(S)\right)^2\cdot q^n\]
for some terms $g_n\in H_\T^{2n-2}(\pt)$.
\end{corollary}

Motivated by our study of the higher degree Segre series in Section \ref{sec:sym} which as we see after using Theorem \ref{thm:SV2d-deg-pos-intro} are symmetric under interchanging $\mu$ and $\xi$\textemdash the partitions keeping track of Chern classes of $\alpha$ and $E$ respectively, we conjecture the following \textit{strong Segre symmetry}.
\begin{conjecture}\label{con:strong-sym}
Let $r,N>0$, $S=\CC^2$. For $E=\oplus_{i=1}^N\O_S\<y_i\>,V=\oplus_{i=1}^r\O_S\<v_i\>$, we have the following symmetry
\equa{{\mathcal{S}}_{S}(E,V;(-1)^Nq)={\mathcal{S}}_{S}(V,E;(-1)^rq).}
\end{conjecture}
We checked that this condition holds using a program in the cases summarized in \eqref{eq:2dsegsym}.

\subsubsection{Reduced Segre-Verlinde correspondence}
For the reduced invariants on $S=\CC^2$, let us denote ${\mathcal{S}}_i^{\redu}$ and ${\mathcal{V}}_i^{\redu}$ the degree $i$ parts of the reduced Segre and Verlinde series respectively, in the sense of Definition \ref{defn:degree}. In this setting, the fact that the series from Theorem \ref{thm:reduced-intro} are the same ones from Theorem \ref{thm:SV2d-virtual-intro} allows us to obtain relations from the non-reduced case. For example, the Segre-Verlinde correspondence of Corollary \ref{cor:weak-SV2d-virtual-intro} and symmetry of Corollary \ref{cor:sv-sym2d-intro} give the following reduced Segre-Verlinde correspondence and symmetry in degree $-1$. 
\begin{corollary}
We have the following correspondence in degree $-1$:
\equa{{\mathcal{S}}_{-1}^{\redu}(E,V;q)
&={\mathcal{V}}_{-1}^{\redu}(E,\a;(-1)^Nq).
}
When $\a=V$ is an equivariant vector bundle, we have the following symmetry:
\equa{{\mathcal{S}}_{-1}^{\redu}(E,V;(-1)^Nq)
&={\mathcal{S}}_{-1}^{\redu}(V,E;(-1)^rq).
}
\end{corollary}

Let $\alpha\in K_\T(S)$ with rank $r>0$. Write $c_1\lambda=c_1(\a)$ and $c_2\lambda^2=c_2(\a)$ for some $c_1,c_2\in\QQ$, then we can use the Theorem \ref{thm:reduced-intro} to show that for some series $A_2(q),A_1(q),A_0(q),B_1(q),B_0(q)$, dependent on $r$ and $N$,
\equa{{\mathcal{S}}_0^{\redu}(E,\alpha;q)|_{m_1=\dots=m_N=0}&=-\log \left(A_{(2),(0)}(q)\right)\cdot c_2-\log \left(A_{(1,1),(0)}(q)\right)\cdot c_1^2+\log \left(A_{(0),(2)}(q)\right)\\
&=:A_2(q)c_2+A_1(q)c_1^2+A_0(q),\\
{\mathcal{V}}_0^{\redu}(E,\alpha;q)|_{m_1=\dots=m_N=0}&=-\log \left(B_{(1,1),(0)}(q)\right)\cdot c_1^2+\log \left(B_{(0),(2)}(q)\right)\\
&=:B_1(q)c_1^2+B_0(q).}
These degree 0 terms should correspond to the ones for compact K3-surfaces and their reduced invariants. In particular, one can use the results obtained in Theorem \ref{thm:SV2d-deg-pos-intro}, Theorem \ref{cor:SV2d-closed-intro}, and the ones below to predict the Segre and the Verlinde series in the compact case and describe their correspondences.

 Corollary \ref{cor:sv2d-deg1-intro} provides relations between $A_1,A_2$ and $B_1$, and we have the following formula for $B_1(q)$ obtained from Theorem \ref{thm:SV2d-deg-pos-intro} by setting $k_1=2,k_2=0$.
\begin{corollary}

The series $B_1(q)$ is explicitly given by
\[[q^n]B_1(q)=-\frac{1}{2}\binom{n(N+r)-1}{nr}\]
for $n,r>0$.
\end{corollary}

We give some conjectural formulas for $A_0,B_0$ in terms of $A_1,B_1$ when $N=1$, respectively, checked for $n\leq 20,r<5$.
\begin{conjecture}
When $N=1$ and $r\in\ZZ$, we have
\[[q^n]B_0(q)=\frac16\left(\binom{r+1}{2}(n-1)-1\right)[q^n]B_1(q).\]
When $N=1$, $r<-1$ and $n>1$, we have
\[[q^n]A_0(q)=\frac{1}{12} r (n r + n + 2)[q^n]A_1(q).\]
\end{conjecture}

\subsection{Correspondence for 4-folds and other observations}
\subsubsection{Segre-Verlinde correspondence}
Since all toric Calabi-Yau 4-folds are non-compact, we do not know how the invariants in the non-compact case relate to the ones in the compact case. We have seen for the surface case the powers on the universal series have a factor of $c_1(S)$. Considering the series given in \cite[Proposition 4.13]{boj} and \cite[(3.38)]{Bojko2}, together with Section \ref{sec:factor}, we see that this term should be replaced by $c_3(X)$ in the 4-fold case. Motivated by Corollary \ref{cor:weak-SV2d-virtual-intro} and \ref{cor:sv-sym2d-intro}, we conjecture a weak Segre-Verlinde correspondence and symmetry for $X=\CC^4$. 

\begin{conjecture}\label{con:sv4d-intro}
Let $X=\CC^4$, $E=\oplus_{i=1}^{N}\O_X\<y_i\>$, $V=\oplus_{i=1}^{r}\O_X\<v_i\>$, and $\a\in K_\T(X)$, then for some choice of signs $o(\mathcal{L})$, we have the following symmetry and correspondence
\equa{{\mathcal{S}}_X(E,V;(-1)^Nq)&={\mathcal{S}}_X(V,E;(-1)^rq),\\
{\mathcal{S}}_{X,0}(E,\alpha;q)-{\mathcal{V}}_{X,0}(E,\alpha;(-1)^Nq)&=\sum_{n=1}^\infty \frac{F_n}{(\lambda_1\lambda_2\lam_3\lam_4)^{n-2}}\cdot\left(\int_X c_3(X)\right)^2\cdot q^n}
for some terms $F_n\in H_\T^{4n-6}(\pt)$ dependent on $\a$ through its rank and Chern classes.
\end{conjecture}
We checked both the correspondence and symmetry using a computer program as summarized in \eqref{eq:4DSVcomp} and \eqref{eq:4DSymp}.

\subsubsection{Nekrasov's conjectures}
In \cite[Appendix B]{CK1}, Y. Cao and M. Kool gave a mathematical formulation of Nekrasov's conjecture \cite[§5]{Mag}. 
We generalize this to Quot schemes of $\CC^2$ and $\CC^4$ as follows.
\begin{conjecture}\label{con:cao-kool-quot}
Let $Y=\CC^d$, $E=\oplus_{i=1}^{r+1}\O_Y\<y_i\>$, and $V=\oplus_{i=1}^r\O_Y\<v_i\>$. When $d=2$, or $d=4$ with some choice of signs $o(\mathcal{L})$, we have
\[{\mathcal{C}}_Y(E,V;q)=\exp\left(q\int_Yc_{d-1}(Y)\right).
\]
\end{conjecture}
Note that when $N=1$ and $Y=\CC^2$, this is exactly Corollary \ref{cor:vanishing-intro}. When $Y=\CC^4$, we shall show that this conjecture is a consequence of Nekrasov-Piazzalunga's conjecture \cite[§2.5]{magcolor} using a Quot scheme version of Cao-Kool-Monavari's cohomological limit \cite[Appendix A]{CKM} in Proposition \ref{prop:nek-con-limit}. For $Y=\CC^2$, we check it for $N=2,3,4,5$ up to and including $n=7,4,2,2$ respectively.

Since $[\quot_Y(E,n)]^\vir$ has virtual dimension $nN$, we have $C^N_Y(V;q)=1$ when $N>r$ in the compact case simply for degree reason, which means the corresponding Verlinde series is also trivial due to the Segre-Verlinde correspondence. In the non-compact case, the above conjectures suggest that they may contain negative degree terms. However, with a computer calculation for $Y=\CC^2$, $N=2,3,4,5$ and all possible $r$, up to and including $n=8,5,2,2$, we see a complete vanishing when $N>r+1$ for Chern numbers, and when $r<N$ for rank $-r$ Verlinde numbers.
\begin{conjecture}\label{con:low-rank-vanish}
Let $Y=\CC^d$, $N>1$, and $E=\oplus_{i=1}^{N}\O_Y\<y_i\>$. When $d=2$, or $d=4$ with some choice of signs, we have for $r=0,1,\dots,N-2$ and $V=\oplus_{i=1}^r\O_Y\<v_i\>$,
\equa{{\mathcal{C}}_Y(E,V;q)=1.}
Furthermore, for $r=1,\dots,N-1$, we have
\[{\mathcal{V}}_Y(E,-[V];q)=1.\]
\end{conjecture}
In Proposition \ref{prop:nek-con-limit}, we also show that the Chern series part of this conjecture for $d=4$ is a consequence of Nekrasov-Piazzalunga's Conjecture \ref{con:nek2}.

	\section*{Acknowledgements}
 We would like to thank Rahul Pandharipande for asking about equivariant Segre-Verlinde correspondences during the first author's seminar talk and Henry Liu whose box-counting code sped up the progress of this project. We further wish to express gratitude towards Martijn Kool and Sergej Monavari for helpful discussions on the topic.

 A.B. was supported by ERC-2017-AdG-786580-MACI. This project has received funding from the European Research Council (ERC) under the European Union Horizon 2020 research and innovation program (grant agreement No 786580).
\section{Preliminaries}

\subsection{Equivariant cohomology and K-theory}\label{sec:equivar-coho}
Given a topological group $G$ acting on a topological space $M$, the \emph{equivariant cohomology} $H_G^*(M)$ is defined to be $H^*(EG\times M/G)$, where $EG\->BG$ is the universal principle $G$-bundle on the classifying space $BG$. The map $M\->\text{pt}$ induces a ring homomorphism $H^*_G(\text{pt})\->H^*_G(M)$, making $H^*_G(M)$ a module over $H^*_G(\text{pt})$ for any $M$, and we can view $H^*_G(\text{pt})$ as a ``coefficient ring". 

\begin{definition} Given a $G$ representation $V$, viewed as a vector bundle $V\->\{\text{pt}\}$, we define its \emph{equivariant characteristic classes} by taking the associated bundle 
\[EG\times_GV\-> EG\times_G\{\pt\}=BG\]
and taking its characteristic classes in $H^*(BG)=H^*_G(\pt)$. Denote $c^G_i,e_G,\ch_G,\td_G$ the equivariant versions of the $i$-th Chern class, the Euler class, the Chern character, and the Todd class respectively.
\end{definition}

\begin{example}For the action of a $d$-dimensional torus $\T=(\CC^*)^d=\{(t_1,\dots ,t_d):t_i\neq 0\}$, the coefficient ring is
\[H^*_{\T}(\pt)=H^*(B\T)=H^*((\CC P^\infty)^d)=\CC[\lambda_1,\dots ,\lambda_d]\]
and $\lambda_1,\dots ,\lambda_d$ are exactly the equivariant first Chern classes of 1-dimensional $\T$-representations with weight $t_1,\dots ,t_d$ respectively. In general \cite[§3.2]{edidin},
\[c_1^{\T}(\CC\< t_1^{w_1}\dots t_d^{w_d}\>)=w_1\lambda_1+\dots +w_d\lambda_d.\]
For the $d-1$-dimensional subtorus $\T'=\{(t_1,\dots ,t_d)\in \CC^d: t_1\dots t_d=1\}\seq \T$, the inclusion induces the following isomorphism to the quotient ring:
\[H_{\T'}^*(\text{pt})\cong\CC[\lambda_1,\dots ,\lambda_d]/(\lambda_1+\dots +\lambda_d).\]
\end{example}
We can construct the \emph{equivariant K-group} $K_G(M)$ from the $G$-equivariant vector bundles. When $M=\CC^d$,  vector bundles over $M$ are trivial, but they may carry non-trivial $G$-actions. Therefore the equivariant bundles on $\CC^d$ correspond to finite-dimensional $G$-representations. The equivariant characteristic classes of vector bundles can then be extended to the K-theory classes. For example, the Euler class of $\alpha=[V]-[W]\in K_G(\pt)$ is $e^G(\a)=e^G(V)/e^G(W)$, which lives in the ring of fractions $H_G^*(\pt)_\loc$; the Chern character is $\ch_G(\a)=\ch_G(V)-\ch_G(W)$, which lives in $\prod_{i=0}^\infty H^i_G(\pt)$.

\begin{example} When $Y=\CC^d$ with the natural action by $\T=(\CC^*)^d$ or $\T'=(\CC^*)^{d-1}$, we have the following character rings:
\[K_{\T}(Y)\cong \ZZ[t_1^{\pm1},\dots ,t_d^{\pm1}]\cong K_{\T}(\pt),\]
\[K_{\T'}(Y)\cong \frac{\ZZ[t_1^{\pm1},\dots ,t_d^{\pm1}]}{(t_1\cdots t_d-1)}\cong K_{\T'}(\pt) \]
where for any weight $w=(w_1,\dots ,w_d)$, the line bundle $\O_Y\<t^w\>:=\mathcal{O}_Y\otimes t^w$ simply corresponds to its character $t^w=t_1^{w_1}\cdots t_d^{w_d}$.
\end{example}

\begin{remark}\label{rmk:K-H-identification}
We will occasionally identify Chern characters, which are power series in cohomology, with elements in K-theory by
\[t_1^{w_1}\cdots t_d^{w_d}\leftrightarrow \ch_\T(\O_Y\<t_1^{w_1}\cdots t_d^{w_d}\>)=e^{w_1\lambda_1+\dots +w_d\lambda_d}.\]
This allows us to consider certain classes in cohomology as elements of $K_\T(\pt)$. For example, for the line bundle $L=\O_Y\<t_1^{w_1}\cdots t_d^{w_d}\>$, we write
\[\ch_\T(\Lambda_{-1}L^\vee)=1-e^{-c_1^\T(L)}=1-t_1^{-w_1}\cdots t_d^{-w_d}\in K_\T(Y).\]
\end{remark}

The reason we consider equivariant cohomology is for equivariant integration. The integration formula of \cite[Corollary 1]{Bott-residue} via equivariant localization states that on a smooth complete variety $Y$ with the action of a torus $\T$, for $\lambda$ an equivariant cohomological class, we have
\[\pi_{Y\ast}(\lambda)=\sum_{F}\pi_{F\ast}\left(\frac{i_F^*\lambda}{e_\T(N_FY)}\right)\]
where the sum goes through the components $F$ of the fixed locus, $N_FY$ denotes the normal bundle, $\pi$ denotes projection to a point, and $i$ denotes the inclusion map. The right-hand side of this formula can be used to define equivariant integration in general; for $Y$ an arbitrary smooth variety with finitely many fixed (reduced) points, the \emph{equivariant push-forward} of $\pi_Y$ is
\equanum{\label{eqn:equivar-int}\int_Y:H^*_\T(Y)&\->H^*_\T(\pt)_{\loc},\\
\alpha\indent&\mapsto\sum_{x\in Y^\T}\frac{i_x^*\alpha}{e_\T(T_xY)}.}

\begin{example}\label{ex:equi-push} Again let $Y=\CC^d$ with the natural $\T=(\CC^*)^d$-action. The only $\T$-fixed point of $Y$ is the origin. At the origin, the character for the tangent space is $T_{0}Y=t_1+t_2+\cdots+t_d\in K_T(\pt)$, so $e_\T(T_{0}Y)=\lambda_1\cdots\lambda_d$. Substituting into (\ref{eqn:equivar-int}), we have
\[\int_Y\alpha=\frac{\alpha}{\lambda_1\cdots\lambda_d}\]
\end{example}

\subsection{Partitions and solid partitions}\label{sec:partition}
A \emph{partition} $\mu$ is a finite sequence $(\mu_1,\mu_2,\dots ,\mu_\ell)$ of non-increasing positive integers. The size $|\mu|$ is the sum of $\mu_i$'s and we call $\ell=\ell(\mu)$ its length. The empty sequence $(0)$ is the \emph{empty partition} with size $|(0)|=0$. Each partition $\mu$ corresponds to a young diagram which consists of pairs of non-negative integers $(i,j)\in\ZZ_{\geq 0}^2$ as follows
\[\mu\xleftrightarrow{} \{(i,j):j< \mu_{i+1}\}.\]
A pair $\square=(i,j)$ in the above set is called a box in $\mu$, which we denote $\square\in\mu$. The \emph{conjugate partition} $\mu^t$ is defined to be the partition whose boxes are $\{(j,i):(i,j)\in\mu\}$. Denote $c(\square),r(\square),a(\square),l(\square)$ the \emph{column index}, \emph{row index}, \emph{arm length} and \emph{leg length} of $\square=(i,j)\in\mu$, defined explicitly as follows
\equa{&c(\square)=j,\indent r(\square)=i,\\ &a(\square)=\mu_{i+1}-j-1,\indent l(\square)=\mu_{j+1}^t-i-1.
}

When $i,j>0$, a necessary condition for box $(i,j)$ to be in $\mu$ is that both $(i-1,j)$ and $(i,j-1)$ are in $\mu$. When $i=0$ (resp. $j=0$), we only need $(i,j-1)\in\mu$ (resp. $(i-1,j)\in\mu$).

A \emph{solid partition} $\pi$ is a finite sequence $(\pi_{ijk})_{i,j,k\geq 1}$ of positive integers such that
\equa{\pi_{ijk}\geq \pi_{i+1,j,k},\indent \pi_{ijk}\geq \pi_{i,j+1,k},\indent\pi_{ijk}\geq \pi_{i,j,k+1}.}
The size of $|\pi|$ is the sum of the $\pi_{ijk}$'s. As a 4-dimensional analogue to partitions, the solid partition can also be viewed as a collection of boxes
\[\pi\xleftrightarrow{}\{(i,j,k,l):l< \pi_{i,j,k}\}\seq \ZZ^4_{\geq 0}.\]
Similar to partitions, we have
\equanum{\label{solid-partition}(i,j,k,l)\in\pi\text{ implies }\begin{cases}
(i-1,j,k,l)\in\pi \text{ unless $i=0$,}\\
(i,j-1,k,l)\in\pi \text{ unless $j=0$,}\\
(i,j,k-1,l)\in\pi \text{ unless $k=0$,}\\
(i,j,k,l-1)\in\pi \text{ unless $l=0$.}\\
\end{cases}
}

For a positive integer $N$, an \emph{$N$-colored partition} of size $n$ is an $N$-tuple of partitions $\mu=(\mu^{(1)},\dots ,\mu^{(N)})$ such that $|\mu|:=\sum|\mu^{(i)}|=n$. Figure \ref{fig1} illustrates how the partitions are colored based on their index. Similarly, an \emph{$N$-colored solid partition} is an $N$-tuple of solid partitions.
\begin{figure}
\centering
\includegraphics[width=0.4\textwidth]{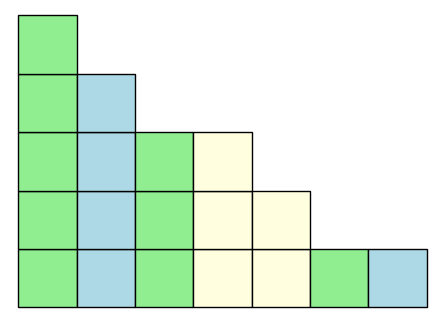}
\caption{A $3$-colored partition $\mu=(\mu^{(1)},\mu^{(2)},\mu^{(3)})$ of size $|\mu|=19$ where $\mu^{(1)}=(5,3,1)$, $\mu^{(2)}=(4,1)$, $\mu^{(3)}=(3,2)$ are colored by green, blue and yellow respectively}
\label{fig1}
\end{figure}

\subsection{Admissible functions and universal series}
We consider the notion of admissibility in the sense of \cite{mel-HLV}, which will be an important condition in finding universal series for equivariant invariants.

\begin{definition}\label{def:admissible}
Let $F(Q_1,Q_2\dots ;q_1,\dots ,q_d)\in \QQ(q_1,\dots ,q_d)[\![Q_1,Q_2,\dots ]\!]$ be a series in finitely many variables $Q_1,Q_2,\dots $ with constant term equal to 1. Then using the plethystic exponential $\Exp$, we can write
\[F=\Exp\left(\frac{L}{(1-q_1)\cdots(1-q_d)}\right)\]
such that $L$ is a power series in the variables $Q_1,Q_2,\dots$ whose coefficients are rational functions in $q_1,\dots ,q_d$. The series $F$ is called \emph{admissible with respect to the variables} $q_1,\dots, q_d$ if the coefficients of $L$ are polynomials in $q_1,\dots ,q_d$.
\end{definition}

Suppose $F(Q;m_1,\dots ,m_N;w_1,\dots ,w_r;q_1,\dots ,q_d)\in \QQ(q_1,\dots ,q_d)[\![Q;m_1,\dots ,m_N;w_1,\dots ,w_r]\!]$ is admissible with respect to $q_1,\dots,q_d$ with constant term 1, we have the following Laurent expansion
\equa{\log F(Q;\vec{m};\vec{w};;e^{\lambda_1},\dots ,e^{\lambda_d})=\sum_{k_1,\dots ,k_d=-\infty}^\infty H_{k_1,\dots ,k_d}(Q;\vec{m};\vec{w})\lambda_1^{k_1}\dots \lambda_d^{k_d}.}
Since $F$ is admissible, by the definition of plethystic exponential, 
\[(1-q_1)\cdots(1-q_d)\log F(Q;\vec{m};\vec{w};\vec q)\]
is regular in a neighbourhood of $q_1=\dots =q_d=0$ as a power series in $q_1,\dots,q_d$, meaning we have a lower bound $k_1,\dots ,k_d\geq -1$ for the above summation.

Furthermore, suppose $F$ is symmetric in $w_1,\dots ,w_r$ and symmetric in $m_1,\dots ,m_N$, then we can expand in the following elementary symmetric polynomial basis:
\[\log F(Q;\vec{m};\vec{w};e^{\lambda_1},\dots ,e^{\lambda_d})=\sum_{\substack{\mu,\xi\text{ partitions}\\k_1,\dots ,k_d\geq -1}}H_{\mu,\xi,\vec{k}}(Q)\prod_{i=1}^{\ell(\mu)}e_{\mu_i}(\vec w)\prod_{i=1}^{\ell(\xi)}e_{\xi_i}(\vec m)\lambda_1^{k_1}\dots \lambda_d^{k_d}\]
for some series $H_{\mu,\xi,\vec{k}}$.

Let $Y=\CC^d$ and 
\[\T_0=(\CC^*)^d,\T_1=(\CC^*)^N,\T_2=(\CC^*)^r\] 
with the natural actions on $Y,E=\CC^N\otimes\O_Y,V=\CC^r\otimes \O_Y$ respectively. Denote $\T=\T_0\times \T_1\times \T_2$. Say the equivariant cohomology ring of $\T$ is $H_\T^*(\pt)=\CC[\lambda_1,\dots ,\lambda_d;m_1,\dots ,m_N;w_1,\dots ,w_r]$. Then $V$ as a $\T$-equivariant bundle has equivariant Chern roots $w_1,\dots ,w_r$, and $E$ has Chern roots $m_1,\dots ,m_N$, so $e_i(m_1,\dots ,m_N)=c_i^\T(E),e_i(w_1,\dots ,w_r)=c_i^\T(V)$. Therefore
\[\log F(Q;\vec{m};\vec{w};e^{\lambda_1},\dots ,e^{\lambda_d})=\sum_{\substack{\mu,\xi\text{ partitions}\\k_1,\dots ,k_d\geq -1}}H_{\mu,\xi,\vec k}(Q)c_\mu(V)c_\xi(E)\lambda_1^{k_1}\dots \lambda_d^{k_d}.\]

For $\vec k=(k_1,\dots ,k_d)$ where $k_1,\dots ,k_d\geq -1$, there exist polynomials $E_{\vec k}$ such that 
\equa{\frac1{d!}\sum_{\tau\text{ permutation}}\lambda_1^{k_{\tau(1)}}\cdots\lambda_d^{k_{\tau(d)}}=\frac{E_{\vec k}(e_1(\lambda_1,\dots ,\lambda_d),\dots ,e_d(\lambda_1,\dots ,\lambda_d))}{\lam_1\cdots\lam_d}}
Now suppose $F$ is symmetric in the variables $q_1,\dots ,q_d$, so $H_{\mu,k}=H_{\mu,\tau(k)}$ for any permutation $\tau$. Hence
\equa{&\log F(Q;\vec{m};\vec{w};e^{\lambda_1},\dots ,e^{\lambda_d})\\
=&\sum_{\substack{\mu\text{ partition}\\k_1,\dots ,k_d\geq -1}}H_{\mu,\xi,\vec k}(Q)E_{\vec k}(e_1(\lambda_1,\dots ,\lambda_d),\dots ,e_d(\lambda_1,\dots ,\lambda_d))c_\mu(V)c_\xi(E)}
Note the equivariant weights of the tangent space $T_0Y$ are exactly $\lambda_1,\dots ,\lambda_d$, so $e_i(\lambda_1,\dots ,\lambda_d)=c_i^\T(Y)$. By Example  \ref{ex:equi-push}, we have
\equanum{\label{eqn:log-partition-E}\log F(Q;\vec{m};\vec{w};e^{\lambda_1},\dots ,e^{\lambda_d})=\sum_{\substack{\mu\text{ partition}\\k_1,\dots ,k_d\geq -1}}H_{\mu,\xi,\vec{k}}(Q)\int_Y E_{\vec k}(c_1(Y),\dots ,c_d(Y))c_\mu(V)c_\xi(E)}
Redistributing the terms, we get 
\[\log F(Q;\vec{m};\vec{w};e^{\lambda_1},\dots ,e^{\lambda_d})=\sum_{\mu,\nu, \xi \text{ partitions}}H_{\mu,\nu,\xi}(Q)\int_Y c_\nu(Y)c_\mu(V)c_\xi(E)\]
for some series $H_{\mu,\nu,\xi}$. Exponentiate both sides, and we obtain the following \emph{universal series expression} for $F$.

\begin{proposition}\label{prop:series-admissible}
Let $F(Q;\vec{m};\vec{w};\vec{q})\in \QQ(q_1,\dots ,q_d)[\![Q;m_1,\dots ,m_N;w_1,\dots ,w_r]\!]$ be admissible with respect to the variables $q_1,\dots,q_d$. Suppose $F$ is symmetric in $w_1,\dots ,w_r$, in $m_1,\dots ,m_N$, and symmetric in $q_1,\dots ,q_d$, then there exist power series $G_{\mu,\nu,\xi}(Q)$ labeled by partitions $\mu,\nu,\xi$, such that
\[F(Q;\vec{m};\vec{w};e^{\lambda_1},\dots,e^{\lambda_n})=\prod_{\mu,\nu}G_{\mu,\nu,\xi}(Q)^{\int_Y c_\nu(Y)c_\mu(V)c_\xi(E)}.\]
\end{proposition}

\section{Segre and Verlinde invariants on \texorpdfstring{$\CC^2$}{Lg}}

\subsection{Virtual equivariant invariants on Hilbert schemes}

Before defining virtual invariants, we recall the notion of a perfect obstruction theory in the sense of \cite[Definition~4.4]{BF}. For our purposes, we use the following simplified version.

\begin{definition}\label{defn:obs} Let $X$ be a scheme over $\CC$. An \emph{obstruction theory} is a complex of vector bundles
\[E^\bullet=[\dots \->E^{-2}\->E^{-1}\->E^0]\]
for some $a\in \ZZ$, together with a morphism in the derived category $D(\text{QCoh}(X))$ to the cotangent complex
\[\f:E^\bullet\->L_X^\bullet\]
such that $h^0(\f)$ is an isomorphism and $h^{-1}(\f)$ is surjective. It is a \emph{(2-term) perfect obstruction theory} if $E^i=0$ for $i\neq 0,-1$. The \emph{virtual tangent space} $T^{\vir}=E_\bullet=(E^\bullet)^*$ is the class of the dual complex of a given obstruction theory.
\end{definition}

Let $S$ be a surface and $E$ a torsion-free sheaf. It is well known that $\quot_S(E,n)$ admits an obstruction theory given by the dual complex of $\mathbf{R}\shom_{p}(\mathcal{I},\mathcal{F})$, where $\mathcal{I},\mathcal{F}$ are respectively the universal subsheaf and quotient sheaf. When $S$ is a projective surface, \cite[Lemma~1]{MOP1} shows that this obstruction theory is perfect of virtual dimension $nN$. Using this, we can define a virtual fundamental class $[\quot_S(E,n)]^\vir$ via the methods from \cite{BF,LiTian} as well as a virtual structure sheaf $\O^{\vir}$ using \cite{k-local-vir}. Applying the same argument for $S=\CC^2$ gives us a $\T$-equivariant perfect obstruction theory. We note that since $\mathcal{F}$ is compactly supported, the $\ext$-groups are finite dimensional vector spaces, so the steps involving Serre duality still work.

Let $S=\CC^2$ and $E=\oplus_{i=1}^N\O_S\<y_i\>$. Recall from the introduction the following tori:
\[\T_0=(\CC^*)^2, \indent \T_1=(\CC^*)^N,\indent\T_2=(\CC^*)^{r+s}.\] 
Set $\T=\T_0\times \T_1\times\T_2$, with
\equa{K_\T(\pt)&=\ZZ[t_1^{\pm1},t_2^{\pm1};y_1^{\pm1},\dots ,y_N^{\pm1};v_1^{\pm1},\dots ,v_{r+s}^{\pm1}],\\
H^*_\T(\pt)&=\CC[\lambda_1,\lambda_2;m_1,\dots ,m_N;w_1,\dots ,w_{r+s}].}
This gives us a $\T$-action on $\quot_S(E,n)$. For a $\T$-equivariant bundle $V$, the tautological bundle $V^{[n]}$ is also $\T$-equivariant. The $\T_1$-fixed quotients of $\quot_S(E,n)$ decomposes into the form
\[0\->\oplus_{i=1}^N I_i\<y_i\>\->\oplus_{i=1}^N \O_S\<y_i\>\->\oplus_{i=1}^N F_i\<y_i\>\->0.\]
Thus the $\T_1$-fixed locus can be identified as
\[\bigsqcup_{n_1+\dots +n_N=n}\hilb^{n_1}(S)\times\dots \times \hilb^{n_N}(S).\]
The $\T_0$-fixed locus of these Hilbert schemes are collections of finitely many reduced points, labeled by partitions. Consequently, the fixed locus $\quot_S(E,n)^\T$ consists of finitely many reduced points of form
\[Z_{\mu}=\left([Z_{1}],[Z_{2}],\dots ,[Z_{N}]\right)\in \hilb^{n_1}(S)\times\dots \times \hilb^{n_N}(S),\]
labeled by $N$-colored partitions $\mu=\left(\mu^{(1)},\dots,\mu^{(N)}\right)$.

For a vector bundle $V$ over $Y$, define
\[\Lambda_{z}(V):=\sum_{i\geq 0}[\Lambda^iV]z^i\in K^0(Y)[z],\indent\Lambda_{z}(-V):=\sum_{i\geq 0}[\sym^iV](-z)^i\in K^0(Y)[\![z]\!]\]
which extends to a homomorphism $\Lambda_z:(K^0(Y),+)\->(K^0(Y)[\![z]\!],\cdot)$.
We define the following equivariant invariants using virtual equivariant localization \cite{vir-loc} and K-theoretic virtual equivariant localization \cite[Theorem 5.3.1]{k-local-vir}. Since we are interested in comparing the Segre and Verlinde series, we convert the Verlinde invariants into cohomological invariants using virtual Hirzebruch-Riemann-Roch formula \cite[Corollary 1.2]{HRR-vir}. 

\begin{definition}
Let $S=\CC^2$ and
\[\alpha=[\oplus_{i=1}^r\O_Y\<v_i\>]-[\oplus_{i=r+1}^{r+s}\O_Y\<v_i\>]\in K_{\T}(S)\]
the \emph{equivariant virtual Segre, Chern, Verlinde series} on Quot schemes are respectively
\equa{&{\mathcal{S}}_S(E,\a;q):=\sum_{n=0}^\infty q^n\sum_{Z\in\quot_S(E,n)^\T}\frac{s(\alpha^{[n]}|_{Z})}{e(T_{Z}^\vir)},\\
&{\mathcal{C}}_S(E,\a;q):=\sum_{n=0}^\infty q^n\sum_{Z\in\quot_S(E,n)^\T}\frac{c(\alpha^{[n]}|_{Z})}{e(T_{Z}^\vir)},\\
&{\mathcal{V}}_S(E,\a;q):=\sum_{n=0}^\infty q^n\sum_{Z\in\quot_S(E,n)^\T}\frac{\ch(\det(\alpha^{[n]}|_{Z}))}{\ch(\Lambda_{-1}(T_{Z}^\vir)^\vee)}.
}
\end{definition}

We shall describe how to calculate these invariants, and refer to \cite[§5.1]{Quot-DT} and \cite[§3.3]{lim} for the following argument. On each $\T_1$-fixed locus 
\[D=\hilb^{n_1}(S)\times\dots \times \hilb^{n_N}(S),\]
the universal subsheaf and universal quotient sheaf of $D$ are $\bigoplus_{i=1}^N I_{\mathcal{Z}_i}\<y_i\>\text{ and } \bigoplus_{j=1}^N \O_{\mathcal{Z}_j}\<y_i\>$
where $\mathcal{Z}_i$ is the universal subscheme of $\hilb^{n_i}(S)$.
The virtual tangent bundle over $D$ 
is then
\[T^\vir_D=\bigoplus_{i,j=1}^N \mathbf{R}\shom_{p}( I_{\mathcal{Z}_i},\O_{\mathcal{Z}_j})\<y_i^{-1}y_j\>\]
where $p:D\times X\->D$ is the projection. Further restricting to each $Z_\mu=([Z_1],[Z_2],\dots ,[Z_N])\in\quot_S(E,n)^\T$ gives the virtual tangent bundle at $Z_{\mu}$ as follows
\equanum{\label{eqn:tangent2}T^\vir_{Z_{\mu}}=\bigoplus_{i,j=1}^N \ext( I_{Z_i},\O_{Z_j})\<y_i^{-1}y_j\>\in K_{\T}(S).}
To give an explicit formula for $T^\vir$, we consider a $\T_0$-equivariant free resolution of $I_{Z_i}$. We refer to \cite[Page 439]{Eisenbud} for the following \emph{Taylor resolution}. Say $I_{Z_i}$ is generated by monomials $m_1,\dots,m_s$. For each $k=0,\dots,s$, let $F_k$ be the free $\CC[x_1,\ldots, x_n]$-module module with basis $\{e_I\}$, indexed by subsets $I\seq \{1,\dots,s\}$ of size $k$. Set
\[m_I=\text{least common multiple of }\{m_i:i\in I\}.\]
For $k=1,\dots, s$, define differential $d_k:F_k\->F_{k-1}$ by 
\[d_k(e_I)=\sum_{j=1}^k(-1)^j\frac{m_I}{m_{I- \{i_j\}}}e_{I- \{i_j\}}\]
for each subset $I=\{i_1,\dots, i_k\}$ such that $i_1<\dots<i_k$. Giving each $e_I$ the weight of $m_I$, we obtain the $T_0$-equivariant free resolution
\[0\->F_s\->\dots \->F_0\->I_{Z_i}\->0\]
where
\[F_k=\bigoplus_{I\seq \{1,\dots, s\},|I|=k}\O_S\<m_I(t)\>\]
for some $d_{kI}\in\ZZ^2$. Define
\equanum{\label{def:poincare}P(I_{Z_i})=\sum_{k,|I|=k}(-1)^km_I(t).}
Note that the character of $\O_S=\CC[x_1,x_2]$ is $\sum_{i,j\geq 0}t_1^{-i}t_2^{-j}=1/(1-t_1^{-1})(1-t_2^{-1})$, so the character of $\O_{Z_i}=\O_S/I_{Z_i}$ is
\[Q_{i}:=\frac{1-P(I_{Z_i})}{(1-t_1^{-1})(1-t_2^{-1})}.\]
Therefore the character of $T^\vir_{Z_{\mu}}$ in $K_{\T}(\pt)$ can be expressed as
\equanum{\label{eqn:Tvir-quot2d}\bigoplus_{i,j=1}^N \ext( I_{Z_i},\O_{Z_j})\<y_i^{-1}y_j\>=&\sum_{i,j=1}^N\sum_{k,|I|=k}(-1)^k \hom(\O_S\<m_I(t)\>,\O_{Z_j})y_i^{-1}y_j\\
=&\sum_{i,j=1}^N\sum_{k,I}(-1)^k\O_{Z_j}\<m_I(t)^{-1}\>y_i^{-1}y_j\\
=&\sum_{i,j=1}^N\overline{P(I_{Z_i})}Q_{j}y_i^{-1}y_j\\
=&\sum_{i,j=1}^N( Q_{j}-(1-t_1)(1-t_2)\overline{Q_{i}} Q_{j})\cdot y_i^{-1}y_j
}
where $\overline{(\cdot)}$ denotes the involution $t_i\mapsto t_i^{-1}$. For the $\T$-equivariant bundle $V=\oplus_{i=1}^r\O_S\<v_i\>$, the fiber of $V^{[n]}$ over ${Z_{\mu}}=(Z_1,\dots Z_N)$ is the $rn$-dimensional representation
\equanum{\label{eqn:vb2d}\bigoplus_{i=1}^r\bigoplus_{j=1}^N\O_{Z_{j}}\<v_iy_j\>=\sum_{i=1}^r \sum_{j=1}^N\sum_{\square\in\mu^{(j)}} v_iy_jt_1^{-c(\square)}t_2^{-r(\square)}.}

Substituting the above calculations into the definition, we obtain the following expressions for the Chern and Verlinde series of vector bundles
\equanum{\label{eqn:CV2d-local}{\mathcal{C}}_S(E,V;q):=&\sum_{\mu} q^{|\mu|}
\frac{\prod_{j=1}^N\prod_{\square\in\mu^{(j)}}\prod_{i=1}^r\left(1+w_i+m_j-c(\square)\lam_1-r(\square)\lam_2\right)}{e(T^\vir|_{Z_\mu})},\\
{\mathcal{V}}_S(E,V;q):=&\sum_{\mu} q^{|\mu|}
\frac{\prod_{j=1}^N\prod_{\square\in\mu^{(j)}}\prod_{i=1}^rv_iy_jt_1^{-c(\square)}t_2^{-r(\square)}}{\ch(\Lambda_{-1}(T^\vir|_{Z_\mu})^\vee)}.
}

\subsection{Relation to projective toric surfaces}\label{sec:proj-reduction}
We consider what the equivariant invariants will be for a projective toric surface $S'$, and compare them with the case $S=\CC^2$. More details on this reduction can be found in \cite[§3.2]{GM}; see also \cite[§6.2]{Arbesfeld} and \cite[§3.2]{lyz}. 

Suppose we are interested in the integral
\[\int_{[\quot_S(E,n)]^\vir}f(V^{[n]}).\]
for a $\T$-equivariant bundle $V$ and some multiplicative genus $f$. On $S$ where $V$ has Chern roots $w_1,\dots, w_r$, this can be computed as
\[\int_{[\quot_S(E,n)]^\vir}f(V)=\sum_{|\mu|=n}
\frac{\prod_{j=1}^N\prod_{\square\in\mu^{(j)}}\prod_{i=1}^rf\left(w_i+m_j-c(\square)\lam_1-r(\square)\lam_2\right)}{e(T^\vir|_{Z_\mu})}.\]
Let $S'$ be a toric projective surface with a natural action by $\T_0=(\CC^*)^2$. Say the fixed points are $p_1,\dots ,p_M$ and the Chern roots of the tangent space of $S'$ at $p_i$ are $a_1^{(i)},a_2^{(i)}$, which live in $H_{\T_0}^*(\pt)=\CC[\lam_1,\lam_2]$. Suppose $E'$ is a $\T$-equivariant bundle with Chern roots $b_1^{(i)}$ at each $p_i$ such that the fixed locus of $\quot_{S'}(E',n)$ is a finite collection of reduced points for all $n$. Let $V'$ be an arbitrary $\T$-equivariant bundle on $S'$ with Chern roots $c_1^{(i)},\dots ,c_r^{(i)}$. By virtual equivariant localization, we have
\[\int_{[\quot_{S'}(E',n)]^\vir}f(V')=\left(\sum_{i=1}^M\left(\int_{[\quot_S(E,n)]^\vir}f(V)\right)\Big|_{\vec{\lam}\leadsto\vec{a}^{(i)},\vec m\leadsto\vec{b}^{(i)}, \vec w\leadsto\vec c^{(i)}}\right)\bigg\vert_{\vec{\lam}=\vec{m}=\vec{w}=0}.\]
where the symbol $\leadsto$ denotes a substitution of variables. Since $S'$ is compact, the sum on the right-hand side lives in $H^*_\T(\pt)=\CC[\lam_1,\lam_2]$. Therefore it is indeed valid to set $\lambda_1=\lambda_2= 0$, and the equality follows from virtual Bott residue formula. If we define $I_S(E,V;q)=\sum_{n=0}^\infty q^n\int_{[\quot_S(E,n)]} f(V)$, then from the above expansions, we obtain
\equanum{\label{eqn:proj-reduction}
I_{S'}(E',V';q)&
=\left(\prod_{i=1}^MI_S(E,V;q)\Big|_{\vec{\lam}\leadsto\vec{a}^{(i)},\vec m\leadsto\vec{b}^{(i)}, \vec w\leadsto\vec c^{(i)}}\right)\bigg\vert_{\vec{\lam}=\vec{m}=\vec{w}=0}.
}

Suppose furthermore that the series the above expression have the following universal series structures
\equanum{\label{eqn:series-structure}I_{S'}(E',V';q)=\prod U^\vir(q)^{\gamma'},\indent I_{\CC^2}(E,V;q)=\prod U(q)^{\gamma}}
where $\gamma'$ are numbers dependent on characteristic classes of $E',V'$ and $S'$ and $\gamma$ are cohomology classes in $H^*_\T(\pt)_\loc$ dependent on characteristic classes of $E,V,S$. Then (\ref{eqn:proj-reduction}) implies,
\equa{I_{S'}(E',V';q)&=\left(\prod_{i,\gamma} U(q)^{\gamma}\Big|_{\vec{\lam}\leadsto\vec{a}^{(i)},\vec m\leadsto\vec{b}^{(i)}, \vec w\leadsto\vec c^{(i)}}\right)\bigg\vert_{\vec{\lam}=\vec{m}=\vec{w}=0}\\
&=\prod_{\gamma\in H^0_\T(\pt)} U(q)^{\gamma'}.
}
By universality, we conclude that if the series structure (\ref{eqn:series-structure}) exists, then the series $U^\vir(q)$ for the projective case coincide with the series $U(q)$ whose powers $\gamma$ lie in $H^0_\T(\pt)$.

\subsection{Universal series expansion}
A general tactic for studying the Segre and Verlinde series is using a more general genus. In the non-virtual surface case this could be \cite[(1.1)]{GM}. In the virtual case, we use the invariant (\ref{defn:nek2d}) defined below. For Calabi-Yau 4-folds, we consider the Nekrasov genus (\ref{def:nek-genus}), introduced by \cite{magcolor}.

For projective surfaces, define an auxiliary virtual invariant as follows
\equanum{\label{defn:nek2d}{\mathcal{N}}_S(E,\a;q,z):=\sum_{n=0}^\infty q^n\chi^{\vir}\left(\quot_S(E,n),\Lambda_{-z}\a^{[n]}\right).}
where $z$ is considered as the weight of an extra $\CC^*$-action that is trivial on $S$ and $\quot_S(E,n)$. We shall refer to this as the \emph{Nekrasov genus} for Quot schemes of surfaces (c.f. (\ref{def:nek-genus})).

We generalize this to the equivariant setting using virtual equivariant localization. On $S=\CC^2$ for vector bundles, this is given by:
\equanum{\label{eqn:nek2d-local}{\mathcal{N}}_S(E,V;q,z):=&\sum_{\mu} q^{|\mu|}
\frac{\prod_{j=1}^N\prod_{\square\in\mu^{(j)}}\prod_{i=1}^r(1-t_1^{-c(\square)}t_2^{-r(\square)}v_iy_jz)}{\ch(\Lambda_{-1}(T^\vir|_{Z_\mu})^\vee)}\\
&\in\QQ(t_1,t_2;y_1,\dots ,y_N)[\![q,z]\!] .
}
The following Chern and Verlinde limits are satisfied, analogous to \cite[Proposition 3.5]{GM}.

\begin{lemma}\label{lem:CV-limit}
For $S=\CC^2$, the Chern series and the Verlinde series can be retrieved from ${\mathcal{N}}_S$ by taking limits. We have
\equa{{\mathcal{C}}_S(E,V;q)&=\lim_{\e\->0}
    {\mathcal{N}}_S\left(E,V;(-1)^Nq\e^{N-r}(1+\e)^{r},(1+\e)^{-1}\right)\big|_{\vec\lam\leadsto-\e\vec\lambda,\vec w\leadsto -\e \vec w,\vec m\leadsto-\e\vec m},\\
    {\mathcal{V}}_S(E,V;q)&=\lim_{\e\->0}{\mathcal{N}}_S\left(E,V;(-1)^{r}q\e^{r},\e^{-1}\right).}
\end{lemma}
\begin{proof}
For the Chern limit, first consider the substitutions $\lam_i\leadsto -\e\lam_i,w_i\leadsto -\e w_i,m_i\leadsto -\e m_i$. This turns the term $\prod_{j=1}^N\prod_{\square\in\mu^{(j)}}\prod_{i=1}^r(1-t_1^{-c(\square)}t_2^{-r(\square)}v_iy_j(1+\e)^{-1})$ into 
\equa{&\prod_{j=1}^N\prod_{\square\in\mu^{(j)}}\prod_{i=1}^r1-\frac{e^{-\e(w_i+m_j-c(\square)\lam_1-r(\square)\lam_2)}}{1+\e}\\
=&\frac{1}{(1+\e)^{r|\mu|}}\prod_{j=1}^N\prod_{\square\in\mu^{(j)}}\prod_{i=1}^r(1+\e-e^{-\e(w_i+m_j-c(\square)\lam_1-r(\square)\lam_2)})\\
=&\left(\frac{\e}{1+\e}\right)^{r|\mu|}\prod_{j=1}^N\prod_{\square\in\mu^{(j)}}\prod_{i=1}^r(1-c(\square)\lam_1-r(\square)\lam_2+w_i+m_j+O(\e))
}
For the denominator in the sum (\ref{eqn:nek2d-local}), we note that for a Chern root $x$, substituting it by $-\e x$ turns $1-e^{-x}=x-\frac{x^2}2+...$ into $1-e^{\e x}=-\e(x+O(\e))$. Therefore after the substitution, the denominator $\ch(\Lambda_{-1}(T^\vir|_{Z_\mu})^\vee)$ becomes
\[(-1)^{N|\mu|}\e^{N|\mu|}(e(T^\vir_{Z_\mu})+O(\e))\]
Substituting back into (\ref{eqn:nek2d-local}), the Chern limit becomes the limit of
\equa{&\sum_{\mu} (-1)^{N|\mu|}q^{|\mu|}\e^{(N-r)|\mu|}(1+\e)^{r|\mu|}\cdot \frac{\e^{r|\mu|}}{(-1)^{N|\mu|}\e^{N|\mu|}(1+\e)^{r|\mu|}}\cdot\\
&\frac{\prod_{j=1}^N\prod_{\square\in\mu^{(j)}}\prod_{i=1}^r(1-c(\square)\lam_1-r(\square)\lam_2+w_i+m_j+O(\e))}{(e(T^\vir_{Z_\mu})+O(\e))}\\
=&\sum_{\mu} q^{|\mu|}\frac{\prod_{j=1}^N\prod_{\square\in\mu^{(j)}}\prod_{i=1}^r(1-c(\square)\lam_1-r(\square)\lam_2+w_i+m_j+O(\e))}{(e(T^\vir_{Z_\mu})+O(\e))}
}
which converges to ${\mathcal{C}}_S(E,V;q)$ by (\ref{eqn:CV2d-local}). 

For the Verlinde series, we have
\equa{&\lim_{\e\->0}{\mathcal{N}}_S(E,V;(-1)^{r}q\e^{r},\e^{-1})\\
=&\lim_{\e\->0} \sum_{\mu} (-1)^{r|\mu|}q^{|\mu|}\e^{r|\mu|}
\cdot \frac{\prod_{j=1}^N\prod_{\square\in\mu^{(j)}}t_1^{c(\square)}t_2^{r(\square)}\prod_{i=1}^r(1-t_1^{-c(\square)}t_2^{-r(\square)}v_iy_j\e^{-1})}{\ch(\Lambda_{-1}(T^\vir|_{Z_\mu})^\vee)}\\
=&\lim_{\e\->0}\sum_{\mu} q^{|\mu|}
\frac{\prod_{j=1}^N\prod_{\square\in\mu^{(j)}}\prod_{i=1}^r(t_1^{-c(\square)}t_2^{-r(\square)}v_iy_j-\e)}{\ch(\Lambda_{-1}(T^\vir|_{Z_\mu})^\vee)}\\
=&{\mathcal{V}}_S(E,V;q).
}

\end{proof}

Before starting the proof for universal series expressions, let us discuss how the expansion of ${\mathcal{N}}_S(E,V;q,z)$ as a formal Laurent series in the variables $\vec{\lam},\vec{m},\vec{w},q,z$ would look like. In \cite[Proposition 3.2]{Arbesfeld}, N. Arbesfeld shows that invariants such as $[q^n]{\mathcal{N}}_S(E,V;q,z)$ can be written as a quotient whose numerator is a Laurent polynomial in $\vec{t},\vec{y},\vec{v},z$, and whose denominator is of the form $\prod_{\mathsf{w}}(1-\mathsf{w})$ for some \emph{non-compact weights} $\mathsf{w}$ in the sense of the following definition.

\begin{definition}\cite[Definition 3.1]{Arbesfeld} Let $M$ be a quasi-projective scheme with an action by some torus $\T$. For a weight $\mathsf{w}\in \T^\vee$, denote $\T_{\mathsf{w}}$ the maximal subtorus of $\T$ containing $\ker\mathsf{w}$. If the fixed locus $M^{\T_\mathsf{w}}$ is proper, then $\mathsf{w}$ is a \emph{compact weight}, otherwise, it is a \emph{non-compact weight}.

\end{definition}

Fasola-Monavari-Ricolfi used this to prove that the K-theoretic Donaldson-Thomas partition functions on $\CC^3$ are Laurent polynomials with respect to the variables $y_1,\dots,y_N$ \cite[Theorem 6.5]{Quot-DT}. We give an outline of their argument, applied to the invariant ${\mathcal{N}}_S$ for $S=\CC^2$. First note that by (\ref{eqn:Tvir-quot2d}), for any $N$-colored partition $\mu$, we have
\[\frac1{\ch(\Lambda_{-1}(T^\vir|_{Z_\mu})^\vee)}=A(\vec{t}\,)\prod_{1\leq i,j\leq N,i\neq j }\frac{\prod_{a\in A_{ij}}(1-y_i^{-1}y_jt^{a})}{\prod_{b\in B_{ij}}(1-y_i^{-1}y_jt^{b})}\]
for some series $A(\vec{t}\,)\in\QQ[\![t_1,t_2]\!]_{\loc}$ and some sets of weights $A_{ij},B_{ij}$. We shall show that the denominator of ${\mathcal{N}}_S$ does not have factors of the form $(1-y_i^{-1}y_jt^b)$ for any $i\neq j$ and $b\in\ZZ^2$. By \cite[Proposition 3.2]{Arbesfeld}, we need to prove $\mathsf{w}=y_i^{-1}y_jt^b$ is a compact weight. Since
\[\ker\mathsf{w}=\{(\vec{t},\vec{y},\vec{v}):y_i=y_jt^b\}\]
is itself a torus, we have $\T_{\mathsf{w}}=\ker\mathsf{w}$. By definition, it suffices to show $\quot_S(E,n)^{\T_{\mathsf{w}}}$ is proper. With the automorphism $\T\->\T$ defined by 
\[(\vec{t},y_1,\dots,y_j,\dots,y_N,\vec{v})\mapsto(\vec{t},y_1,\dots,y_jt^b,\dots,y_N,\vec{v}),\]
we identify the subgroup $\T_{\mathsf{w}}$ to $\T_0\times\{(w_1,\dots,w_N):w_i=w_j\}\times \T_2$, which contains the subgroup $\T_0=\T_0\times\{(1,\dots,1)\}$. This gives us an inclusion
\[\quot_S(E,n)^{\T_{\mathsf{w}}}\xhookrightarrow{}\quot_S(E,n)^{\T_0}.\]
The quotients in the fixed locus $\quot_S(E,n)^{\T_0}$ are all supported at the origin $0\in\CC^2$, so the fixed locus lies inside the punctual Quot scheme $\quot_S(E,n)_0$. The punctual Quot scheme is proper since it is a fiber of the Quot-to-Chow map $\quot_S(E,n)\->\sym^nS$, which is a proper morphism \cite[Remark 3.4]{Quot-DT}. In conclusion, $[q^n]{\mathcal{N}}_S(E,V;q,z)$ is a Laurent polynomial with respect to the variables $y_1,\dots,y_N$, so it can be expanded into a power series with respect to the cohomological parameters $m_1,\dots,m_N$.

Furthermore, if $\mathsf{w}$ is a weight that contains both $t_1$ and $t_2$, then we have $\T_{\mathsf{w}}\cong\{(t_1,t_2):t_1t_2=1\}\times \T_1\times \T_2$. The fixed locus of this subgroup remains the same as that of $\T$, as explained in the next section for reduced invariants. Therefore $\mathsf{w}$ is a compact weight, and the denominator of ${\mathcal{N}}_S$ will not contain factors of the form $(1-t_1^at_2^b)$ for any $a\neq 0$, $b\neq 0$. This means in cohomology, $[q^n]{\mathcal{N}}_S(E,V;q,z)$ can be expanded into a Laurent series in $\lam_1,\lam_2$ whose coefficients are power series in $\vec{m},\vec{w},z$, where the degrees on $\lam_1,\lam_2$ are bounded below individually. We shall see the importance of this lower bound in the proof of the following theorem.

\begin{theorem}\label{thm:SV-univ-sieres}
Let $S=\CC^2$. For any $r\in\ZZ$, $N>0$, there exist universal power series $A_{\mu,\nu,\xi}(q),B_{\mu,\nu,\xi}(q)$, dependent on $N$ and $r$, such that for $E=\oplus_{i=1}^{N}\O_S\<y_i\>$ and $\a\in K_\T(S)$ of rank $r$, the equivariant virtual Segre and Verlinde series on $\quot_S(E,n)$ can be written as the following infinite products
\equa{{\mathcal{S}}_S(E,\alpha;q)=&\prod_{\mu,\nu,\xi\text{ partitions}}A_{\mu,\nu,\xi}(q)^{\int_{S}c_\mu(\a) c_\nu(S)  c_\xi(E)c_1(S)},\\
{\mathcal{V}}_S(E,\alpha;q)=&\prod_{\mu,\nu,\xi\text{ partitions}}B_{\mu,\nu,\xi}(q)^{\int_{S}c_\mu(\a) c_\nu(S)  c_\xi(E)c_1(S)},\\
{\mathcal{C}}_S(E,\alpha;q)=&\prod_{\mu,\nu,\xi\text{ partitions}}C_{\mu,\nu,\xi}(q)^{\int_{S}c_\mu(\a) c_\nu(S)  c_\xi(E)c_1(S)}.
}
\end{theorem}
\begin{proof}
We begin with the case where $\a$ is a vector bundle $V$. Assume $V=\oplus_{i=1}^r\O_S\<v_i\>$, and at end of the proof, we can generalize this to arbitrary $\T$-equivariant bundles by substituting $\T$-weights into the variables $v_1,\dots ,v_r$.

Begin by expanding $\log {\mathcal{N}}_S(E,V;q,z)$ as a Laurent series in $\lam_1,\lam_2$ as follows:
\[\log {\mathcal{N}}_S(E,V;q,z)=\sum_{ (j,k)\in\ZZ^2}H_{j,k}(q,z;\vec{m};\vec{w})\lambda_1^{j}\lambda_2^{k}\] 
for some series $H_{j,k}\in\QQ[\![q,z;m_1,\dots ,m_N;w_1,\dots ,w_r]\!]$. By the symmetry in $w_1,\dots ,w_r$ and the symmetry in $m_1,\dots, m_N$, this expands to 
\equa{\log {\mathcal{N}}_S(E,V;q,z)=&\sum_{\substack{\mu,\xi \text{ partitions}\\j,k\geq -1}}G_{\mu,\xi,j,k}(q,z)\cdot\lam_1^j\lam_2^{k}c_\mu(V)c_\xi(E)\\
&+\sum_{\substack{\mu,\xi \text{ partitions}\\\min\{j,k\}\leq -2}}G_{\mu,\xi,j,k}(q,z)\cdot  \lam_1^j\lam_2^{k}c_\mu(V)c_\xi(E)
}
for some series $G_{\mu,\xi,j,k}\in\QQ[\![q,z]\!]$.

Our goal is to get a universal series expression by exponentiating the above equality. To do so, we show the terms in the second summation vanish using the relations from Section \ref{sec:proj-reduction}. This proves that ${\mathcal{N}}_S(E,V;q,z)$ is admissible, from which we deduce the desired expressions by taking the limits of Lemma \ref{lem:CV-limit}.

Let $S'$ be a toric projective surface with a natural $\T_0=(\CC^*)^2$-action and fixed points $p_1,\dots, p_M$. Set
\[E'=\bigoplus_{j=1}^N\O_{S'}\<y_j\>,\indent V'=\bigoplus_{j=1}^r\O_{S'}\<v_j\>\]
with Chern roots $m_1,\dots ,m_N$ and $w_1,\dots ,w_r$ respectively at each $p_i$, independent of $i=1,\dots,M$. Then (\ref{eqn:proj-reduction}) applied to $\mathcal{N}$ becomes
\[{\mathcal{N}}_{S'}(E',V';q,z)=\left(\prod_{i=1}^M{\mathcal{N}}_S(E,V;q,z)|_{\vec\lam\leadsto\vec  a^{(i)}}\right)\bigg\vert_{\vec{\lam}=\vec{w}=\vec{m}=0}.\]
Substituting the previous expansion of $\log {\mathcal{N}}_S(E,V;q,z)$, we see that
\equa{\log {\mathcal{N}}_{S'}(E',V';q,z)=&\left(\sum_{i=1}^{M}\sum_{\substack{\mu,\xi \text{ partitions}\\\min\{i,j\}\geq -1}}G_{\mu,\xi,i,j}(q,z)\cdot  \lam_1^j\lam_2^kc_\mu(V)c_\xi(E)\bigg\vert_{\vec\lam\leadsto\vec  a^{(i)}}\right)\Bigg\vert_{\vec{\lam}=\vec{w}=\vec{m}=0}\\
&+\left(\sum_{i=1}^{M}\sum_{\substack{\mu,\xi \text{ partitions}\\\min\{i,j\}\leq -2}}G_{\mu,\xi,i,j}(q,z)\cdot  \lam_1^j\lam_2^kc_\mu(V)c_\xi(E)\bigg\vert_{\vec\lam\leadsto \vec a^{(i)}}\right)\Bigg\vert_{\vec{\lam}=\vec{w}=\vec{m}=0}
.
}
Since the elementary symmetric polynomials form a basis for symmetric polynomials, we know the coefficients in front of $q,z$ of the terms
\equanum{\label{poly-1}\sum_{i=1}^{M}\sum_{j,k\in\ZZ}G_{\mu,\xi,j,k}(q,z)\cdot  \lam_1^j\lam_2^k\bigg\vert_{\vec\lam\leadsto \vec a^{(i)}}}
are necessarily power series in $\lam_1,\lam_2$ for each $\mu,\xi$ (as opposed to arbitrary Laurent series).

Let $S'=\PP^1\times \PP^1$, with $\T_0$-action
\[(t_1,t_2)\cdot([x_0:x_1],[y_0:y_1])=([x_0:t_1x_1],[y_0,t_2y_1]).\]
We refer to \cite[§3.7]{lyz} for the following computations of equivariant weights. The fixed points are 
\[p_1=p_{00}=([1:0],[1:0]),\indent p_2=p_{01}=([1:0],[0:1]),\]
\[p_3=p_{10}=([0:1],[1:0]),\indent p_4=p_{11}=([0:1],[0:1]).\]
The corresponding weights are $\vec{a}^{(1)}=\vec{a}^{(00)},\vec{a}^{(2)}=\vec{a}^{(01)},\vec{a}^{(3)}=\vec{a}^{(10)},\vec{a}^{(4)}=\vec{a}^{(11)}$ where
\[a_1^{(ij)}=(-1)^i\lam_1,\indent  a_2^{(ij)}=(-1)^j\lam_2\]
for $i,j\in\{0,1\}$. Substituting into (\ref{poly-1}), we see the summands with odd $j$ or $k$ would cancel each other out, leaving us with
\equa{\sum_{j,k\text{ even}}G_{\mu,\xi,j,k}(q,z)\cdot  4\lam_1^j\lam_2^k.
}
Since $\lam_1^j\lam_2^k$ are linearly independent for all distinct $j,k$, and they are not polynomials for $\min\{j,k\}\leq -2$, the coefficients $G_{\mu,\xi,j,k}$ must all be 0 for these $j$ and $k$.

Having dealt with the case where $j,k$ are both even, we would like to apply the same argument to the other cases. To do so we need to solve the problem that the summands vanish whenever one of $j,k$ is odd. The fixed points on $\quot_{S'}(E',n)$ correspond to $M$-tuples of $N$-coloured partitions $(\mu^{(1)},\dots ,\mu^{(M)})$, where $\mu^{(i)}=(\mu^{(i,1)},\dots,\mu_i^{(i,N)})$ and each $\mu^{(i,j)}$ is a partition for $i=1,\dots,M$ and $j=1,\dots,N$. If we replace $w_k$ by $u_k(l+\lam_1+\lam_2)^2$ for some symmetric polynomial $p$ and numbers $l,u_1,\dots ,u_r$, the Chern roots of $(\O_{S'}\<w_k\>)^{[n]}$ would be replaced by
\[\bigcup_{i=1}^M\bigcup_{j=1}^N\bigcup_{\square\in\mu^{(i,j)}}\{u_k (l+a_1^{(i)}+a_2^{(i)})^2+m_j-c(\square)a_1^{(i)}-r(\square)a_2^{(i)}\}.\]
We claim that taking symmetric series of these Chern roots would result in terms composed of symmetric series of Chern roots of $K_{S'}^{[n]}$, and that of $\O_{S'}^{[n]}$, which are respectively given by the sets
\[\bigcup_{i=1}^M\bigcup_{j=1}^N\bigcup_{\square\in\mu^{(i,j)}}\{a_1^{(i)}+a_2^{(i)}+m_j-c(\square)a_1^{(i)}-r(\square)a_2^{(i)}\},\]
\[\bigcup_{i=1}^M\bigcup_{j=1}^N\bigcup_{\square\in\mu^{(i,j)}}\{m_j-c(\square)a_1^{(i)}-r(\square)a_2^{(i)}\}.\]
This follows from Lemma \ref{lem:sym-poly} by setting $\vec{x}=(a_1^{(i)}+a_2^{(i)})_{i=1,\dots M},\vec{y}=(m_j-c(\square)a_1^{(i)}-r(\square)a_2^{(i)})_{i=1,\dots,M}^{\square\in\mu^{(i,j)}}$. 
Therefore after replacing $w_k$ by $u_k(l+a_1^{(i)}+a_2^{(i)})$, the resulting invariant is still an integral of characteristic classes of tautological bundles, so the new version of (\ref{poly-1}) remains a power series in $\lam_1,\lam_2$. View $u_1,\dots ,u_r$ as formal variables and replace them with $w_1,\dots ,w_r$, we see in total we have replaced each $w_k$ by $(l+a_1^{(i)}+a_2^{(i)})w_k$ for $k=1,\dots,r$.

As a result of the previous paragraph, the coefficients of
\equa{&\sum_{i=1}^{M}\sum_{j,k}G_{\mu,\xi,j,k}(q,z)\cdot  \lam_1^j\lam_2^k\cdot (l+\lambda_1+\lambda_2)^{2|\mu|}\bigg\vert_{\vec\lambda\leadsto\vec a^{(i)}}\\
=&\sum_{s=0}^{2|\mu|}\sum_{i=1}^M\sum_{j,k}\binom{2|\mu|}{s}G_{\mu,\xi,j,k}(q,z)\cdot\lam_1^j\lam_2^k\cdot l^{2|\mu|-s}(\lam_1+\lam_2)^{s}\bigg\vert_{\vec\lam\leadsto \vec a^{(i)}}}
are power series in $\lam_1,\lam_2$ for any integer $l\geq 0$. When $\mu\neq (0)$, the matrix formed by the vectors \[\left(\binom{2|\mu|}{0}l^{2|\mu|},\binom{2|\mu|}{1}l^{2|\mu|-1},\binom{2|\mu|}{2}l^{2|\mu|-2},\dots ,\binom{2|\mu|}{2|\mu|}l^{0}\right)\] 
for $l=1,2,3,\dots $ has maximal rank, we may take a linear combination of the above expression, and get that
\equa{&\sum_{i=1}^M\sum_{j,k\in\ZZ}G_{\mu,\xi,j,k}(q,z)\cdot\lam_1^j\lam_2^k\cdot (\lam_1+\lam_2)^{s}\bigg\vert_{\vec\lam\leadsto \vec a^{(i)}}}
is a power series in $\lam_1,\lam_2$ for each $s=0,1,\dots ,2|\mu|$.

Take $s=2$, we get
\equa{&\sum_{i=1}^{M}\sum_{j,k}G_{\mu,\xi,j,k}(q,z)\cdot  \lam_1^j\lam_2^k\cdot (\lambda_1+\lambda_2)^{2}\bigg\vert_{\vec\lam\leadsto \vec a^{(i)}}\\
=&\sum_{\substack{j,k\text{ odd}}}G_{\mu,\xi,j,k}(q,z)\cdot  8\lam_1^{j+1}\lam_2^{k+1} .}
Again, since $\lam_1^{j+1}\lam_2^{k+1}$ are linearly independent for distinct $j,k$ and are not polynomials in $\lam_1,\lam_2$ for any $\min\{j,k\}\leq -2$, we know $G_{\mu,\xi,j,k}=0$ whenever $j,k$ are both odd and $\mu\neq (0)$. 

In the case one of $j,k$ is odd and the other is even, continuing to assume $\mu\neq (0)$, we take $s=1$ and get
\equa{&\sum_{i=1}^{M}G_{\mu,\xi,j,k}(q,z)\cdot  \lam_1^j\lam_2^k\cdot (\lam_1+\lam_2)\bigg\vert_{\vec\lam\leadsto \vec a^{(i)}}\\
=&\begin{cases}
G_{\mu,\xi,j,k}(q,z)\cdot  8\lam_1^{j+1}\lam_2^{k}\text{, if $j$ odd, $k$ even}\\
G_{\mu,\xi,j,k}(q,z)\cdot  8\lam_1^{j}\lam_2^{k+1}\text{, if $k$ odd, $j$ even}.
\end{cases}
}
Although these are not polynomials when $\min\{j,k\}\leq -2$, we see there might be some linear dependence, i.e. we could have
\[G_{\mu,\xi,j,k}=-G_{\mu,\xi,j+1,k-1}\]
for $j$ odd and $k$ even, and terms canceling each other out in the sum. To solve this issue, we further apply the argument to the $s=3$ case and obtain the following dependencies
\[G_{\mu,\xi,j,k}=-G_{\mu,\xi,j+3,k-3}\]
for all $j$ odd, $k$ even and $\min\{j,k\}\leq -4$. Combining these relations we see for $\min\{j,k\}\leq -2$, there exist some constants $C_{\mu,\xi,a,b,l}^{\pm}$, labeled by the partitions $\mu,\xi$, integers $a,b,l$ and a sign $\pm$, such that
\equa{G_{\mu,\xi,j,k}(q,z)=&\sum_{a,b}((-1)^j-(-1)^k)C_{\mu,\xi,a,b,j+k}^{\pm}q^{a}z^{b}\\
=&\begin{cases}\sum_{a,b}2C_{\mu,\xi,a,b,j+k}^+q^{a}z^{b}\indent&\text{ if $j$ even, $k$ odd, $j\geq 0$}\\
\sum_{a,b}-2C_{\mu,\xi,a,b,j+k}^+q^{a}z^{b}\indent&\text{ if $j$ odd, $k$ even, $j> 0$,}\\
\sum_{a,b}2C_{\mu,\xi,a,b,j+k}^-q^{a}z^{b}\indent&\text{ if $j$ even, $k$ odd, $j<0$}\\
\sum_{a,b}-2C_{\mu,\xi,a,b,j+k}^-q^{a}z^{b}\indent&\text{ if $j$ odd, $k$ even, $j<0$,}\\
\end{cases}
}
The reason for the superscript $\pm$ is due to cases such as $j=-1,k=0$, where we would have $\min\{j,k\}>-2$, so the dependence does not necessarily hold. Because of this gap, we can not always relate the coefficient when $j\geq 0$ to $j<0$, resulting in separated cases. By the paragraphs preceding this theorem, for a fixed $a$, the degrees $j,k$ on $\lam_1,\lam_2$ of the $[q^{a}]$ coefficient are bounded below. However the above indicates that the constants $C_{\mu,\xi,a,b,l}^\pm$ only depend on the value $l=j+k$, and we can make $j$ or $k$ arbitrarily small. Hence $C_{\mu,\xi,a,b,l}^\pm=0$ whenever $\mu\neq (0)$.

With all the vanishings of $G_{\mu,\xi,j,k}$, we write
\equa{\log {\mathcal{N}}_S(E,V;q,z)=\sum_{\substack{\mu,\xi \text{ partitions}\\j,k\geq -1}}G_{\mu,\xi,j,k}(q,z)\cdot\lam_1^j\lam_2^k\cdot c_\mu(V)\\
+\sum_{\substack{\xi \text{ partition}\\\min\{j,k\}\leq -2}}G_{(0),\xi,j,k}(q,z)\cdot  \lam_1^j\lam_2^k.
}
To deal with the terms $G_{(0),\xi,j,k}(q,z)$ for $\min\{j,k\}\leq -2$, we apply Lemma \ref{lem:differential} and find
\[D_zG_{(0),\xi,j,k}(q,z)=kG_{(1),\xi,j,k}(q,z)=0,\]
so $G_{(0),j,k}$ is constant with respect to the variable $z$. Let us attempt to extract the $[q^{n}\lam_1^j\lam_2^kc_{(0)}(V)c_\xi(E)]$ coefficient of the Chern series from $G_{(0),\xi,j,k}$ using the Chern limit of Lemma \ref{lem:CV-limit}. This results in a limit $\e\->0$ of the term $\e^{n(N-r)}\e^{|\xi|}\e^{j+k}$, which does not make sense when the rank $r$ is sufficiently large, so we must have $G_{(0),\xi,j,k}=0$ for such $r$. To generalize this to arbitrary ranks, we apply \cite[Lemma 3.3]{GM} to ${\mathcal{N}}_S$, which says the coefficients of ${\mathcal{N}}_S$ are polynomials in $r$ when $r\geq 0$. Now we can write
\equa{\log {\mathcal{N}}_S(E,V;q,z)=\sum_{\substack{\mu,\xi \text{ partitions}\\j,k\geq-1}}G_{\mu,\xi,j,k}(q,z)\cdot\lam_1^j\lam_2^k\cdot c_\mu(V)c_\xi(E).
}
As noted in \cite[(31)]{OP}, the obstruction on $\hilb^n(S)$ at a fixed point $[Z_\mu]$ is $(K_S^{[n]})^\vee|_{Z_\mu}$. From (\ref{eqn:vb2d}), we see a copy of $K_S^\vee=t_1t_2$ is in $K_S^{[n]}|_{Z_\mu}$. By (\ref{eqn:Tvir-quot2d}), the obstruction bundle on $\quot_S(E,n)$ at any fixed point has at least one copy of $K_S^\vee$ as a direct summand. For a line bundle $L$, we have
\[\ch(\Lambda_{-1}L^\vee)=1-e^{-c_1(L)}=e(L)\cdot(1+\dots )\] 
where $\dots $ are some omitted terms in $H_\T^{>0}(\pt)$. Therefore $1/\ch(\Lambda_{-1}(T^\vir_Z)^\vee)$ has a factor of $e(K_S^\vee)=c_1(S)=\lambda_1+\lambda_2$ in its numerator. We also note that this factor does not appear in the denominator because if we pass to the subtorus $\{(t_1,t_2):t_1t_2=1\}$, the Zariski tangent space has no $\T$-fixed parts: by (\ref{eqn:tangent2}), the fixed part can only come from the direct summands with $i=j$, which correspond to the Hilbert scheme case; but by \cite[(1.6)]{Arbesfeld}, these summands have no fixed parts because $a(\square),l(\square)\geq 0$ for any box $\square$. Therefore we may extract this factor of $c_1(S)$ and obtain
\equa{\log {\mathcal{N}}_S(E,V;q,z)=\sum_{\substack{\mu,\xi \text{ partitions}\\j,k\geq-1}}H_{\mu,\xi,j,k}(q,z)\cdot\lam_1^j\lam_2^k\cdot c_\mu(V)c_\xi(E)c_1(S).
}
for some series $G_{\mu,\xi,j,k}\in\QQ[\![q,z]\!]$. Furthermore, since $j,k$ are now bounded below by $-1$, multiplying by $\lam_1\lam_2$ would give us a power series expansion in $\lam_1,\lam_2$, allowing us to use the symmetry in $\lam_1,\lam_2$ and write
\equanum{\label{eqn:log-nek2d-vir}\log {\mathcal{N}}_S(E,V;q,z)=\sum_{\mu,\nu, \xi\text{ partitions}}H_{\mu,\nu,\xi}(q,z)\cdot \int_Sc_\mu(V) c_\nu(S)  c_\xi(E) c_1(S).
}
for some series $H_{\mu,\nu,\xi}\in\QQ[\![q,z]\!]$.

Finally, taking Chern and Verlinde limits of $H_{\mu,\nu,\xi}$ as in Lemma \ref{lem:CV-limit}, then exponentiating gives us series $C_{\mu,\nu,\xi},B_{\mu,\nu,\xi}$ such that
\equa{
{\mathcal{C}}_S(E,V;q)=&\prod_{\mu,\nu,\xi\text{ partitions}}C_{\mu,\nu,\xi}(q)^{\int_{S}c_\mu(V) c_\nu(S)  c_\xi(E)c_1(S)},\\
{\mathcal{V}}_S(E,V;q)=&\prod_{\mu,\nu,\xi\text{ partitions}}B_{\mu,\nu,\xi}(q)^{\int_{S}c_\mu(V) c_\nu(S)  c_\xi(E)c_1(S)}.
}
and the fact that ${\mathcal{S}}_S(E,V;q)={\mathcal{C}}_S(E,-V;q)$ implies that there exists series $A_{\mu,\nu,\xi}$ such that
\[{\mathcal{S}}_S(E,V;q)=\prod_{\mu,\nu,\xi\text{ partitions}}A_{\mu,\nu,\xi}(q)^{\int_{S}c_\mu(V) c_\nu(S)  c_\xi(E)c_1(S)}.\]
To generalize this to arbitrary K-theory classes $\a=[V']-[V'']\in K_\T(S)$ for equivariant bundles $V',V''$ of rank $m,l$ respectively, we apply \cite[Lemma 3.3]{GM} once more; it states that the invariants for $\alpha$ are obtained by substituting
\[r\leadsto m-l,\indent p_n(v_1,v_2,\dots ,v_r)\leadsto p_n(v'_1,v'_2,\dots ,v'_m)-p_n(v''_1,v''_2,\dots ,v''_l),\]
where $p_n$ are the power-sum symmetric polynomials. Hence the above universal series expressions hold for all $\a\in K_\T(S)$.
\end{proof}

\begin{lemma}\label{lem:sym-poly}
Let $F(\vec x,\vec y)$ be a polynomial symmetric in $\vec x=(x_1,\dots, x_n)$ and symmetric in $\vec y=(y_1,\dots, y_m)$, then $F$ can be written as a polynomial expression of symmetric functions in $\{y_1,\dots,y_m\}$ and symmetric functions in $\{x_i+y_j\}_{i=1,\dots,n}^{j=1,\dots,m}$.
\end{lemma}
\begin{proof}
Expand $F$ in the elementary symmetric polynomial basis with respect to the variables $y_1,\dots, y_m$: 
\[F(\vec x,\vec y)=\sum_{\mu\text{ partition}}f_\mu(\vec x)e_\mu(\vec y).\]
An induction on the degree of $F$ shows that $f_\mu(\vec x)$ can be written in the desired form for $\mu\neq (0)$. Thus we may assume $F$ is independent of $y$. Furthermore, since if the statement holds for $F$ and $G$, then it holds for $F+G$ and $F\cdot G$, we can assume $F=e_k(\vec x)$. We have
\[e_k(\{x_i+y_j\}_{i=1,\dots,n}^{j=1,\dots,m})=K\cdot e_k(\vec x)+G(\vec x,\vec y)\]
for some constant $K\in\ZZ$, and $G$ is a polynomial symmetric in $\vec x$ and in $\vec y$, and every monomial term in $G$ contains some $y_j$. Apply the same argument to $G$ and we conclude that $G$ satisfies the claim, therefore so does $F$ by the above equation.
\end{proof}

By the non-equivariant Segre-Verlinde correspondence \cite[Theorem 1.7]{Bojko2} and the relations between the non-equivariant series and equivariant series illustrated in Section \ref{sec:proj-reduction}, we have a weak Segre-Verlinde correspondence as the following corollary. The same argument of the following proof also gives us a weak symmetry in the form of Corollary \ref{cor:sv-sym2d-intro}.

\begin{corollary}\label{cor:SV2d-virtual}
In the setting of Theorem \ref{thm:SV-univ-sieres}, we have the following correspondence
\equa{A_{\mu,\nu,\xi}(q)&=B_{\mu,\nu,\xi}((-1)^Nq).
}
whenever one of $\mu,\nu,\xi$ is $(1)$ and the other two are $(0)$. In particular, the degree 0 part satisfies
\[{\mathcal{S}}_{S,0}(E,\alpha;q)-{\mathcal{V}}_{S,0}(E,\alpha;(-1)^Nq)=\sum_{n=2}^\infty \frac{f_n}{(\lambda_1\lambda_2)^{n-2}}\cdot\left(\int_S c_1(S)\right)^2\cdot q^n\]
for some terms $f_n\in H_\T^{2n-2}(\pt)$ dependent on $\a$ through its rank and Chern classes.
\end{corollary}
\begin{proof}
By the last paragraph of Section \ref{sec:proj-reduction}, the universal series in Theorem \ref{thm:SV-univ-sieres}, when passed to a toric projective surface, must give the Segre-Verlinde correspondence of \cite[Theorem 1.7]{Bojko2} in degree 0. Since the degree 0 terms occur only when one of $\mu,\nu,\xi$ is $(1)$ and the other two are $(0)$, we have
\equa{A_{\mu,\nu,\xi}(q)&=B_{\mu,\nu,\xi}((-1)^Nq)}
in those cases.

Note that when we take $\exp$ of (\ref{eqn:log-nek2d-vir}), the total degree 0 part might come from the product of a negative-degree term and a positive-degree term, but since each term in the integrand is accompanied by a copy $c_1(S)$, we know this difference must be a multiple of $c_1(S)^2$. We also see the $[q^n]$ coefficients are sums of products of at most $n$ such integrals, giving a denominator of $\lambda_1^n\lambda_2^n$, so we are done. 

For illustration, we shall extract this difference, and express it explicitly. This is just a standard computation. For a partition $\mu=(\mu_1,\mu_2,\dots ,\mu_L)$ of size $n$ with length $L$, and a sequence of positive integers $\kappa=(k_1,k_2,\dots ,k_L)$, write $\kappa|\mu$ if each $k_i|\mu_i$. We also associate a set of tuples of partitions to $\kappa$ by
\[
M_{\kappa} := \left\{ \left((\mu^{(i)})_{i=1}^{L},(\nu^{(i)})_{i=1}^{L},(\xi^{(i)})_{i=1}^{L}\right)\ \middle\vert \begin{array}{l}
    \text{$\mu^{(i)},\nu^{(i)},\xi^{(i)}$ are partitions for each $i$, s.t.}\\
    \indent \sum_{i=1}^Lk_i(|\mu^{(i)}|+|\nu^{(i)}|+|\xi^{(i)}|-1)=0\text{ and}\\
    \indent \text{$\mu_i=\nu_i=\xi_i=0$ for some $i$.}
  \end{array}\right\}.
\]
For each $n>0$, we would like to find the degree 0 part of the $[q^n]$ coefficient of $\exp$ of (\ref{eqn:log-nek2d-vir}). By expanding the exponential using definition, we observe that these terms come from products of integrals labeled by $\mu^{(i)},\nu^{(i)},\xi^{(i)}$ in $M_\kappa$ for some tuples $\kappa|\pi$ for some partition $\pi$ of size $n$. A more precise description is given by the following equation. Suppose $\log(A_{\mu,\nu,\xi}(q))=\sum_{i=1}^\infty a_{\mu,\nu,\xi,i}q^i$ and $\log(B_{\mu,\nu,\xi}(q))=\sum_{i=1}^\infty b_{\mu,\nu,\xi,i}q^i$, we have
\equa{&[q^n]({\mathcal{S}}_{S,0}(E,\alpha;q)-{\mathcal{V}}_{S,0}(E,\alpha;(-1)^Nq)\\
=&\sum_{\substack{|\pi|=n\\\ell(\pi)>1}}\sum_{\kappa|\pi}\prod_{(\vec\mu,\vec\nu,\vec\xi)\in M_\kappa}\prod_{i=1}^{\ell(\pi)}\frac{1}{k_i!}\left(a_{\mu^{(i)},\nu^{(i)},\xi^{(i)},\pi_i/k_i}\int_Sc_{\mu^{(i)}}(\a)c_{\nu^{(i)}}(S)c_{\xi^{(i)}}(E)c_1(S)\right)^{k_i}\\
&\indent-(-1)^{Nn}\sum_{\substack{|\pi|=n\\\ell(\pi)>1}}\sum_{\kappa|\pi}\prod_{(\vec\mu,\vec\nu,\vec\xi)\in M_\kappa}\prod_{i=1}^{\ell(\pi)}\frac{1}{k_i!}\left(b_{\mu^{(i)},\nu^{(i)},\xi^{(i)},\pi_i/k_i}\int_Sc_{\mu^{(i)}}(\a)c_{\nu^{(i)}}(S)c_{\xi^{(i)}}(E)c_1(S)\right)^{k_i}\\
=&\sum_{\substack{|\pi|=n\\\ell(\pi)>1}}\sum_{\kappa|\pi}\prod_{(\vec\mu,\vec\nu,\vec\xi)\in M_\kappa}\left(\prod_{i=1}^{\ell(\pi)}a_{\mu^{(i)},\nu^{(i)},\xi^{(i)},\pi_i/k_i}^{k_i|M_k|} -(-1)^{Nn}\prod_{i=1}^{\ell(\pi)}b_{\mu^{(i)},\nu^{(i)},\xi^{(i)},\pi_i/k_i}^{k_i|M_k|}\right)\\
&\cdot \prod_{i=1}^{\ell(\pi)}\frac{c_1(S)^{k_i}}{k_i!\lambda_1^{k_i}\lambda_2^{k_i}}\left(c_{\mu^{(i)}}(\a)c_{\nu^{(i)}}(S)c_{\xi^{(i)}}(E)\right)^{k_i}.
}
In the summation we have $\ell(\pi)>1$ because $M_\kappa$ is empty for any $\kappa|\pi$ if $\pi=(n)$ by definition. Therefore $2\leq \sum_{i=1}^{\ell(\pi)}k_i\leq n$. By multiplying some appropriate power of $\lambda_1\lambda_2$ to the denominator and numerator of the right-hand side, we can express it as a rational function in $\lambda_1,\lambda_2$, with denominator $\lambda_1^n\lambda_2^n$ and numerator a multiple of $c_1(S)^2$. Setting this multiple as $f_n\in\CC[\lambda_1,\lambda_2]$ gives the desired expression.

\end{proof}

\begin{example}
The universal series for ${\mathcal{N}}_S$ are known explicitly in the compact case \cite[Theorem 1.2]{Bojko2}. For a smooth projective surface $S$ and $\alpha$ of rank $r$, we apply \cite[Theorem~17, (16),(17)]{KHilb} for $f(x)=1-ze^x,g(x)=\frac{x}{1-e^{-x}}$ and get
\equa{{\mathcal{N}}_S(\O_S,\a;q,z)=\left[\left(\frac{1-zQ}{1-z}\right)^{r}\left(\frac{-rzQ(1-Q)}{1-zQ}+1\right)\right]^{c_1(S)^2}\left[\frac{1-zQ}{1-z}\right]^{c_1(S)\cdot c_1(\alpha)}}
via the substitution 
\[q=\frac{1-Q^{-1}}{(1-zQ)^r}.\]
As mentioned in the introduction, when $N=1$, we shall set $y_1=1$, and omit the subscript $\xi$ for $H_{\mu,\nu,\xi}$. The formula above allows us to compute $H_{(1),(0)}(q,z)$ and $H_{(0),(1)}(q,z)$.  For a smooth projective surface $S$ and $\alpha$ of rank $0$, we have
\[{\mathcal{N}}_S(\O_S,\a,q,z)=\left(\frac{1 - q - z}{(1 - q) (1 - z)}\right)^{ c_1(S)\cdot c_1(\alpha)}.\]
The exponent is interpreted as intersection product, which in the toric case corresponds to the equivariant push-forward
\[\int_Sc_1(\alpha)c_1(S).\]
Therefore for rank 0, we have the following series from expansion (\ref{eqn:log-nek2d-vir})
\[H_{(1),(0)}(q,z)=\log\frac{1 - q - z}{(1 - q) (1 - z)},\indent H_{(0),(1)}(q,z)=0,\]
Take the Chern limit of Lemma \ref{thm:SV-univ-sieres} by substituting $q\leadsto -q\e$, $z\leadsto (1+\e)^{-1}$ and get
\[C_{(1),(0)}(q)={1+q}\]
Replacing $\a$ by $-\a$ in the Chern series to get the Segre series for $\a$, we see
\[A_{(1),(0)}(q)=C_{(1),(0)}(q)^{-1}=\frac{1}{1+q}\]
On the other hand, the Verlinde limit yields
\[B_{(1),(0)}(q)=\frac{1}{1-q}\]
The Segre-Verlinde correspondence of Corollary \ref{cor:SV2d-virtual} is indeed satisfied.
\end{example}

\begin{example}
Let $S=\CC^2$, $n=N=2$, $E=\O_S\<y_1\>\oplus\O_S\<y_2\>$ and $L=\O_S\<v\>$. The $\T_1$-fixed locus of $\quot_S(E,n)$ is the disjoint union of 
\[\hilb^0(S)\times \hilb^2(S),\indent \hilb^1(S)\times \hilb^1(S), \indent\hilb^2(S)\times\hilb^0(S)\]
Denote $Z_\mu$ the point in $\hilb^i(S)^{\T_0}$ corresponding to a partition $\mu$, then the $\T$-fixed points of $\quot_S(\CC^2,2)$ are
\[(Z_\phi,Z_{(2)}),(Z_\phi,Z_{(1,1)}), (Z_{(1)},Z_{(1)}), (Z_{(2)},Z_\phi),(Z_{(1,1)},Z_\phi)\]
Therefore by (\ref{eqn:tangent2}), the virtual tangent bundles at these five points are respectively 
\equa{&(t_1^2 + t_2-t_1^2t_2  + y_1^{-1}y_2)(1 + t_1^{-1})\\
&(t_2^2 + t_1-t_1t_2^2  + y_1^{-1}y_2)(1 + t_2^{-1})\\
&(t_1+t_2-t_1t_2)(1 + y_1^{-1}y_2)(1+y_1y_2^{-1})\\
&(t_1^2 + t_2-t_1^2t_2  + y_1y_2^{-1})(1 + t_1^{-1})\\
&(t_2^2 + t_1-t_1t_2^2  + y_1y_2^{-1})(1 + t_2^{-1})\\
}
The equivariant Chern roots of $\alpha^{[n]}$ at these points are respectively
\[\{m_2 + w, m_2 - \lambda_1 + w\},\{m_2 + w, m_2 - \lambda_2 + w\},\{m_1+w,m_2 + w\}, \{m_1+w, m_1-\lambda_1 + w\},\{m_1+w, m_1-\lambda_2 + w\}\\
\]
The contribution to the Segre numbers at each of these fixed points are
\[\frac{(2\lam_1 + \lam_1)(\lam_1 + \lam_2)}{2(m_1 - m_2 + \lam_1)(m_1 - m_2)(m_2 - \lam_1 + w_1 + 1)(m_2 + w_1 + 1)(\lam_2 - \lam_1)\lam_1^2\lam_2},\]
\[\frac{(\lam_1 + 2\lam_2)(\lam_1 + \lam_2)}{2(m_1 - m_2 + \lam_2)(m_1 - m_2)(m_2 - \lam_2 + w_1 + 1)(m_2 + w_1 + 1)(\lam_1 - \lam_2)\lam_1\lam_2^2},\]
\[\frac{( \lam_1 + \lam_2+m_1 - m_2)(\lam_1 + \lam_2-m_1 + m_2 )(\lam_1 + \lam_2)^2}{(m_1 - m_2 + \lam_1)(m_1 - m_2 - \lam_1)(m_1 - m_2 + \lam_2)(m_1 - m_2 - \lam_2)(m_1 + w1 + 1)(m_2 + w1 + 1)\lam_1^2\lam_2^2},\]
\[\frac{(2\lam_1 + \lam_1)(\lam_1 + \lam_2)}{2(m_1 - m_2 - \lam_1)(m_1 - m_2)(m_1 - \lam_1 + w_1 + 1)(m_1 + w_1 + 1)(\lam_2 - \lam_1)\lam_1^2\lam_2},\]
\[\frac{(\lam_1 + 2\lam_2)(\lam_1 + \lam_2)}{2(m_1 - m_2 - \lam_2)(m_1 - m_2)(m_1 - \lam_2 + w_1 + 1)(m_1 + w_1 + 1)(\lam_1 - \lam_2)\lam_1\lam_2^2}.\]

Summing them up, we have
\[\scriptstyle[q^2]S^2(\a;q)=\frac{{\begin{pmatrix*}[l]\scriptstyle m_{1} m_{2} \lambda_1 - m_{1} \lambda_1^{2} - m_{2} \lambda_1^{2} + \lambda_1^{3} + m_{1} m_{2} \lambda_2 - 3  m_{1} \lambda_1 \lambda_2 - 3  m_{2} \lambda_1 \lambda_2 + 3  \lambda_1^{2} \lambda_2 - m_{1} \lambda_2^{2} - m_{2} \lambda_2^{2}\\\scriptstyle + 3  \lambda_1 \lambda_2^{2} + \lambda_2^{3} + m_{1} \lambda_1 w + m_{2} \lambda_1 w - 2  \lambda_1^{2} w + m_{1} \lambda_2 w + m_{2} \lambda_2 w - 6  \lambda_1 \lambda_2 w - 2  \lambda_2^{2} w\\\scriptstyle + \lambda_1 w^{2} + \lambda_2 w^{2} + m_{1} \lambda_1 + m_{2} \lambda_1 - 2  \lambda_1^{2} + m_{1} \lambda_2 + m_{2} \lambda_2 - 6  \lambda_1 \lambda_2 - 2  \lambda_2^{2} + 2  \lambda_1 w + 2  \lambda_2 w + \lambda_1 + \lambda_2\end{pmatrix*}} {\left(\lambda_1 + \lambda_2\right)}}{2  {\left(m_{1} - \lambda_1 + w + 1\right)} {\left(m_{1} - \lambda_2 + w + 1\right)} {\left(m_{1} + w + 1\right)} {\left(m_{2} - \lambda_1 + w + 1\right)} {\left(m_{2} - \lambda_2 + w + 1\right)} {\left(m_{2} + w + 1\right)} \lambda_1^{2} \lambda_2^{2}}.\]
A similar computation yields another complicated expression for the Verlinde number. We are interested in the total degree 0 part of their difference in the variables $\vec \lam,\vec m,\vec w$, this computes to
\equa{[q^2](S_{S,0}(E,L;q)-V_{S,0}(E,L;q))&=-\frac{{\left(3  m_{1} m_{2} - \lambda_1\lambda_2 + 3  m_{1} w + 3  m_{2} w + 3  w^{2}\right)} {\left(\lambda_1 + \lambda_2\right)}^{2}}{3  \lambda_1^{2} \lambda_2^{2}}\\
&=\left(\frac13c_2(S)-c_2(E)-c_1(E)c_1(V)-c_1(V)^2\right)\left(\int_Sc_1(S)\right)^2.
}
Note that even though the expressions for Segre and Verlinde numbers are complicated, their difference in degree 0 simplifies tremendously and satisfies Corollary \ref{cor:SV2d-virtual}.
\end{example}

\subsection{Reduced virtual classes and invariants}
As mentioned previously, the obstruction for $\quot_S(E,n)$ at $Z$ contains at least one copy of $K_S^\vee$. For $K$-trivial surfaces, this causes the Euler class of $T^\vir$ to vanish. Therefore the virtual Verlinde and Segre numbers both vanish. One can instead study the ``reduced" versions of these invariants. By \cite[Proposition 9]{lim}, when $S$ is a $K$-trivial surface, $n>0$, and $E$ a torsion-free sheaf, there is a \emph{reduced obstruction theory} that is perfect in the sense of Definition \ref{defn:obs}. The reduced (virtual) tangent bundle in this case is given by adding a trivial summand to the usual virtual tangent bundle:
\[T^\redu=T^\vir+\O_{\quot_S(E,n)}.\]

In this section, we study the equivariant analogue where $S=\CC^2$ with the natural action of the the 1-dimensional torus \equa{\T_0=\CC^*=\{(t_1,t_2):t_1t_2=1\}.}
Write
\[H^*_{\T_0}(\pt)=\CC[\lambda_1,\lambda_2]/(\lambda_1+\lambda_2)=\CC[\lambda],\]
\[K_{\T_0}(\pt)=\ZZ[t_1^{\pm1},t_2^{\pm1}]/(t_1t_2-1)=\ZZ[t^{\pm1}].\]

Using the argument of \cite[Lemma~3.1]{CK1}, we see that the $\T=(\T_0\times\T_1\times\T_2)$-fixed locus of $\quot_S(E,n)$ stays unchanged, and the Zariski tangent space at each of the fixed points has no fixed parts by (\ref{eqn:Tvir-quot2d}). The \emph{equivariant reduced Segre and Verlinde series} $\mathcal{S}^\redu_S$ and $\mathcal{V}^\redu_S$ are defined the same as the virtual ones by replacing $T^\vir$ with $T^\redu$. Here we omit the subscript since we are only interested in the $S=\CC^2$ case.
\equa{{\mathcal{S}}^{\redu}(E,\alpha;q):=&\sum_{n>0}^\infty q^n\sum_{Z\in\quot_S(E,n)^\T}\frac{c(\alpha^{[n]}|_{Z})}{e(T_{Z}^\redu)},\\
{\mathcal{V}}^{\redu}(E,\alpha;q):=&\sum_{n>0}^\infty q^n\sum_{Z\in\quot_S(E,n)^\T}\frac{\ch(\det(\alpha^{[n]}|_{Z}))}{\ch(\Lambda_{-1}(T_{Z}^\redu)^\vee)}.}
Note that we do not include the $n=0$ term because condition 2 of \cite[Proposition 9]{lim} is only satisfied when $n>0$.

The same strategy from the previous section can be applied to study these invariants. For $E=\oplus_{i=1}^N\O_S\<y_i\>$ and $V=\oplus_{i=1}^r\O_S\<v_i\>$, define
\equa{{\mathcal{N}}^{\redu}(E,V;q,z):=&\sum_{\mu\neq(0)} q^{|\mu|}\prod_{\square\in\mu}
\frac{\prod_{j=1}^N\prod_{i=1}^r(1-t^{-c(\square)+r(\square)}v_iy_jz)}{\ch(\Lambda_{-1}(T^\redu|_Z)^\vee)}.
}
Again note that the $[q^0]$ coefficient is 0. We can think of the reduced obstruction as removing a copy of $K_S^\vee$ from the usual obstruction in $\ZZ[t_1^{\pm1},t_2^{\pm1}]$, then passing to the quotient ring $\ZZ[t_1^{\pm1},t_2^{\pm1}]/(t_1t_2-1)$. This gives us the following corollary.
\begin{corollary}
For $n>0$, the $[q^n]$ coefficient of $\mathcal{N}^\redu$ can be obtained from the non-reduced version by taking the following limit:
\equa{[q^n]{\mathcal{N}}^{\redu}(E,V;q,z)&=[q^n]\frac{{\mathcal{N}}_S(E,V;q,z)}{1-e^{-c_1(K_S^\vee)}}\bigg\vert_{-\lambda_2\->\lam_1=\lambda}\\
&=[q^n]\lim_{-\lam_2\->\lam_1=\lam}\frac{{\mathcal{N}}_S(E,V;q,z)}{\lambda_1+\lambda_2}.
}
\end{corollary}

Expand the universal series expression from Theorem \ref{thm:SV-univ-sieres} and obtain
\equa{{\mathcal{S}}_S(E,\a;q)=&\sum_{i=1}^\infty \frac{1}{i!}(\lambda_1+\lambda_2)^i\left(\sum_{\mu,\nu,\xi} \log A_{\mu,\nu,\xi}(q)\cdot  \int_Sc_\mu(\a)c_\nu(S)c_{\xi}(E)\right)^i.
}
Using the above corollary to extract the reduced coefficients, we see the terms with $i>1$ all vanish and
\[{\mathcal{S}}^{\redu}(E,\a;q)=\sum_{\mu,\nu,\xi} \log A_{\mu,\nu,\xi}(q)\cdot  \int_Sc_\mu(\a)c_\nu(S)c_{\xi}(E).\]
The Chern and Verlinde cases are similar, thus we have the following result.

\begin{theorem}\label{cor:reduced-series}
When $S=\CC^2$, the equivariant reduced Segre and Verlinde series for $E=\oplus_{i=1}^N\O_S\<y_i\>$ and $\a\in K_\T(S)$ are
\equa{
{\mathcal{S}}^{\redu}(E,\a;q)=&\sum_{\mu,\nu,\xi}\log\left(A_{\mu,\nu,\xi}(q)\right)\cdot \int_S c_\mu(\a)c_\nu(S)c_\xi(E),\\
{\mathcal{V}}^{\redu}(E,\a;q)=&\sum_{\mu,\nu,\xi}\log\left(B_{\mu,\nu,\xi}(q)\right)\cdot \int_S c_\mu(\a)c_\nu(S)c_\xi(E),\\
{\mathcal{C}}^{\redu}(E,\a;q)=&\sum_{\mu,\nu,\xi}\log\left(C_{\mu,\nu,\xi}(q)\right)\cdot \int_S c_\mu(\a)c_\nu(S)c_\xi(E)\\
}
where $A_{\mu,\nu,\xi}, B_{\mu,\nu,\xi}$ and $C_{\mu,\nu,\xi}$ are the same series from Theorem \ref{thm:SV-univ-sieres}.
\end{theorem}
The integrals in the above theorem labeled by $\mu,\nu,\xi$ have degree $|\mu|+|\nu|+|\xi|-2$, which is one degree lower than the integrals in the non-reduced expressions. Therefore we have a Segre-Verlinde correspondence in degree $-1$ for the reduced setting. However, results for degree $-1$ have no compact analogues since they automatically vanish in the compact setting. In the Section \ref{sec:reduced-computation}, we compute some of the universal series explicitly, giving us some Segre-Verlinde relations in non-negative degrees for the reduced case.

\section{Explicit computations of universal series}

\subsection{Virtual Segre number in rank \texorpdfstring{$r=-1$}{Lg}} \label{sec:chern-rank-2}
We shall compute the explicit expression for the Chern series at rank $r=1$ for Hilbert schemes. To begin, we consider the non-virtual Chern series for rank $r=2$, which is known to have universal series structure as a result of \cite[(2.6)]{GM}. By definition, the non-virtual Chern number $[q^n]I^{\mathcal{C}}(V;q)$ of any rank 2 equivariant bundle $V$ has total degree $-2n$ to $0$ in $\lambda_1,\lambda_2$. Comparing coefficients, we see the positive degree terms must vanish. By a non-virtual version of the argument of Section \ref{sec:proj-reduction}, the degree 0 terms are given by taking the Chern limit \cite[Proposition 3.4]{GM} of the expressions \cite[(3.16)]{GM}.

The following lemma, based on a similar statement mentioned in \cite[§4.6]{GM}, allows us to compute the remaining negative degree terms using these degree 0 terms. At last, we obtain
\equanum{\label{eqn:chern-rank-2}I^{\mathcal{C}}(V;q)=&\exp\left(\log(1+q)\int_Sc_0(V)+\log(1+q)\int_Sc_1(V)+\log(1+q)\int_Sc_2(V)\right)\\
=&(1+q)^{\int_Sc(V)}}
for any $\T$-equivariant bundle $V$ of rank 2.

\begin{lemma}\label{lem:differential}
Let $r\geq 0$ and suppose $H(z_1,\dots,z_r)$ is a power series in $z_1,\dots z_r$ whose coefficients are series in some other variables $q_1,q_2,\dots$. If $H$ is symmetric in $z_1,\dots,z_r$ with expansion
\[H(ze^{x_1},\dots,ze^{x_r})=\sum_{\substack{\mu\text{ partition}}}H_{\mu}(z)\prod_{i=1}^{\ell (\mu)}e_{\mu_i}(x_1,\dots,x_r),\]
then for any $k\geq 0$, we have
\[D_z^kH_{(0)}(z)=k!\sum_{|\mu|=k}\binom{r}{\mu}H_{\mu}(z)\]
where $\binom{r}{\mu}$ denotes $\prod_{i=1}^{\ell(\mu)}\binom{r}{\mu_i}$, and $D_z=z\frac{\partial}{\partial z}$.
\end{lemma}
\begin{proof}
We begin by claiming that the statement is closed under polynomial expressions; that is, if the equality holds for both $F(z_1,\dots,z_r)$ and $G(z_1,\dots,z_r)$, then it holds for $F\cdot G$ and  $aF+G$ for any $a\in\QQ[\![q_1,q_2,\dots]\!]$. The additive part is straightforward, and we shall prove the multiplicative part of this claim. Expand 
\equa{F(ze^{x_1},\dots ,ze^{x_r})=\sum_{\mu\text{ partition}}F_\mu(z) e_\mu(x_1,\dots,x_r),\\
G(ze^{x_1},\dots ,ze^{x_r})=\sum_{\mu\text{ partition}}G_\mu(z) e_\mu(x_1,\dots,x_r),
}
then $H=F \cdot G$ can be expanded as
\[H(ze^{x_1},\dots ,ze^{x_r})=\sum_{\mu\text{ partition}}H_\mu(z) e_\mu(x_1,\dots,x_r)=\sum_{\nu+\xi=\mu}F_\nu(z) G_\xi(z) e_\mu(x_1,\dots,x_r)\]
where by $\nu+\xi$ we mean combining them as sequences to get a partition of size $|\nu|+|\xi|$ with length $\ell(\nu)+\ell(\xi)$. Suppose the statement holds for both $F$ and $G$, then we have
\equa{D_z^kH_{(0)}(z)&=D_z^k\left(F_{(0)}(z)G_{(0)}(z)\right)\\
&=\sum_{i=1}^k\binom{k}{i}D_z^iF_{(0)}D_z^{k-i}G_{(0)}\\
&=\sum_{i=1}^k\binom{k}{i}i!(k-i)!\sum_{|\nu|=i,|\xi|=k-i}\binom{r}{\nu}F_{\nu}\binom{r}{\xi}G_{\xi}\\
&=k!\sum_{|\mu|=k}\binom{r}{\mu}H_{\mu}.
}

Since $H(z_1,\dots ,z_r)$ is symmetric, by the above observation, it suffices to prove the statement when $H$ is the power sum symmetric polynomial $p_n(z_1,\dots ,z_r)=z_1^n+\dots +z_r^n$. For each $n\geq 0$, we expand
\equa{H(ze^{x_1},\dots ,ze^{x_r})=p_n(ze^{x_1},\dots ,ze^{x_r})&=\sum_{j=1}^rz^ne^{nx_j}=z^n\left(r+\sum_{i>0}\frac{n^i}{i!}p_i(x_1,\dots ,x_r)\right).}
This means $H_{(0)}(z)=rz^n$ and 
\[D_z^kH_{(0)}(z)=rn^kz^n.\]

Fix $r\geq 1$, we write
\[p_n(x_1,\dots ,x_r)=\sum_{|\mu|=n}C_\mu e_\mu(x_1,\dots ,x_r)\]
for some constant terms $C_\mu$. Evaluate at $x_1=\dots=x_r=1$ and get
\[r=\sum_{|\mu|=n}\binom{r}{\mu}C_\mu.\]
Hence
\[k!\sum_{|\mu|=k}\binom{r}{\mu}H_\mu(z)=z^nn^k\sum_{|\mu|=k}\binom{r}{\mu}C_\mu=rn^kz^n=D_z^kH_{(0)}(z).\]
A quick calculation for the $k=0$ or $r=0$ cases finishes the proof.

\end{proof}

Now we proceed with the virtual case. Recall that on Hilbert schemes, the obstruction theory at a fixed point $[Z_\mu]$ is given by $(K^{[n]}_S)^\vee|_{Z_\mu}$, so
\[\frac{1}{e(T^\vir_{Z_\mu})}=\frac{e((K_S^{[n]})^\vee|_{Z_\mu})}{e(T_{Z_\mu})} =(-1)^{|\mu|}\frac{e((K_S^{[n]})|_{Z_\mu})}{e(T_{Z_\mu})}.\] 
Let $L=\O_S\<v_1\>$ be an equivariant line bundle, and $V=L\oplus\O_S\<v_2\>$. Apply (\ref{eqn:chern-rank-2}) to $V$ and we have
\[I^{\mathcal{C}}(V;q)=(1+q)^{\int_Sc(V)}=(1+q)^{\int_S(1+w_1+w_2+w_1w_2)}.\]
Set $w_2=c_1(K_S)-1$ and replace $q$ by $-q$, then this becomes
\[I^{\mathcal{C}}(V;-q)|_{w_2=c_1(K_S)-1}=(1-q)^{\int_Sc(L)c_1(K_S)}.\]
On the other hand, we have by definition
\equa{&I^{\mathcal{C}}(V;-q)|_{w_2=c_1(K_S)-1}\\
=&\sum_{\mu} (-1)^{|\mu|}q^{|\mu|}\prod_{\square\in\mu}\frac{(1+w_1-c(\square)\lam_1-r(\square)\lam_2)(c_1(K_S)-c(\square)\lam_1-r(\square)\lam_2)}{((a(\square)+1)\lambda_1-l(\square)\lambda_2)((l(\square)+1)\lambda_2-a(\square)\lambda_1)}\\
=&\sum_{\mu}(-1)^{|\mu|}q^{|\mu|}\frac{c(L^{[n]}|_{Z_\mu})e((K_S^{[n]})|_{Z_\mu})}{e(T|_{Z_\mu})}\\
=&\sum_{\mu}q^{|\mu|}\frac{c(L^{[n]}|_{Z_\mu})}{e(T^\vir|_{Z_\mu})}=\mathcal{C}_S(\O_S,L;q)
}
Therefore we conclude
\equa{{\mathcal{C}}_S(\O_S,L;w)=(1-q)^{\int_Sc(L)c_1(K_S)}=\left(\frac1{1-q}\right)^{\int_Sc(L)c_1(S)}.}
In particular, restricting to the lowest degree part in the variables $\lam_1,\lam_2,w_1$, we obtain the following Corollary.
\begin{corollary}\label{cor:nek2d}For $S=\CC^2$, the following equality holds
\[\sum_{n=1}^\infty q^n\int_{[\hilb^n(S)]^\vir}1:=\sum_{Z\in\hilb^n(S)^\T}\frac{1}{e(T^\vir_Z)}=e^{\frac{(\lambda_1+\lambda_2)}{\lambda_1\lambda_2}q}.\]
\end{corollary}

\subsection{Segre-Verlinde correspondence in non-zero degrees}\label{sec:reduced-computation} Recall that we use the notation $H_{\mu,\nu,\xi}$ for the series from (\ref{eqn:log-nek2d-vir}) describing the virtual Nekrasov genus for  Quot schemes on $S=\CC^2$. Note that $\mathcal{N}_S$ satisfies
\equa{{\mathcal{N}}_{S}(y_1,\dots,y_N;v_1,\dots,v_r;q,z)&={\mathcal{N}}_{S}(y_1,\dots,y_N;ze^{w_1},\dots,ze^{w_r};q,1)\\
&={\mathcal{N}}_{S}(ze^{m_1},\dots,ze^{m_N};v_1,\dots,v_r;q,1).
}
Applying Lemma \ref{lem:differential} to $H_{\mu,\nu,\xi}$ in the variables $w_1,\dots, w_r$ gives us for $r>0$, 
\[D_z^{k}H_{(0),\nu,\xi}(q,z)=rD_z^{k-1}H_{(1),\nu,\xi}(q,z)=k!\sum_{|\mu|=k}\binom{r}{\mu}H_{\mu,\nu,\xi}(q,z),\]
and applying in the variables $m_1,\dots,m_N$ yields
\equanum{\label{eqn:xi}D_z^{k}H_{\mu,\nu,(0)}(q,z)=ND_z^{k-1}H_{\mu,\nu,(1)}(q,z)=k!\sum_{|\xi|=k}\binom{N}{\xi}H_{\mu,\nu,\xi}(q,z).}

When the rank $r$ is negative, we consider $\a=-[V]$ where $V=\oplus_{i=1}^{-r}\O_S\<v_i\>$. Write 
\[\log {\mathcal{N}}_S(E,-[V];q,z)=\sum_{\mu,\nu, \xi\text{ partitions}}H_{\mu,\nu,\xi}(q,z)\cdot \int_Sc_\mu(V) c_\nu(S)  c_\xi(E) c_1(S).\]
Then the same argument with Lemma \ref{lem:differential} applies, and for all $r\neq 0$,
\equanum{\label{eqn:mu}D_z^{k}H_{(0),\nu,\xi}(q,z)=|r|D_z^{k-1}H_{(1),\nu,\xi}(q,z)=k!\sum_{|\mu|=k}\binom{|r|}{\mu}H_{\mu,\nu,\xi}(q,z).}
Observe that when $r>0$, the Chern limit of $H_{\mu,\nu,\xi}$ returns the universal series for the Chern invariants $\log C_{\mu,\nu,\xi}$, and the Verlinde limit returns $\log B_{\mu,\nu,\xi}$. On the other hand, when $r<0$, the Chern series retrieves the rank $-r$ Segre invariants $\log A_{\mu,\nu,\xi}$, but the Verlinde limit does not give the Verlinde series. This is because $H_{\mu,\nu,\xi}$ is associated to $c_\mu(V)$, whereas $\log B_{\mu,\nu,\xi}$ is associated to $c_\mu(\a)=c_\mu(-[V])$. This results in a change of basis for the symmetric series in the Chern roots of $V$, and Lemma \ref{lem:differential} no longer applies for these negative ranks.

Our goal for this section is to apply Chern and Verlinde limits to (\ref{eqn:xi}) and (\ref{eqn:mu}) for $k>0$, then obtain relations for the Chern and Verlinde series for various $\mu$, $\nu$, and $\xi$. We may obtain explicit expressions for when $|\mu|+|\nu|+|\xi|=1$ from the compact setting. For a smooth projective surface $S'$, a torsion-free sheaf $E'$ of rank $N$, and a K-theory class $\a$ of rank $r$, the universal series structure for $\mathcal{N}_S$ can be obtained from \cite[(5.1), Theorem 5.1]{Bojko2} by setting 
\[f(t)=1-ze^t,\indent g(t)=\frac{t}{1-e^{-t}}.\]
This gives us
\[\mathcal{N}_S(E,\a;q,z)=\left(\prod_{i=1}^NF(H_i)\right)^{c_1(S)c_1(\a)}\left(\prod_{i=1}^NF(H_i)\right)^{\frac{r}{N}c_1(S)c_1(E)}G(R)^{c_1(S)^2},\]
where 
\[R=f^rg^N,\indent F=\frac{f}{f(0)},\]
the series $H_i(q)$ are Newton–Puiseux solutions to 
\[H_i^N=q R(H_i),\]
and $G(R)$ is given explicitly by \cite[(4.24)]{Bojko2}. Therefore 
\equanum{\label{eqn:general-N}H_{(1),(0),(0)}(q,z)&=\frac{r}{|r|}\sum_{i=1}^N\log F(H_i),\\
H_{(0),(1),(0)}(q,z)&=\log G(R),\indent H_{(0),(0),(1)}(q,z)=\frac{r}{N}\sum_{i=1}^N\log F(H_i).
}
The sign $\frac{r}{|r|}$ in the first line appears as a result of $c_1(-V)=-c_1(V)$.

\subsubsection{The Chern limit}\label{sec:chern-limit}
We begin with the case $\nu=\xi=(0)$. Apply \cite[(4.17)]{Bojko2}, which gives a term-by-term expression for (\ref{eqn:general-N}), by setting $f(t)=1-ze^t$ and $g(t)=\frac{t}{1-e^{-t}}$. We have
\equanum{\label{eqn:100}H_{(1),(0),(0)}(q,z)&=\frac{r}{|r|}\sum_{n=1}^\infty\frac{1}{n}q^n[t^{nN-1}]\left\{-ze^t(1-ze^t)^{nr-1}\left(\frac{t}{1-e^{-t}}\right)^{nN}\right\}}
By (\ref{eqn:xi}) and (\ref{eqn:mu}), we have for $k_1,k_2\geq 0$ and $k:=k_1+k_2>0$,
\equanum{\label{eqn:H-sum}&k_1!k_2!\sum_{|\mu|=k_1}\sum_{|\xi|=k_2}\binom{|r|}{\mu}\binom{N}{\xi}H_{\mu,(0),\xi}\\
=&|r|D^{k-1}_zH_{(1),(0),(0)}\\
=&r\sum_{n=1}^\infty\frac{1}{n}q^n[t^{nN-1}]\left\{D^{k-1}_z\left(-ze^t(1-ze^t)^{nr-1}\right)\left(\frac{t}{1-e^{-t}}\right)^{nN}\right\}\\
=&r\sum_{n=1}^\infty\frac{1}{n}q^n[t^{nN-1}]\left\{(-1)^k(1-ze^t)^{nr-k}p_{n,k}(ze^t)\left(\frac{t}{1-e^{-t}}\right)^{nN}\right\}
}
where $p_{n,k}$ is a polynomial of degree $k$. We may show inductively that $p_{n,k}(1)=(nr-1)_{(k-1)}$. With this expansion, we would like to apply the Chern limit of Lemma \ref{lem:CV-limit}. In the expansion of universal series (\ref{eqn:log-nek2d-vir}), the series $H_{\mu,(0),(0)}$ are multiplied by a term in $\vec{\lam},\vec{w},\vec{m}$ of homogeneous degree $|\mu|-1=k-1$. When taking the Chern limit, we need to make substitutions 
\[\vec\lam\leadsto-\e\vec\lambda,\indent \vec w\leadsto -\e \vec w,\indent \vec m\leadsto-\e\vec m.\]
Therefore we need to multiply by a factor of $(-\e)^{k-1}$ when taking the limit of $H_{\mu,(0),(0)}$. Furthermore, we substitute $q\leadsto (-1)^Nq\e^{N-r}(1+\e)^r$ and $z\leadsto (1+\e)^{-1}$, and the right-hand side of the above expansion becomes
\[r\sum_{n=1}^\infty\frac{1}{n}q^n[t^{nN-1}]\left\{(-1)^{nN-1}\e^{n(N-r)+k-1}(1+\e-e^t)^{nr-k} (1+\e)^{k}p_{n,k}\left(\frac{e^t}{1+\e}\right)\left(\frac{t}{1-e^{-t}}\right)^{nN}\right\}.\]
Since we are extracting the $[t^{nN-1}]$ coefficient of the function inside the curly bracket, we may replace $t\leadsto -\e t$ and divide the function by $(-\e)^{nN-1}$, which gives us
\equa{r\sum_{n=1}^\infty\frac{1}{n}q^n[t^{nN-1}]\left\{\e^{k-nr}(1+\e-e^{-\e t})^{nr-k} (1+\e)^k p_{n,k}\left(\frac{e^{-\e t}}{1+\e}\right)\cdot \left(\frac{-\e t}{1-e^{\e t}}\right)^{nN}\right\}.
}
Taking $\e\->0$, the term $p_{n,k}\left(\frac{\exp(-\e t)}{1+\e}\right)$ converges to $p_{n,k}(1)=(nr-1)_{(k-1)}$. Thus the limit gives us \equa{r\sum_{n=1}^\infty\frac{1}{n}q^n[t^{nN-1}]\left\{(nr-1)_{(k-1)}(1+t)^{nr-k}\right\}=r\sum_{n=1}^\infty\frac{(nr-1)_{(k-1)}}{n}\binom{nr-k}{nN-1}q^n.
}
We further apply the following identity
\equa{(nr-1)_{(k-1)}\binom{nr-k}{nN-1}=(n(r-N))_{(k-1)}\binom{nr-1}{nN-1}.}
The Chern limit of the left-hand side of (\ref{eqn:H-sum}) will be the Chern series of rank $r$ for $r>0$, and the Segre series of rank $-r$ for $r<0$. Therefore for all $r>0$,
\equanum{\label{eqn:segre-mu} k_1!k_2!\sum_{|\mu|=k_1}\sum_{|\xi|=k_2}\binom{r}{\mu}\binom{N}{\xi}\log A_{\mu,(0),\xi}(q)&=-r\sum_{n=1}^\infty\frac{(-n(r+N))_{(k-1)}}{n}\binom{-nr-1}{nN-1}q^n\\
&=-r\sum_{n=1}^\infty\frac{1}{n}q^n[t^{nN-1}]\left\{(-n(r+N))_{(k-1)}(1+t)^{-nr-1}\right\},\\
}
\equa{k_1!k_2!\sum_{|\mu|=k_1}\sum_{|\xi|=k_2}\binom{r}{\mu}\binom{N}{\xi}\log C_{\mu,(0),\xi}(q)&=r\sum_{n=1}^\infty\frac{(n(r-N))_{(k-1)}}{n}\binom{nr-1}{nN-1}q^n.
}

\subsubsection{The Verlinde limit}\label{sec:verlinde-limit}
We apply a similar argument for the Verlinde limit using (\ref{eqn:general-N}). To simplify computation, we consider a different change of variable. Let
\[\tilde{H}_i=1-e^{-H_i}.\]
Then $\tilde{H}_i$ are the Newton-Puiseux solutions to
\[\tilde{H}_i^N=q\frac{(1-\tilde{H}_i-z)^r}{(1-\tilde{H}_i)^r}.\]
Also,
\[F(H_i)=\frac{1-ze^{H_i}}{1-z}=\frac{1-\tilde{H}_i-z}{(1-\tilde{H}_i)(1-z)},\] 
so
\[H_{(1),(0),(0)}(q,z)=\frac{r}{|r|}\sum_{i=1}^N\log F(H_i)=\frac{r}{|r|}\sum_{i=1}^N\log\frac{1-\tilde{H}_i-z}{(1-\tilde{H}_i)(1-z)}.\]
Apply Lagrange inversion theorem \cite[Corollary 2]{bojko-lag} with $\f(t)=\log((1-t-z)/(1-t)(1-z))$ and $R=(1-t-z)^r/(1-t)^{r}$, we have
\equa{&k_1!k_2!\sum_{|\mu|=k_1}\sum_{|\xi|=k_2}\binom{|r|}{\mu}\binom{N}{\xi}H_{\mu,(0),\xi}\\
=&|r|D^{k-1}_zH_{(1),(0),(0)}\\
=&rD^{k-1}_z\sum_{n=1}^\infty\frac{1}{n}q^n[t^{nN-1}]\left\{\f'(t)R(t)^n\right\}\\
=&r\sum_{n=1}^\infty\frac{1}{n}q^n[t^{nN-1}]\left\{D^{k-1}_z\left(-z(1-t-z)^{nr-1}\right)\left(1-t\right)^{-nr-1}\right\}\\
=&r\sum_{n=1}^\infty\frac{1}{n}q^n[t^{nN-1}]\left\{(-1)^k(1-t-z)^{nr-k}q_{n,k}(z)(1-t)^{-nr-1}\right\}.
}
Here $q_{n,k}(z)$ is a polynomial of degree $k$ in $z$ whose coefficients involve the variable $t$, and its leading coefficient is $(nr)^{k-1}$. To take the Verlinde limit for $r>0$, we substitute 
\[q\leadsto (-1)^rq\e^r,\indent z\leadsto \e^{-1},\] 
then take $\e\->0$ and get
\equa{k_1!k_2!\sum_{|\mu|=k_1}\sum_{|\xi|=k_2}\binom{r}{\mu}\binom{N}{\xi}\log  B_{\mu,(0),\xi}&=r\sum_{n=1}^\infty\frac{1}{n}q^n[t^{nN-1}]\left\{(nr)^{k-1}(1-t)^{-nr-1}\right\}.
}
Observe that the Verlinde series for $\a\in K_\T(S)$ only depends on $c_1(\a)$ by definition, so the universal series are non-trivial only when $\mu=(1)_k:=(1,\dots, 1)$ has $k$ copies of 1. We can therefore simplify the left-hand side of the above equation and get
\equanum{\label{eqn:verlinde-mu}k_1!k_2!r^{k_1}\sum_{|\xi|=k_2}\binom{N}{\xi}\log B_{(1)_{k_1},(0),\xi}(q)&=r\sum_{n=1}^\infty\frac{1}{n}q^n[t^{nN-1}]\left\{(nr)^{k-1}(1-t)^{-nr-1}\right\}\\
&=r^k\sum_{n=1}^\infty (-1)^{nN-1}n^{k-2}\binom{-nr-1}{nN-1}q^n.
}

By \cite[Lemma 3.3]{GM}, the universal series are polynomials in $r$ for $r\geq 0$. Therefore the above results hold for the rank $r=0$ case as well.

\subsubsection{The degree $-1$ case}
When $k=0$, the argument from above still applies, where $D_z^{-1}(-)$ denotes taking the anti-derivative of $\frac1z(-)$ with respect to $z$. However this would result in an undetermined constant term from the integration, so we deal with this case separately. We compute $A_{(0),(0),(0)},B_{(0),(0),(0)},C_{(0),(0),(0)}$ using expressions for $H_{(0),(0),(0)}$, which we shall denote as $A,B,C$ and $H$ respectively. Since these series are associated to the part of their respective invariants independent of the weights of $\a$, we have
\[A^{N,r}_{(0),(0),(0)}=C^{N,-r}_{(0),(0),(0)},\]
and by the second part of \cite[Lemma 3.3]{GM}, $A,B,C,H$ are polynomials with respect to $r$ for all $r\in\ZZ$.

By Lemma \ref{lem:differential}, $D_zH=|r|H_{(1),(0),(0)}$. Taking $D_z^{-1}$ of (\ref{eqn:100}) with respect to $z$, we get
\equa{H(q,z)&=H_0(q)+\sum_{n=1}^\infty\frac{1}{n^2}q^n[t^{nN-1}]\left\{(1-ze^t)^{nr}\left(\frac{t}{1-e^{-t}}\right)^{nN}\right\}\\
&=:H_0(q)+H_1(q,z)}
for some $H_0(q)$ independent of the variable $z$.

Apply the result of Section \ref{sec:chern-limit} with $k=0$. We see that $H_1(q)$ admits the following Chern limit
\[\sum_{n=1}^\infty\frac{1}{n^2}\binom{nr}{nN-1}q^n.\]
Write $H_0(q)=\sum_{n=1}^\infty h_nq^n$, then its Chern limit is
\[\lim_{\e\->0}\sum_{n=1}^\infty h_n(-1)^{nN-1}\e^{n(N-r)-1}(1+\e)^{rn}q^n.\]
When $N-r\leq 0$, we must have $h_n=0$ for all $n$ since otherwise we would have negative powers on $\epsilon$ and the limit does not make sense. As each $h_n$ is polynomial in $r$, we conclude that $H_0(q)=0$. Hence for all $r\in\ZZ$
\[\log A(q)=\sum_{n=1}^\infty\frac{1}{n^2}\binom{-nr}{nN-1}q^n,\]
\[\log C(q)=\sum_{n=1}^\infty\frac{1}{n^2}\binom{nr}{nN-1}q^n.\]
When $0\leq r\leq N-1$, the formula for $C(q)$ is consistent with Conjecture \ref{con:cao-kool-quot} and Conjecture \ref{con:low-rank-vanish}. 

Similarly by Section \ref{sec:verlinde-limit}, the Verlinde limit of $H_{(0),(0),(0)}$ is
\[\log B(q)=\sum_{n=1}^\infty\frac1{n^2} (-1)^{nN-1}\binom{-nr-1}{nN-1}q^n.\]

Combining the results of the above sections, we have the following theorem.

\begin{theorem}\label{thm:SV2d-deg-pos}
For rank $r\geq 0$ and integers $k_1,k_2\geq 0$ with $k:=k_1+k_2$, the universal series of Theorem \ref{thm:SV-univ-sieres} satisfy
\equa{k_1!k_2!\sum_{|\mu|=k_1}\sum_{|\xi|=k_2}\binom{r}{\mu}\binom{N}{\xi}\log A_{\mu,(0),\xi}(q)&=-r\sum_{n=1}^\infty\frac{(-n(r+N))_{(k-1)}}{n}\binom{-nr-1}{nN-1}q^n,\\
k_1!k_2!r^{k_1}\sum_{|\xi|=k_2}\binom{N}{\xi}\log B_{(1)_{k_1},(0),\xi}(q)&=-r^{k}\sum_{n=1}^\infty n^{k-2}\binom{-nr-1}{nN-1}\left((-1)^Nq\right)^n.}
Furthermore, we have
\equa{
&k_1!k_2!\sum_{|\mu|=k_1}\sum_{|\xi|=k_2}\binom{r}{\mu}\binom{N}{\xi}\log C_{\mu,(0),\xi}(q)=r\sum_{n=1}^\infty\frac{(n(r-N))_{(k-1)}}{n}\binom{nr-1}{nN-1}q^n,
}
which can be compared to the identities above by replacing $r$ with $-r$.
\end{theorem}

When $k=2$, we have the following Segre-Verlinde correspondences in degree $1$.
\begin{corollary}\label{cor:sv2d-deg1}
For rank $r\geq 0$, the universal series of Theorem \ref{thm:SV-univ-sieres} satisfy the following correspondences
\equa{&\indent A_{(1,1),(0),(0)}(q)^{-r}A_{(2),(0),(0)}(q)^{\frac{-(r-1)}{2}}=B_{(1,1),(0),(0)}\left((-1)^Nq\right)^{r+N},\\
&\indent A_{(1),(0),(1)}(q)^{-r}=B_{(1),(0),(1)}\left((-1)^Nq\right)^{r+N},\text{ and }\\
&\indent A_{(0),(0),(1,1)}(q)^{-rN}A_{(0),(0),(2)}(q)^{\frac{-r(N-1)}{2}}\\
&=B_{(0),(0),(1,1)}\left((-1)^Nq\right)^{N(r+N)}B_{(0),(0),(2)}\left((-1)^Nq\right)^{\frac{(N-1)(r+N)}{2}}.
}
\end{corollary}

\begin{remark}
As mentioned in the introduction, combining Theorem \ref{thm:SV2d-deg-pos} with Theorem \ref{cor:reduced-series} yields the corresponding relations for reduced invariants. In particular, Corollary \ref{cor:sv2d-deg1} implies a correspondence in degree 0 for reduced invariants, which could provide insight into the reduced invariants for K3 surfaces in the compact setting.
\end{remark}

\subsection{Universal series via Lagrange inversion}\label{sec:closed-form}
Let $r>0$. Consider (\ref{eqn:segre-mu}) and (\ref{eqn:verlinde-mu}) from the previous section. The right-hand sides of these identities are linear combinations of series of form
\[\sum_{n=1}^\infty\frac{1}{n}q^n[t^{nN-1}]\left\{n^{k-1}(1+t)^{-nr}\f'\right\}.\]
for $\f(t)=\log(1+t)$ and $k>0$. The goal of this section is to express these series without the process of extracting coefficients. 

By Lagrange inversion theorem \cite[Corollary 2]{bojko-lag}, we have
\equa{&\sum_{n=1}^\infty\frac{1}{n}q^n[t^{nN-1}]\left\{n^{k-1}(1+t)^{-nr}\f'\right\}\\
=&
D_q^{k-1}\sum_{n=1}^\infty\frac{1}{n}q^n[t^{nN-1}]\left\{(1+t)^{-nr}\f'\right\}\\
=&D_q^{k-1}\sum_{i=1}^N\left(\f(H_i)-\f(0)\right)=D_q^{k-1}\sum_{i=1}^N\f(H_i)
}
where $D_q=q\frac{d}{d q}$ and $H_i$ are the Newton-Puiseux solutions to
\equanum{\label{eqn:newton}H_i^N=q(1+H_i)^{-r}.}
Note that 
\[D_q\f(H_i)=q\frac{d}{dq}\f(H_i)=q\f'(H_i)\cdot\frac{dH_i}{dq}=\f'(H_i)\cdot D_qH_i.\]
Here $H_i(q^{\frac1N})$ is a Puiseux series, and by $dH_i/dq$ we mean $(dq/dH_i)^{-1}$. Differentiating both sides of (\ref{eqn:newton}) with respect to $H_i$ yields
\equa{NH_i^{N-1}\frac{dH_i}{dq}&=(1+H_i)^{-r}-rq(1+H_i)^{-r-1}\frac{dH_i}{dq}\\
NqH_i^{-1}\frac{dH_i}{dq}&=1-rq(1+H_i)^{-1}\frac{dH_i}{dq}\\
D_qH_i&=\frac{1}{NH_i^{-1}+r(1+H_i)^{-1}}
}
Let $\psi(t)=Nt^{-1}+r(1+t)^{-1}$, then $D_qH_i=\frac1{\psi(H_i)}$.
Define $D_\psi=\frac1\psi\cdot \frac{d}{dt}$. We conclude
\[D_q\f(H_i)=(D_\psi \f)(H_i)\]
for arbitrary power series $\f(t)$. Therefore
\equa{\sum_{n=1}^\infty\frac{1}{n}q^n[t^{nN-1}]\left\{n^{k-1}(1+t)^{-nr}\f'\right\}=\sum_{i=1}^N(D_\psi^{k-1}\f)(H_i).}
The following corollary follows directly by applying this to (\ref{eqn:segre-mu}) and (\ref{eqn:verlinde-mu}).

\begin{theorem}\label{cor:SV2d-closed}
Let $\f(t)=\log(1+t)$ and $\psi(t)=Nt^{-1}+r(1+t)^{-1}$. Define the differential operator
\[ D_\psi=\frac{1}{\psi}\cdot \frac{d}{dt}.\]
Furthermore, use the notation
\[D_{\mathcal{S}}^{(k)}=(-(r+N)D_\psi)_{(k-1)},\indent D_{\mathcal{V}}^{(k)}=r^{k-1}D_\psi^{k-1}\]
for $k\geq 0$ where $D_\psi^{-1}(-)$ denotes integrating $\psi\cdot (-)$ assuming a constant term 0. In the setting of Theorem \ref{thm:SV2d-deg-pos}, we have the following relations
\equa{
k_1!k_2!\sum_{|\mu|=k_1}\sum_{|\xi|=k_2}\binom{r}{\mu}\binom{N}{\xi}\log A_{\mu,(0),\xi}(q)&=-r\sum_{i=1}^N\left(D_{\mathcal{S}}^{(k)}\f\right)(H_i),\\
k_1!k_2!r^{k_1}\sum_{|\xi|=k_2}\binom{N}{\xi}\log B_{(1)_{k_1},(0),\xi}\left((-1)^Nq\right)&=-r\sum_{i=1}^N\left(D_{\mathcal{V}}^{(k)}\f\right)(H_i)}
where $H_i$ are the Newton-Puiseux solutions to $H_i^N=q(1+H_i)^{-r}$.
\end{theorem}

\begin{example}
We shall compute the $\log B_{(1,1,1),(0),(0)}$ using the above theorem. Set $k_1=3,k_2=0$. We have
\[D_F^{2}\f=\frac{N t \left(1+t\right)}{\left(\left(N+r\right) t+N\right)^{3}}.\]
Hence
\equa{\log B_{(1,1,1),(0),(0)}((-1)^Nq)&=\frac{-1}{3!}\sum_{n=1}^\infty \left(D^2_F\f\right)(H_i)\\
&=\frac{-N}{6}\sum_{i=1}^N\frac{ H_i \left(1+H_i\right)}{\left(\left(N+r\right) H_i+N\right)^{3}}.
}
\end{example}

\begin{example}
Similarly, we may compute $\sum_{|\mu|=3}\binom{r}{\mu}\log A_{\mu,(0),(0)}$.
\equa{&\sum_{|\mu|=3}\binom{r}{\mu}\log A_{\mu,(0),(0)}\\
=&\frac{-r}{3!}\sum_{n=1}^\infty\left(\left((r+N)^2 D_F^2+(r+N)D_F\right)\f\right) (H_i)\\
=&\frac{-r(r+N)}{6}\sum_{i=1}^N\frac{H_i\left((N+r)^2H_i^2+3N(N+r)H_i+N(2N+r)\right)}{\left((N+r)H_i+N\right)^{3}}.
}
Again, $H_i$ are the Newton-Puiseux solutions to $H_i^N=q(1+H_i)^{-r}$.
\end{example}

\subsection{Strong Segre symmetry and weak Verlinde symmetry}\label{sec:sym}
Let $r>0$. According to Theorem \ref{thm:SV2d-deg-pos}, the following identity holds for the Segre series
\equa{&[q^n]\sum_{|\mu|=k}\binom{r}{\mu}\log A^{r,N}_{\mu,(0),(0)}((-1)^Nq)=[q^n]\sum_{|\xi|=k}\binom{N}{\xi}\log A^{r,N}_{(0),(0),\xi}((-1)^Nq)\\
=&(-1)^{nN-1}\frac{r}{n}(-n(r+N))_{(k-1)}\binom{-nr-1}{nN-1}\\
=&(-1)^{nN-1}\frac{r}{n}\cdot\frac{(-nr-1)\cdots (-n(r+N)-k+2)}{(nN-1)!}\\
=&(-1)^{k-1}rN\cdot(n(r+N)+k-2)_{(k-2)}\binom{n(r+N)}{nN}
}
The right-hand side does not change if we swap $r$ and $N$, that is
\equa{\sum_{|\mu|=k}\binom{r}{\mu}\log A^{r,N}_{\mu,(0),(0)}((-1)^Nq)=\sum_{|\mu|=k}\binom{r}{\mu}\log A^{N,r}_{(0),(0),\mu}((-1)^rq).}
This is consistent with the strong Segre symmetry of Conjecture \ref{con:strong-sym}. We have checked that this conjecture holds when
\begin{equation}
\label{eq:2dsegsym}
\begin{cases}
n=1,\text{for $N\leq 5$, $r\leq 3$,}\\
n=2,\text{for $N\leq 3$, $r\leq 3$,}\\
n=3,\text{for $N\leq 3$, $r\leq 2$,}\\
n=4,5,\text{for $N\leq 2$, $r=1$.}\\
\end{cases}
\end{equation}

As for the Verlinde series, the weak Segre symmetry of Corollary \ref{cor:sv-sym2d-intro} together with the weak Segre-Verlinde correspondence of Corollary \ref{cor:weak-SV2d-virtual-intro} would imply a weak Verlinde symmetry. However, calculations using a computer program show the ``strong'' Verlinde symmetry does not hold for $S=\CC^2$. This can also be observed from the fact that the equivariant Verlinde series only depends on $V$ through $c_1(V)$ but depends on $E$ through $c_1(E),c_2(E),\dots,c_N(E)$ thus breaking the symmetry.
\section{Segre and Verlinde invariants on \texorpdfstring{$\CC^4$}{Lg}}

Consider $X=\CC^4$ with a $(\CC^4)^*$-action by scaling coordinates
\[(t_1,t_2,t_3,t_4)\cdot (x_1,x_2,x_3,x_4) = (t_1x_1,t_2x_3,t_3x_3,t_4x_4).\]
Let $\T_0=\{(t_1,t_2,t_3,t_4):t_1t_2t_3t_4=1\}\seq (\CC^4)^*$ be the subtorus which preserves the usual volume form on $X$, making $X$ a smooth quasi-projective toric Calabi-Yau 4-fold. As in the surface case, we include two additional tori
\[\T_1=(\CC^*)^N,\indent\T_2=(\CC^*)^{r+s}.\]
where $\T_1$ acts naturally on $\CC^N$, and $\T_2$ acts naturally on $\CC^{r}\times \CC^s$. Set $\T=\T_0\times\T_1\times \T_2$, and consider $E=\oplus_{i=1}^N\O_X\<y_i\>, \a=[\oplus_{i=1}^r\O_X\<v_i\>]-[\oplus_{i=r+1}^{r+s}\O_X\<v_i\>]$. Write
\equa{K_\T(\pt)&=\ZZ[t_1^{\pm1},t_2^{\pm1},t_3^{\pm1},t_4^{\pm1};y_1^{\pm1},\dots ,y_N^{\pm1};v_1^{\pm1},\dots ,v_{r+s}^{\pm1}]/(t_1t_2t_3t_4-1),\\
H^*_\T(\pt)&=\CC[\lambda_1,\lambda_2,\lambda_3,\lambda_4;m_1,\dots ,m_N;w_1,\dots ,w_{r+s}]/(\lam_1+\lam_2+\lam_3+\lam_4).}

By \cite[Theorem~4.1]{HT} the truncated Atiyah class of the universal subsheaf $\mathcal{I}$ defines an obstruction theory
\[\mathbf{R}\shom_{p}(\mathcal{I},\mathcal{I})_0^\vee[-1]\->L_{\quot_X(E,n)}^\bullet\]
where $\mathbf{R}\shom_{q}=\mathbf{R}q_*\circ \mathbf{R}\shom$, $(\cdot)_0$ denotes the trace free part. Note that the obstruction theory is $\T$-equivariant by \cite[Theorem B]{Ri}. The virtual tangent bundle is then
\[T^\vir=-\mathbf{R}\shom_{p}(\mathcal{I},\mathcal{I})_0\in K_\T(\quot_X(E,n)).\]

\subsection{Cohomological virtual invariants}\label{sec:vir-inv4d}
When $X$ is a projective Calabi-Yau 4-fold, the virtual fundamental class involves a choice of orientation on $\quot_X(E,n)$. Let $\mathcal{L}=\text{det}\mathbf{R}\shom_{q}(\mathcal{I},\mathcal{I})$ be the determinant line bundle. An \emph{orientation} $o(\mathcal{L})$ is a choice of square root of the isomorphism
\[Q:\mathcal{L}\otimes \mathcal{L}\->\mathcal{O}_{\quot_X(E,n)}\]
induced by Serre duality. A virtual class $[\quot_X(E,n)]^\vir_{o(\mathcal{L})}\in H_{2nN}(\quot_X(E,n),\ZZ)$ was constructed in \cite[§2.1]{Bojko2} for $X$ a strict Calabi-Yau 4-fold and $E$ a simple rigid locally-free sheaf. For $\gamma\in H^{2nN}(\quot_X(E,n))$, the holomorphic Donaldson invariants are defined to be
\[\mathcal{Q}(\gamma)=\int_{[\quot_X(E,n)]^\vir_{o(\mathcal{L})}}\gamma.\]

For the non-compact $X=\CC^4$, similar to the surface case, we have that the $\T$-fixed locus of $\hilb^n(X)$ consists of only finitely many reduced points \cite[Lemma~3.6]{CK1}, so we can define these invariants equivariantly using Oh-Thomas' localization formula \cite[Theorem~7.1]{Oh:2020rnj}. For a fixed orientation and any $SO(2k)$-bundle $B$, one can define its Edidin-Graham square root Euler class $\sqrt{e}(B)$ \cite[§3]{Oh:2020rnj}. Consider the self-dual resolution (\ref{eqn:res}) of $T^\vir$
\[\phi: B^\bullet\to L^\bullet_{\quot_{X}(E,n)}\]
where $B^\bullet = (T\-> B\->T^*)$ and $B$ a $SO(6n^2)$-bundle. Recall that $T^\vir$ has no fixed parts. Set \cite[(115)]{Oh:2020rnj}
\[\sqrt{e_\T}(T^\vir):=\sqrt{e_\T}(B^\bullet):=\frac{e_\T(T)}{\sqrt{e_\T}(B)}.\]
For computational purposes, we consider a square root $\sqrt{T^\vir|_Z}\in K_\T(\pt)$ for each fixed point $Z$ such that
\[T^\vir|_Z=\sqrt{T^\vir|_Z}+\overline{\sqrt{T^\vir|_Z}},\]
where $\overline{(\cdot)}$ denotes the involution $t_i\mapsto t_i^{-1}$. This allows us to compute the square root Euler class at the cost of a sign dependent on the choice of orientation $o(\mathcal{L})$
\[\sqrt{e}(T^\vir|_Z)=\pm e\left(\sqrt{T^\vir|_Z}\right).\]
As each fixed point is reduced, Kool-Rennemo \cite{KR-draft} show that their virtual fundamental classes are given by further signs determined by some induced orientation on $B^\bullet$.
We will denote the product of the two signs at each $Z\in\quot_X(E,n)^\T$ by 
$(-1)^{o(\mathcal{L})|_Z}$.
\begin{definition} For $n>0$, $\gamma\in H^{*}_\T(\quot_X(E,n))$, the \emph{holomorphic Donaldson invariants} are
\[
\mathcal{Q}(X,E,n,\gamma):=\sum_{Z\in\quot_X(E,n)^\T}(-1)^{o(\mathcal{L})|_Z}\frac{\gamma|_Z}{ e(\sqrt{T^\vir|_Z})}.\]
\end{definition}

The authors of \cite{KR-draft} further gave an explicit description of the  moduli space $\quot_X(E,n)$ when $X=\CC^4$ as a vanishing locus of an isotropic section. This allowed them to derive the signs $(-1)^{o(\mathcal{L})|_Z}$ by knowing $\L$ and $T^\vir$. For the purpose of motivating the Verlinde invariants from the introduction, we recall their approach here. 

 Consider the quiver 
 \begin{figure}[h]
     \centering
 \includegraphics[scale = 0.5]{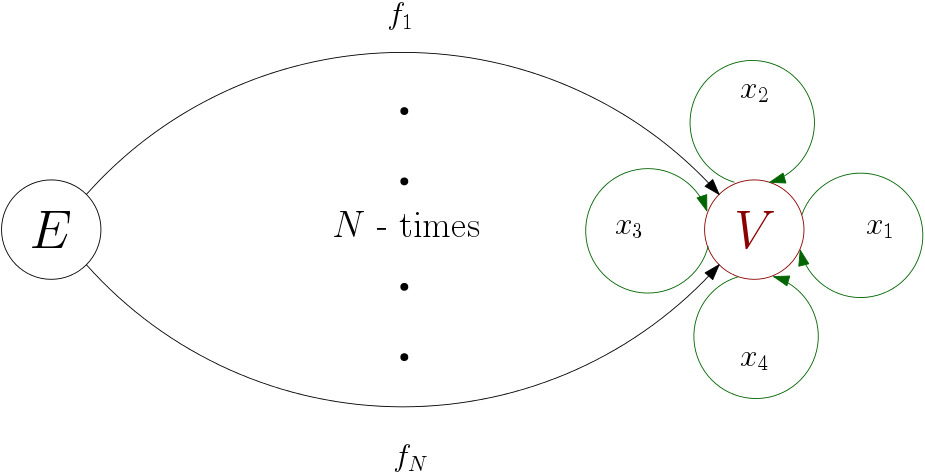}
     \caption{Framed quiver with four loops at one node.}
 \end{figure}
with four loops and $N$ framings. After imposing the relations
\begin{equation}
\label{eq:relation}
[x_i,x_j] = 0\,,
\end{equation}
its representations with dimension vector $(1,n)$ consist of a one-dimensional complex vector space $\CC$ and an $n$-dimensional vector space $V$ together with $N$ morphisms $f_i: \CC\to V$ for $i=1,\ldots,N$ and 4 morphisms $x_i: V\to V$ satisfying \eqref{eq:relation}.

The space of all representations without requiring the relations is 
$$R = \End(V)^{\oplus 4}\oplus\hom(E,V)$$
where we used the suggestive notation $E=\CC^N$. It carries an $SO(6n^2)$ vector bundle $B$ trivial with fiber  
\[
\Lambda^2\CC^4\otimes \End(V)
\]
and the pairing $$
b:  \Lambda^2\CC^4\otimes \End(V) \otimes  \Lambda^2\CC^4\otimes \End(V)\xrightarrow{(-\wedge -)\otimes \tr\big(-\circ-\big)} \Lambda^4\CC^4\otimes \End(V) =\End(V)\,.
$$
The existence of an isotropic section 
$$
s: R\to B\,,\qquad (\vec{x},\vec{f})\mapsto \sum_{i\neq j}(e_i\wedge e_j) \otimes x_i\circ x_j\,.
$$
connects it to the local toy model of \cite[(1)]{Oh:2020rnj}.

To construct the Quot scheme, we need to take a quotient by the $GL(V)$-action on $R$  defined for each $g\in GL(V)$ by $\big(\vec{x},\vec{f}\big)\mapsto \big(g\circ \vec{x}\circ g^{-1}, g\circ \vec{f}\big)$. One can also extend it in a natural way to an action on $B$. In particular, after defining $R^0\subset R$  as the open subscheme of representations satisfying
$$
\CC[x_1,\ldots, x_4] \cdot \CC\big\langle f_1(1),\ldots, f_N(1)\big\rangle = V\,,
$$
the vector bundle $B$ and its section $s$ restrict and then descends to the non-commutative Quot scheme
$$
A = \text{NC}\quot_{X}(E,n) = \big[R^0/GL(V)\big]\,.
$$
The usual Quot scheme is identified with the zero locus
\begin{equation}
\label{eq:quotinncquot}
\quot_{X}(E,n) = s^{-1}(0) \subset A\,.
\end{equation}
We will always use the letters $B$ and $s$ to denote the vector bundle and section on $R$, $R^0$, their descent to $\text{NC}\quot_{X}(E,n)$ and restriction to $\quot_{X}(E,n)$ without distinguishing the four cases.

To make it all equivariant, one uses the action of $\T_0\times \T_1$ on $R$ by 
$$
(\vec t;\vec y)\cdot ( \vec x;\vec f)  =  ( t_1\cdot x_1,t_2\cdot x_2,t_3\cdot x_3,t_4\cdot x_4;y_1\cdot f_1,\ldots,y_N\cdot f_N)
$$
which commutes with the action of $\gl(V)$ so it descends to one on $\quot_{X}(E,n)$. If we use $T$ to denote the tangent bundle of $A$, which at each point is the cokernel of some injective map 
\begin{equation}
\label{eq:T}
\End(V)\hookrightarrow R
\end{equation} obtained by differentiating the action of $\gl(V)$ on $R$,
then there is a $\T_0\times \T_1$-equivariant resolution of $\mathbf{R}\shom_{p}(\mathcal{I},\mathcal{I})^\vee_0[-1]$ given by 
$$
B^\bullet = (T\xrightarrow{ds^*} B\xrightarrow{ds}T^*)
$$
which gives the natural $\T_0\times \T_1$-equivariant obstruction theory 
\equa{\label{eqn:res}
\phi: B^\bullet\to L^\bullet_{\quot_{X}(E,n)}\,.
}
Note that the first term in \eqref{eq:T} has trivial weights.

The choice of orientations $o(\L)$ was shown in \cite[Prop. 4.2]{Oh:2020rnj} to be equivalent in this setting to choosing a positive isotropic subbundle $\sI$ of $B$. This is done explicitly in \cite{KR-draft} by constructing
$
\sI
$ as a trivial bundle with the fiber 
\begin{equation}
\label{eq:I}
\langle v\rangle \wedge \langle v\rangle^{\perp}\otimes \End(V)\end{equation} for some non-zero vector $v\in \CC^4$\footnote{In fact, the particular choice of signs compatible with the existing literature corresponds to setting $v=e_4$ for the fourth vector of the canonical basis of $\CC^4$.}. We will not go further into recalling the explicit derivation of $o(\L)|_Z$ done in \cite{KR-draft}, as we only do computations up to a fixed order $n$ to formulate conjectures the proof of which we leave for the future.

\begin{definition}Let $X=\CC^4$, $\alpha=[\oplus_{i=1}^r\O_X\<v_i\>]-[\oplus_{i=r+1}^{r+s}\O_X\<v_i\>]\in K_\T(X)$, and $E=\oplus_{i=1}^N\O_X\<y_i\>$. The \emph{equivariant Segre and Chern series} for a choice of signs $o(\mathcal{L})$ are respectively
\equa{{\mathcal{S}}_X(E,\a;q):=&\sum_{n=0}^\infty q^n\sum_{Z\in\quot_X(E,n)^\T}(-1)^{o(\mathcal{L})|_Z}\frac{s(\alpha^{[n]}|_Z)}{ e\left(\sqrt{T^\vir|_Z}\right)}\\
&\in\frac{\CC(\lambda_1,\lambda_2,\lambda_3,\lambda_4;m_1,\dots ,m_N;w_1,\dots ,w_{r+s})}{(\lambda_1+\lambda_2+\lambda_3+\lambda_4)}[q],\\
{\mathcal{C}}_X(E,\a;q):=&\sum_{n=0}^\infty q^n\sum_{Z\in\quot_X(E,n)^\T}(-1)^{o(\mathcal{L})|_Z}\frac{c(\alpha^{[n]}|_Z)}{ e\left(\sqrt{T^\vir|_Z}\right)}.
}
\end{definition}
\subsection{K-theoretic virtual invariants}
\label{sec:CY4ktheory}
In the setting where the moduli space $M$ is a zero locus of an isotropic section $s$ of an $SO(2m)$ bundle $B$ on some ambient space $A$, just like in \eqref{eq:quotinncquot} above, \cite{Oh:2020rnj} give a simpler construction of $\hat{\O}^\vir$ relying on their equivariant \textit{localized K-theoretic square root Euler class} $$\sqrt{\mathfrak{e}_{\T}}(E,s): K_0\big(A,\ZZ\big)\to K_0\big(M,\ZZ[2^{-1}]\big)$$ defined after choosing orientations on $M$.

When $B$ admits an isotropic subbundle $\sI$ compatible with the choice of orientation, then the pushforward under the inclusion 
$$
\iota:M=s^{-1}(0)\hookrightarrow A
$$
becomes just tensoring with the equivariant \textit{K-theoretic square root Euler class} 
$$
\iota_*\big(\sqrt{\mathfrak{e}_{\T}}(E,s)\big) = \otimes \sqrt{\mathfrak{e}_{\T}}(E) = (-1)^{ m}\mathfrak{e}_{\T}(I^*)\sqrt{\det}\big(\sI^*\big)
$$
where $\mathfrak{e}_{\T}(\sI^*) = \Lambda_{-1}\sI^*$. In fact, $\sqrt{\mathfrak{e}_{\T}}(E,s)$ can be also written as the product
$$
\sqrt{\mathfrak{e}_{\T}}(E,s) = (-1)^m\mathfrak{e}_{\T}(I^*,s) \sqrt{\det}\big(\sI^*\big)\,,
$$
where $\mathfrak{e}_{\T}(I^*,s)$ is some localization of $\mathfrak{e}_{\T}(I^*)$ to $M=s^{-1}(0)$ constructed using Kiem-Li's cosection localization in K-theory \cite{KLcosectionKtheory}. Their equivariant \textit{twisted virtual structure sheaf} $\hat{\O}^{\vir}$ is then constructed as 
\begin{align*}
\hat{\O}^{\vir} &= \sqrt{\mathfrak{e}_{\T}}(B,s)\big([\![ \O_A]\!]\big) \sqrt{\det}(T^*)\\
&= (-1)^m\mathfrak{e}_{\T}(I^*,s)\big([\![ \O_A]\!]\big) \sqrt{\det}\big(T^*+\sI^* \big)\,.
\end{align*}
For our case of $M=\quot_{\CC^4}(E,n)$,  we can use \eqref{eq:T} and $\eqref{eq:I}$ to show that $$\det(T) = \det\left(\End(V)\right)^{\otimes 4}\prod_{i=1}^4t_i^{n^2}\det\left(\hom(E,V)\right) =\det\left(\hom(E,V)\right)$$ and 
$\det(\sI) = t_4^{2n^2}$. This implies that the only term in the construction of $\hat{\O}^\vir$ that does not admit a square root is
\begin{align*}
{\det}^{-1}\big(\hom(E,V)\big) &= {\det}^{-1}\big((E^\vee)^{[n]}\big) \\
&= (y_1\ldots y_N)^{-n}{\det}^{-N}(V)\,.
\end{align*}
This gives further motivation for the definition of the \emph{untwisted virtual structure sheaf}
\equanum{\label{def:untwist}\mathcal{O}^\vir :=\hat{\mathcal{O}}^\vir\otimes \mathsf{E}^{\frac{1}{2}}\,,\quad \mathsf{E} = \det((E^\vee)^{[n]}).}
that appeared in \cite[Definition 5.10]{boj}, \cite[§1.4]{Bojko2}. From the above discussion it is clear that the following integrality statement holds.
\begin{proposition}
\label{prop:integral}
The untwisted virtual structure sheaf is an integral class:
$$
\mathcal{O}^\vir \in K_0\big(\quot_{\CC^4}(E,n),\ZZ\big)\,.
$$
\end{proposition}

Using $\hat{\O}^\vir$ and $\O^\vir$, we define the following \emph{twisted} and \emph{untwisted Euler characteristics}
\[\hat\chi^\vir(\quot_Y(E,n),-) = \chi\big(\quot_Y(E,n),\hat\O^\vir\otimes (-)\big),\]
$$\chi^\vir(\quot_Y(E,n),-) = \chi\big(\quot_Y(E,n),\O^\vir\otimes (-)\big).$$
For compact $X$ and $\alpha\in K^0(X)$, the \emph{Verlinde series} \cite[§1.3]{Bojko2} is then defined by
\equa{{\mathcal{V}}_X(E,\alpha;q):=&\sum_{n=0}^\infty q^n\chi^{\vir}\left(\quot_X(E,n),\text{det}(\alpha^{[n]})\right)\\
=&\sum_{n=0}^\infty q^n\hat\chi^\vir\left(\quot_X(E,n),\text{det}(\alpha^{[n]})\otimes \sqrt{\det}((E^\vee)^{[n]})\right).
}

Using the virtual Riemann-Roch formula and equivariant localization of Oh-Thomas \cite[Theorem 6.1, Theorem 7.3]{Oh:2020rnj}, we have the following \emph{twisted} equivariant virtual Euler characteristic for $X=\CC^4$:
\[\hat\chi^\vir_\T(\quot_X(E,n),\alpha):=\sum_{Z\in\quot_X(E,n)^\T}(-1)^{o(\mathcal{L})|_Z} e\left(-\sqrt{T^\vir|_Z}\right)\sqrt{\td}\left(T^\vir|_Z\right)\ch_\T(\alpha)\]
where $\sqrt{\td}$ is the the square-root Todd class satisfying
\equa{\sqrt{\td}(T^\vir|_Z)=&\td\left(\sqrt{T^\vir|_Z}\right)\ch\left(\sqrt{\det}\sqrt{T^\vir|_Z}^\vee\right)\\
=&\frac{e\left(\sqrt{T^\vir|_Z}\right)}{\ch\left(\Lambda_{-1}\sqrt{T^\vir|_Z}^\vee\right)}\ch\left(\sqrt{K^\vir}^{\frac12}\right).}
Here we denote
\[K^\vir={\det} (T^{\vir})^\vee, \indent \sqrt{K^\vir}={\det}\sqrt{T^\vir}^{\vee}.\]
Substituting into the above equation, we have
\equa{\hat\chi^\vir(\quot_X(E,n),\alpha)=\sum_{Z\in\quot_X(E,n)^\T}(-1)^{o(\mathcal{L})|_Z} \frac{\ch\left(\sqrt{K^\vir|_Z}^\frac12\right)}{\ch\left({\Lambda_{-1}}\sqrt{T^\vir|_Z}^\vee\right)}\ch(\alpha^{[n]}|_Z).}
Now including the twist of (\ref{def:untwist}), we may define the equivariant Verlinde series as follows.

\begin{definition}
The \emph{equivariant Verlinde series} for a choice of signs $o(\mathcal{L})$ is
\equa{{\mathcal{V}}_X(E,\a;q):=&\sum_{n=0}^\infty q^n\sum_{Z\in\quot_X(E,n)^\T}(-1)^{o(\mathcal{L})|_Z}\frac{ \ch\left(\sqrt{K^\vir|_Z}^\frac12\right)\ch\left(\sqrt{\det}((E^\vee)^{[n]}|_Z)\right)}{\ch\left({\Lambda_{-1}}\sqrt{T^\vir|_Z}^\vee\right)}\ch\left(\text{det}(\alpha^{[n]}|_Z)\right)\\
&\indent \in\frac{\QQ(t_1,t_2,t_3,t_4;y_1,\dots ,y_N;v_1,\dots ,v_{r+s})}{(t_1t_2t_3t_4)}[\![q]\!].
}
\end{definition}

The relation between Segre and Verlinde numbers in the compact case is studied in \cite{Bojko2} using the \emph{Nekrasov genus} for Hilbert schemes, introduced for the 3-fold case by \cite{NO}. We consider the following Quot scheme version from \cite{magcolor}
\equanum{\label{def:nek-genus}{\mathcal{N}}_X(E,\a;q)
:=&\sum_{n=0}^\infty q^n\sum_{Z\in\quot_X(E,n)^\T}(-1)^{o(\mathcal{L})|_Z}\frac{\ch(\sqrt{K^\vir|_Z}^{\frac12})}{\ch(\Lambda_{-1} \sqrt{T^\vir|_Z}^\vee)}
\ch\left(\frac{\Lambda_{-1}}{\sqrt{\det}} \a^{[n]}|_Z\right).
}
\begin{remark}
Recall in \cite{CKM}, the Nekrasov genus for Hilbert schemes involves a variable $y$ coming from a trivial $\CC^*$-action on $X$. This is exactly the $N=1$ case for the above definition, where we have the parameter $y_1$ from the $\T_1$-action on $E$.
\end{remark}


\subsection{Vertex formalism}
The invariants in the previous section can be calculated for Hilbert schemes using a vertex formalism developed by \cite{CK2}, based on the method introduced in \cite{MNOP1} for Calabi-Yau 3-folds. We generalize this to Quot schemes using the computations from \cite[§2.1]{Bojko2}. First when $N=1$, the $\T$-fixed points of $\hilb^n(X)$ correspond to monomial ideals of $\CC[x_1,x_2,x_3,x_4]$ \cite[Lemma 3.1]{CK1}, which are labeled by solid partitions $\pi$ of size $n$ where
\[\O_{Z_\pi}=\CC[x_1,x_2,x_3,x_4]/I_{Z_{\pi}}=\Span\{x_1^{a}x_2^bx_3^cx_4^d:(a,b,c,d)\in\pi\}.\]
We denote $Q_\pi$ the character of $\O_{Z_\pi}$
\[Q_\pi=\sum_{(i,j,k,l)\in\pi}t_1^{-a}t_2^{-b}t_3^{-c}t_4^{-d}\in K_{\T}^*(\mathrm{pt})=\frac{\ZZ[t_1^{\pm1},t_2^{\pm1},t_3^{\pm1},t_4^{\pm1}]}{(t_1t_2t_3t_4-1)}.\]
Similar to the surface case, for $E=\oplus_{i=1}^N\O_X\<y_i\>$, the $\T$-fixed points for $\quot_X(E,n)$ are labeled by $N$-colored solid partitions $\pi=(\pi^{(1)},\dots,\pi^{(n)})$ of size $n$, i.e. sequences of form
\[Z_\pi=([Z_1],[Z_2],\dots ,[Z_N])\in\hilb^{n_1}(X)\times\dots \times \hilb^{n_N}(X)\]
such that each $Z_i$ corresponds by solid partition $\pi^{(i)}$.

Let $Q_i$ be the character of $\O_{Z_i}$. The virtual tangent bundle at $Z_\pi$ is
\equanum{\label{eqn:t-vir-4d}T^\vir_{Z_\pi}=&\ext\left(\bigoplus_{i=1}^N I_{\mathcal{Z}_i}\<y_i\>, \bigoplus_{j=1}^NI_{\mathcal{Z}_j}\<y_j\>\right)_0\\
=&\sum_{i,j=1}^N \O_X\otimes (1- \overline{P(I_{Z_i})}P(I_{Z_j}))y_i^{-1}y_j\\
=&\sum_{i,j=1}^N\left(Q_{j}+t_1t_2t_3t_4\overline{Q_{i}}-t_1t_2t_3t_4P_{1234}\overline{Q_{i}}Q_{j}\right)y_i^{-1}y_j
}
where $P(I)$ is the Poincar\'e polynomial of $I$ defined analogously to (\ref{def:poincare}), and $P_{I}:=\prod_{i\in I}(1-t_i^{-1})$ for the set of indices $I$. Specializing $t_1t_2t_3t_4=1$, we get the following (non-unique) square root
\[\sqrt{T^\vir_{Z_\pi}}=\sum_{i,j=1}^N\left(Q_{j}-\overline{P_{123}}\overline{Q_{i}}Q_{j}\right)y_i^{-1}y_j.\]
The reason for the above choice of square root is so that
\equa{\ch\left(\sqrt{K^\vir|_{Z_\pi}}^{\frac12}\right)&=\ch\left(\prod_{i,j}\text{det}((\overline{Q_j-\overline{P}_{123}Q_j\overline{Q_i}})y_i^{-1}y_j)^{\frac12})\right)\\
&=\ch\left(\prod_{i,j}\sqrt{\det}(\overline {Q_j})y_iy_j^{-1})\right)\\
&=\frac1{\ch\left(\sqrt{\det}((E^\vee)^{[n]}|_{Z_\pi})\right)}
}
matches our twist in (\ref{def:untwist}), and this simplifies our computation as now we have
\[{\mathcal{V}}_X(E,\a;q)=\sum_{\pi}^\infty q^{|\pi|}(-1)^{o(\mathcal{L})|_{Z_\pi}}\frac{ \ch\left(\text{det}(\alpha^{[n]}|_{Z_\pi})\right)}{\ch\left({\Lambda_{-1}}\sqrt{T^\vir|_{Z_\pi}}^\vee\right)}\]
In this case, the signs $(-1)^{o(\mathcal{L})}$ are described in \cite{monavari,KR-draft} as follows: for any solid partition $\pi$,
\[o(\mathcal{L})|_{Z_\pi}:=|\pi|+\#\{(i,i,i,j)\in\pi:i<j\};\]
and for any $N$-colored solid partition $\pi$,
\[o(\mathcal{L})|_{Z_\pi}:=\sum_{i=1}^No(\mathcal{L})|_{Z_{i}}.\]

The fiber of $V^{[n]}=\oplus_{i=1}^r \O_X^{[n]}\<v_i\>$ over ${Z_\pi}=(Z_1,\dots Z_N)$ is the $rn$-dimensional representation
\[V^{[n]}|_{Z_\pi}=\bigoplus_{i=1}^r\bigoplus_{j=1}^N\O_{Z_{j}}\<v_iy_j\>=\left(\sum_{i=1}^r\sum_{j=1}^N \sum_{(a,b,c,d)\in\pi^{(j)}}v_iy_j t_1^{-a}t_2^{-b}t_3^{-c}t_4^{-d}\right)\in K_\T(\pt).\]
Therefore for any point $Z$ corresponding to an $N$-colored solid partition $\pi$, we have
\equa{c(V^{[n]}|_{Z_\pi})=&\prod_{j=1}^N\prod_{(a,b,c,d)\in\pi^{(j)}}\prod_{i=1}^r(1+w_i+m_j-a\lambda_1-b\lambda_2-c\lambda_3-d\lambda_4)\\
\text{det}(V^{[n]}|_{Z_\pi})=&\prod_{j=1}^N\prod_{(a,b,c,d)\in\pi^{(j)}}\prod_{i=1}^rv_iy_jt_1^{-a}t_2^{-b}t_3^{-c}t_4^{-d},\\
\ch(\sqrt{K^\vir|_{Z_\pi}}^{\frac12})
\ch\left(\frac{\Lambda_{-1}}{\text{det}^{\frac12}} V^{[n]}|_{Z_\pi}\right)=&\prod_{j=1}^N\prod_{(a,b,c,d)\in\pi^{(j)}}t_1^{\frac a2}t_2^{\frac b2}t_3^{\frac c2}t_4^{\frac d2}\\
&\cdot \prod_{i=1}^r\left(v_i^{\frac12}y_j^{\frac12}t_1^{-\frac a2}t_2^{-\frac b2}t_3^{-\frac c2}t_4^{-\frac d2}-v_i^{-\frac12}y_j^{-\frac12}t_1^{\frac a2}t_2^{\frac b2}t_3^{\frac c2}t_4^{\frac d2}\right).
}
Using these expressions, we see the Chern and Verlinde series can be extracted by taking limits of the Nekrasov genus, similar to the surface case. Also, it follows that
\equa{{\mathcal{N}}_X(E,V;q)\in\frac{\QQ(t_1^{\frac12},t_2^{\frac12},t_3^{\frac12},t_4^{\frac12})}{(t_1t_2t_3t_4-1)}[\![q,y_1^{\pm\frac12},\dots ,y_N^{\pm\frac12},v_1^{\pm\frac12},\dots ,v_r^{\pm\frac12}]\!].}
The argument of \cite[Proposition 1.13, 1.15]{CKM} can be applied to show that ${\mathcal{N}}_X(E,V;q)$ in fact lives in $\frac{\QQ(t_1,t_2,t_3,t_4)}{(t_1t_2t_3t_4-1)}[\![q,y_1^{\pm\frac12},\dots ,y_N^{\pm\frac12},v_1^{\pm\frac12},\dots ,v_r^{\pm\frac12}]\!]$. This enables us to talk about admissibility (up to specializing $t_1t_2t_3t_4=1$) in the sense of Definition \ref{def:admissible}.

\subsection{Factor of \texorpdfstring{$c_3(X)$}{Lg}}\label{sec:factor}
In the surface case, we saw the powers in the universal series of virtual invariants are multiples of $c_1(S)$. In the $X=\CC^4$ case, we shall show that if the universal expressions exist, then they are multiples of $c_3(X)$ by showing show that \equa{e\left(-\sqrt{T^\vir|_{Z_\pi}}\right)}
has $c_3(X)=-(\lambda_1+\lambda_2)(\lambda_1+\lambda_3)(\lambda_2+\lambda_3)$ in its numerator. This factor of $c_3(X)$ and the weak Segre-Verlinde correspondence and Segre symmetry of Corollary \ref{thm:SV2d-virtual-intro} and Corollary \ref{cor:sv-sym2d-intro} in the surface case shall motivate Conjecture \ref{con:sv4d-intro}. This is because the only degree zero contribution linear in $c_3(X)$ is expected to come from $\int_X c_3(X)c_1(\alpha)$ which was already studied in the compact case in \cite[§5.3]{Bojko2}. We do not expect any additional terms with exponential  $\int_X c_3(X)c_1(\alpha)$ coming from the equivariant setting just like we did not have any in the case of a surface.

It suffices to show that this term vanishes when we set $\lambda_i=-\lambda_j$ for $i\neq j$ in $\{1,2,3\}$. By symmetry, we may assume $i=1,j=2$. 
Recall that $e$ is the top equivariant Chern class which vanishes if its input has a trivial summand. The process of setting $\lam_1=-\lam_2$ in cohomology is the same as setting $t_1=t_2^{-1}$ in K-theory. Therefore we would like to show that $-\sqrt{T^\vir|_{Z_\pi}}$ has a trivial summand when we set $t_1=t_2^{-1}$, i.e. the character of $\sqrt{T^\vir|_{Z_\pi}}$ in $K_\T(\pt)$ having a strictly negative constant term. This occurs if and only if the image of $T^\vir_{Z_\pi}$ in 
\[\ZZ[t_1^{\pm1},t_2^{\pm1},t_3^{\pm1},t_4^{\pm1}]/(t_1t_2-1,t_3t_4-1)\]
has a strictly negative constant term (which is necessarily a negative even integer). From (\ref{eqn:t-vir-4d}), we see it suffices to show this for the term
\[Q_{\pi}+t_1t_2t_3t_4\overline{Q_{\pi}}-t_1t_2t_3t_4P_{1234}\overline{Q_{\pi}}Q_{\pi}\]
whenever $\pi$ is a non-trivial solid partition. 

\begin{lemma}
For any non-trivial solid partition $\pi$, the expression 
\[Q_{\pi}+t_1t_2t_3t_4\overline{Q_{\pi}}-t_1t_2t_3t_4P_{1234}\overline{Q_{\pi}}Q_{\pi}\]
has a strictly negative constant term when viewed in the quotient ring
\[\ZZ[t_1^{\pm1},t_2^{\pm1},t_3^{\pm1},t_4^{\pm1}]/(t_1t_2-1,t_3t_4-1).\]
\end{lemma}
\begin{proof}
Let $x=t_1=\frac1{t_2}$, $y=t_3=\frac1{t_4}$, so that
\[\ZZ[t_1^{\pm1},t_2^{\pm1},t_3^{\pm1},t_4^{\pm1}]/(t_1t_2-1,t_3t_4-1)=\ZZ[x^{\pm1},y^{\pm1}].\]
Let $P_\pi$ be the image of $Q_\pi$ in $\ZZ[x^{\pm1},y^{\pm1}]$, then
\equanum{\label{eqn:vir-tan-image}T^\vir|_{Z_\pi}=P_\pi + \overline{P_\pi} - P_\pi \overline{P_\pi} (1-x)(1-\frac1x)(1-y)(1-\frac1y).}
Write
\[P_\pi=\sum_{i,j\in\ZZ}p_{i,j}x^iy^j.\]
The image of $\overline{Q_\pi}$ is then
\[\overline{P_\pi}=\sum_{i,j\in\ZZ}p_{i,j}x^{-i}y^{-j}.\]
We see the constant terms of $P_\pi$ and $\overline{P_\pi}$ are both $p_{0,0}$. By definition, all monomial terms in $Q_\pi$ have positive coefficients, and $Q_\pi$ has constant term 1, so $p_{0,0}>0$. We need to find the constant term of $P_\pi \overline{P_\pi} (1-x)(1-\frac1x)(1-y)(1-\frac1y)$.

Observe that
\[(1-x)(1-\frac1x)(1-y)(1-\frac1y)=4-2\left(x+y+\frac1x+\frac1y\right)+\left(xy+\frac1{xy}+\frac xy+\frac yx\right).\]
Write $F=\sum f_{i,j}x^iy^j$, the constant term of $F\cdot (1-x)(1-\frac1x)(1-y)(1-\frac1y)$ is equal to
\equanum{\label{const-term}4f_{0,0}-2(f_{0,1}+f_{1,0}+f_{0,-1}+f_{-1,0})+(f_{1,1}+f_{1,-1}+f_{-1,1}+f_{-1,-1}).}
If we set $F=P_\pi\overline{P_\pi}$, then
\[f_{i,j}=\sum_{\substack{
a-c=i\\
b-d=j}}p_{a,b}p_{c,d}.\]
In particular,
\equa{&f_{0,0}=\sum_{a,b\in\ZZ} p_{a,b}^2,\\
&f_{0,1}+f_{1,0}+f_{0,-1}+f_{-1,0}=\sum_{a,b\in\ZZ} p_{a,b}(p_{{a-1},b}+p_{{a+1},b}+p_{{a},b-1}+p_{{a},b+1}),\\
&f_{1,1}+f_{1,-1}+f_{-1,1}+f_{-1,-1}=\sum_{a,b\in\ZZ} p_{a,b}(p_{{a+1},b+1}+p_{{a+1},b-1}+p_{{a-1},b-1}+p_{{a-1},b+1}).
}
Denote 
\equa{&s_{a,b}=4p_{a,b}-2(p_{{a-1},b}+p_{{a+1},b}+p_{{a},b-1}+p_{{a},b+1});\\
&\indent\indent+(p_{{a+1},b+1}+p_{{a+1},b-1}+p_{{a-1},b-1}+p_{{a-1},b+1}),\\
&s_{a,b}^{++}=p_{a,b}-(p_{{a+1},b}+p_{{a},b+1})+p_{{a+1},b+1},\\
&s_{a,b}^{+-}=p_{a,b}-(p_{{a+1},b}+p_{{a},b-1})+p_{{a+1},b-1},\\
&s_{a,b}^{-+}=p_{a,b}-(p_{{a-1},b}+p_{{a},b+1})+p_{{a-1},b+1},\\
&s_{a,b}^{--}=p_{a,b}-(p_{{a-1},b}+p_{{a},b-1})+p_{{a-1},b-1};\\
}
\equa{
&S^{++}=\sum_{a,b\geq 0}p_{a,b}s_{a,b}^{++}+p_{a+1,b}s_{a+1,b}^{-+}+p_{a,b+1}s_{a,b+1}^{+-}+p_{a+1,b+1}s_{a+1,b+1}^{--},\\
&S^{+-}=\sum_{a\geq 0,b\leq 0}p_{a,b}s_{a,b}^{+-}+p_{a+1,b}s_{a+1,b}^{--}+p_{a,b-1}s_{a,b-1}^{++}+p_{a+1,b-1}s_{a+1,b-1}^{-+},\\
&S^{-+}=\sum_{a\leq 0,b\geq 0}p_{a,b}s_{a,b}^{-+}+p_{a+1,b}s_{a+1,b}^{++}+p_{a,b+1}s_{a,b+1}^{--}+p_{a+1,b+1}s_{a+1,b+1}^{+-},\\
&S^{--}=\sum_{a,b\leq 0}p_{a,b}s_{a,b}^{--}+p_{a+1,b}s_{a+1,b}^{+-}+p_{a,b+1}s_{a,b+1}^{-+}+p_{a+1,b+1}s_{a+1,b+1}^{++}.\\
}
Then (\ref{const-term}) becomes
\equa{\sum_{a,b\in\ZZ}p_{a,b}s_{a,b}=&\sum_{a,b\in\ZZ}p_{a,b}\cdot(s_{a,b}^{++}+s_{a,b}^{+-}+s_{a,b}^{-+}+s_{a,b}^{--})\\
=&S^{++}+S^{+-}+S^{-+}+S^{--}.
}

For the remainder of this proof, we shall show $S^{++}\geq p_{0,0}$. The same will hold for the summands $S^{+-},S^{-+},S^{--}$ by symmetry. We conclude the value of (\ref{const-term}) is at least $4p_{0,0}$. Hence by (\ref{eqn:vir-tan-image}) the constant term of $T^{\vir}|_{Z_\pi}$ is at most $-2p_{0,0}<0$, and we are done.

Recall 
\[Q_\pi=\sum_{(i,j,k,l)\in\pi}t_1^it_2^jt_3^kt_4^l,\]
so
\[P_\pi=\sum_{(i,j,k,l)\in\pi}x^{i-j}y^{k-l},\text{ and}\]
\[p_{a,b}=\#\{(i,j,k,l)\in\pi:i-j=a,k-l=b\}.\]
Fix $k$ and $l$, then the set
$\{(i,j): (i,j,k,l)\in\pi\}$
is a plane partition. 
By property (\ref{solid-partition}) of solid partitions, for fixed $b=k-l$, we have $p_{a,b}\geq p_{{a+1},b}$ when $a\geq 0$. For the same reason, we have $p_{a,b}\geq p_{{a},b+1}$ when $b\geq 0$.
Therefore the numbers $(p_{a,b})_{a,b\geq 0}$ are non-increasing as the pair $(a,b)$ move away from the origin.

We apply induction on $\max\{a:p_{a,0}\neq 0\}$. Suppose for all sequences $(q_{a,b})$ with $\max\{a:q_{a,0}\neq 0\}<\max\{a:p_{a,0}\neq 0\}$, we have
\[S^{++}(q_{a,b})\geq q_{0,0}\]
whenever the sequence $(q_{a,b})_{a,b\geq 0}$ satisfies $q_{a,b}$ is non-increasing in $a,b$. The base case is simply when $q_{a,b}=0$ for all $a,b$, which sums to 0. Let $q_{a,b}=p_{{a+1},b}$, then
\equanum{\label{ineq}S^{++}(p_{a,b})=&\sum_{a,b\geq 0}(p_{a,b}-p_{a+1,b}-p_{a,b+1}+p_{a+1,b+1})^2\\
=&S^{++}(q_{a,b})+\sum_{b\geq 0}(p_{0,b}-p_{1,b}-p_{0,b+1}+p_{1,b+1})^2\\
\geq& p_{1,0}+\sum_{b\geq 0}(p_{0,b}-p_{1,b}-p_{0,b+1}+p_{1,b+1})^2
}
where the first equality follows from the definition and the inequality is by the induction hypothesis. 

Now apply another induction on the value of $\max\{b:p_{0,b}\neq 0\}$. The induction hypothesis is that for any sequences $(q_{a,b})$ with $q_{a,b}$ non-increasing in $a,b$ and $\max\{b:q_{0,b}\neq 0\}<\max\{b:p_{0,b}\neq 0\}$, we have
\[\sum_{b\geq 0}(q_{0,b}-q_{1,b}-q_{0,b+1}+q_{1,b+1})^2\geq q_{0,0}-q_{1,0}.\]
Again, the base case is trivial, and we can apply the hypothesis to $q_{a,b}=p_{a,b+1}$, giving us
\[\sum_{b\geq 1}(p_{0,b}-p_{1,b}-p_{0,b+1}+p_{1,b+1})^2\geq p_{0,1}-p_{1,1}\]
So we have the following inequalities
\equa{&(p_{0,0}-p_{1,0}-p_{0,1}+p_{1,1})^2+\sum_{b\geq 1}(p_{0,b}-p_{1,b}-p_{0,b+1}+p_{1,b+1})^2-(p_{0,0}-p_{1,0})\\
\geq& (p_{0,0}-p_{1,0}-p_{0,1}+p_{1,1})^2-(p_{0,0}-p_{1,0}-p_{0,1}+p_{1,1})\\
\geq& 0
}
where the last inequality is due to $p_{0,0}-p_{1,0}-p_{0,1}+p_{1,1}$ being an integer. Therefore
\[\sum_{b\geq 0}(p_{0,b}-p_{1,b}-p_{0,b+1}+p_{1,b+1})^2\geq p_{0,0}-p_{1,0},\]
finishing the second induction. By (\ref{ineq}),
\[S^{++}(p_{a,b})\geq p_{1,0}+(p_{0,0}-p_{1,0})=p_{0,0},\]
which finishes the first induction and the proof.
\end{proof}
We checked the Conjectures \ref{con:sv4d-intro} using a computer program. The correspondence part was checked with a computer program for
\begin{equation}
\label{eq:4DSVcomp}
\begin{cases}n\leq6, \text{ for $N,r\leq 1$,}\\
n\leq3,\text{ for $N,r\leq 2$,}\\
n\leq2,\text{ for $N,r\leq 3$,}\\
n\leq2,\text{ for $N,r\leq 4$.}
\end{cases}
\end{equation}
The symmetry part was checked for
\begin{equation}
\label{eq:4DSymp}
\begin{cases}n\leq4, \text{ for $N=r=1$,}\\
n\leq3,\text{ for $N=1,r=2$,}\\
n\leq2,\text{ for $N=r=2$,}\\
n\leq2,\text{ for $N=1,r=3$.}\\
\end{cases}
\end{equation}

\subsection{Cohomological limits}
Recall the proof of Theorem \ref{thm:SV-univ-sieres} mainly involved showing that the genus $I_S$ is admissible in the sense of Definition \ref{def:admissible}. Also, by Proposition \ref{prop:series-admissible}, universal series expressions for the Nekrasov genus, and therefore the Segre and Verlinde series, can be obtained if and only if the Nekrasov genus is admissible. Thus one might ask when the Nekrasov genus is admissible. For the rank $r=N$ case, we shall show that admissibility is a consequence of the following explicit formula conjectured by Nekrasov-Piazzalunga \cite[§2.5]{magcolor}. Denote
\[[x]=x^{\frac12}-x^{-\frac12}.\]

\begin{conjecture}[Nekrasov-Piazzalunga]\label{con:nek2}
There exists some choice of signs $o(\mathcal{L})$ such that for $E=\oplus_{i=1}^N\O_X\<y_i\>$, $V=\oplus_{i=1}^N\O_X\<v_i\>$, 
\[{\mathcal{N}}_X(E,V;q)=\Exp\left(\frac{[t_1t_2][t_2t_3][t_1t_3]}{[t_1][t_2][t_3][t_4]}\frac{[s]}{[s^{\frac12}q][s^{-\frac12}q]}\right)\]
with a change of variable $s=\prod_{i=1}^Ny_iv_i$.
\end{conjecture}

\begin{proposition}
Nekrasov-Piazzalunga's Conjecture \ref{con:nek2} implies that the Nekrasov genus $\mathcal{N}_X$ of rank $r=N$ is admissible with respect to the variables $t_1,t_2,t_3,t_4$.
\end{proposition}
\begin{proof}
Expand the term inside the plethystic exponential, specializing with the relation $t_1t_2t_3t_4=1$, we have
\[\frac{[t_1t_2][t_2t_3][t_1t_3][s]}{[t_1][t_2][t_3][t_4][s^{\frac12}q][s^{-\frac12}q]}=\frac{(1-t_1t_2)(1-t_2t_3)(1-t_1t_3)}{(1-t_1)(1-t_2)(1-t_3)(1-t_4)}\cdot \frac{[s]}{[s^{\frac12}q][s^{-\frac12}q]}\]
Recall Definition \ref{def:admissible}, we have 
\[L=(1-t_1t_2)(1-t_2t_3)(1-t_1t_3)\frac{[s]}{[s^{\frac12}q][s^{-\frac12}q]}\]
is a series in $q,y_1^{\pm\frac12},\dots ,y_N^{\pm\frac12},v^{\pm\frac12}_1,\dots ,v^{\pm\frac12}_r$ whose coefficients are polynomials in $t_1,t_2,t_3,t_4$, as required.
\end{proof}

Lastly, we prove the claim made in the introduction that Conjecture \ref{con:cao-kool-quot} is a consequence of Conjecture \ref{con:nek2} in the $X=\CC^4$ case.

\begin{proposition}\label{prop:nek-con-limit}
Let $X=\CC^4$. If Conjecture \ref{con:nek2} holds for some choice of signs, then Conjecture \ref{con:cao-kool-quot} holds for $Y=X$.

In particular, we may retrieve the following well-known identity from Conjecture \ref{con:nek2}:
\equa{\sum_{n=0}^\infty q^n\int_{[\quot_{\CC^4}(E,n)]^\vir_{o(\mathcal{L})}}1:=&\sum_{n=0}^\infty q^n \sum_{Z\in\quot_X(E,n)^\T}(-1)^{o(\mathcal{L})|_Z}\frac{1}{e_\T\left(\sqrt{T^\vir_Z}\right)}\\
=&\begin{cases}
e^{\frac{(\lambda_1+\lambda_2)(\lambda_1+\lambda_3)(\lambda_2+\lambda_3)}{\lambda_1\lambda_2\lambda_3(\lambda_1+\lambda_2+\lambda_3)}q},&\text{ when $N=1$}\\
\hfil1,&\text{ otherwise.}
\end{cases}}
\end{proposition} 
\begin{remark}
One can compare this to the 3-fold case where \cite[Theorem 7.2]{Quot-DT} states
\[\sum_{n=0}^\infty q^n\int_{[\quot_{\CC^3}(E,n)]^\vir}1=M((-1)^Nq)^{-N\frac{(\lambda_1+\lambda_2)(\lambda_1+\lambda_3)(\lambda_2+\lambda_3)}{\lambda_1\lambda_2\lambda_3}}.\]
Here $M$ denotes the MacMahon function.
\end{remark}
\begin{proof}
We shall compute the following limit using both the definition and the expression from Conjecture \ref{con:nek2}, then compare the two sides:
\[\lim_{\substack{\e\->0\\w_N\->\infty}} {\mathcal{N}}_X\left(E,V;\frac Q{w_N}\right)\bigg\vert_{\lam_i\leadsto\e\lam_i,m_i\leadsto \e(1+m_i),w_i\leadsto\e w_i}.\]
Let $V=\oplus_{i=1}^N\O_X\<v_i\>$ be a rank $N$ bundle, then for any $Z_\pi\in\quot_X(E,n)^\T$, we have
\equa{&\frac{\ch_\T(\sqrt{K^\vir|_{Z_\pi}}^{\frac12})}{\ch_\T(\Lambda_{-1} \sqrt{T^\vir|_{Z_\pi}}^\vee)}\ch_\T\left(\frac{\Lambda_{-1}}{\sqrt{\det}} V^{[n]}|_{Z_\pi}\right)\bigg\vert_{\lam_i\leadsto\e\lam_i,m_i\leadsto \e(1+m_i),w_i\leadsto\e w_i}\\
=&\e ^{Nn-Nn}\frac{e_\T(V^{[n]}|_{Z_\pi})+O(\e )}{e_\T\left(\sqrt{T^\vir_{Z_\pi}}\right)+O(\e )}\\
=&\frac{\prod_{i=1}^N\prod_{j=1}^N\prod_{(a,b,c,d)\in\pi^{(j)}}(1+w_i+m_j+a\lambda_1+b\lambda_2+c\lambda_3+d\lambda_4)+O(\e )}{e_\T\left(\sqrt{T^\vir_{Z_\pi}}\right)+O(\e )}
.}
Take limit $\e \->0$ and let $Q=m_Nq$, then
\equa{&\lim_{\e \->0}\frac{\ch_\T(\sqrt{K^\vir|_{Z_\pi}}^{\frac12})}{\ch_\T(\Lambda_{-1} \sqrt{T^\vir|_{Z_\pi}}^\vee)}\ch_\T\left(\frac{\Lambda_{-1}}{\sqrt{\det}} V^{[n]}|_{Z_\pi}\right)\bigg\vert_{\lam_i\leadsto\e\lam_i,m_i\leadsto \e(1+m_i),w_i\leadsto\e w_i}\cdot q^n\\
=& \frac{\prod_{i=1}^N\prod_{j=1}^N\prod_{(a,b,c,d)\in\pi^{(j)}}(1+w_i+m_j-a\lambda_1-b\lambda_2-c\lambda_3-d\lambda_4)}{e_\T\left(\sqrt{T^\vir_{Z_\pi}}\right)}\cdot\frac{Q^n}{m_N^n}\\
=&\frac{\prod_{i=1}^{N-1}\prod_{j=1}^{N}\prod_{(a,b,c,d)\in\pi^{(j)}}(1+w_i+m_j-a\lambda_1-b\lambda_2-c\lambda_3-d\lambda_4)}{e_\T\left(\sqrt{T^\vir_{Z_\pi}}\right)}\\
&\cdot \prod_{j=1}^N\prod_{(a,b,c,d)\in\pi^{(j)}}\left(1+\frac{m_j}{w_N}-\frac{a\lambda_1}{w_N}-\frac{b\lambda_2}{w_N}-\frac{c\lambda_3}{w_N}-\frac{d\lambda_4}{w_N}\right) Q^n
.}
Now take $w_N\->\infty$ and substitute into Definition \ref{def:nek-genus}. Let $V'=\oplus_{i=1}^{N-1}\O_X\<v_i\>$, then
\equanum{\label{eqn:neklim-def}\lim_{\substack{\e \->0\\w_N\->\infty}} {\mathcal{N}}_X\left(E,V;\frac Q{w_N}\right)\bigg\vert_{\lam_i\leadsto\e\lam_i,m_i\leadsto \e(1+m_i),w_i\leadsto\e w_i}={\mathcal{C}}_X(E,V';Q).}

On the other hand, we apply the same procedure to 
\[{\mathcal{N}}_X(E,V;q)=\Exp\left(\frac{[t_1t_2][t_2t_3][t_1t_3]}{[t_1][t_2][t_3][t_4]}\frac{[s]}{[s^{\frac12}q][s^{-\frac12}q]}\right).\]
For $n\geq 1$, we have
\equa{&\lim_{\substack{\e \->0\\w_N\->\infty}}\frac{[t_1^nt_2^n][t_2^nt_3^n][t_1^nt_3^n]}{[t_1^m][t_2^n][t_3^n][t_4^n]}\frac{[\prod y_i^nv_i^n]}{[\prod y_i^{\frac n2}v_i^{\frac n2}q^n][\prod y_i^{-\frac n2}v_i^{-\frac n2}q^{n}]}\bigg\vert_{\lam_i\leadsto\e\lam_i,m_i\leadsto \e(1+m_i),w_i\leadsto\e w_i}\\
=&\lim_{\substack{\e \->0\\w_N\->\infty}}\frac{(\e n)^3(\lambda_1+\lambda_2)(\lambda_1+\lambda_3)(\lambda_2+\lambda_3)+O(\e ^5)}{(\e n)^4\lambda_1\lambda_2\lambda_3(\lambda_1+\lambda_2+\lambda_3)+O(\e ^5)}\cdot \frac{(\e n)\sum_i(1+m_i+w_i)+O(\e )}{(q^{\frac n2}-q^{-\frac n2})^2}\\
=&\lim_{w_N\->\infty}\frac{(\lambda_1+\lambda_2)(\lambda_1+\lambda_3)(\lambda_2+\lambda_3)}{\lambda_1\lambda_2\lambda_3(\lambda_1+\lambda_2+\lambda_3)}\cdot \frac{\sum_i(1+m_i+w_i)(\frac{Q}{w_N})^n}{(1-(\frac{Q}{w_N})^n)^2}\\
=&\begin{cases}
\frac{(\lambda_1+\lambda_2)(\lambda_1+\lambda_3)(\lambda_2+\lambda_3)}{\lambda_1\lambda_2\lambda_3(\lambda_1+\lambda_2+\lambda_3)}Q,&\text{ when $n=1$}\\
\hfil0,&\text{ otherwise.}
\end{cases}
}
Together with (\ref{eqn:neklim-def}), we have
\equa{{\mathcal{C}}_X(E,V';Q)=&e^{\frac{(\lambda_1+\lambda_2)(\lambda_1+\lambda_3)(\lambda_2+\lambda_3)}{\lambda_1\lambda_2\lambda_3(\lambda_1+\lambda_2+\lambda_3)}Q}\\
=&\exp\left(Q\int_Xc_3(X)\right).
}
This is exactly Conjecture \ref{con:cao-kool-quot}.

With the same method, we can take limits
\[\lim_{\substack{\e \->0\\w_N\->\infty}} {\mathcal{N}}_X\left(E,V;\frac Q{w_Nw_{N-1}\dots w_{N-i+1}}\right)\]
for $1<i\leq N$ and $V$ of rank $N-i$, and get
\[{\mathcal{C}}_X(E,V;Q)=1.\]
In particular, when $i=N$ and $N>1$, we have
\equa{\sum_{n=0}^\infty Q^n\int_{[\quot_{\CC^4}(E,n)]^\vir_{o(\mathcal{L})}}1=1.
}
\end{proof}

\bibliographystyle{amsalpha}
\bibliography{refs}

\newcommand{\etalchar}[1]{$^{#1}$}
\providecommand{\bysame}{\leavevmode\hbox to3em{\hrulefill}\thinspace}
\providecommand{\MR}{\relax\ifhmode\unskip\space\fi MR }
\providecommand{\MRhref}[2]{%
  \href{http://www.ams.org/mathscinet-getitem?mr=#1}{#2}
}
\providecommand{\href}[2]{#2}
\begin{thebibliography}{MNOP06b}

\bibitem[AJL{\etalchar{+}}21]{KHilb}
Noah Arbesfeld, Drew Johnson, Woonam Lim, Dragos Oprea, and Rahul
  Pandharipande, \emph{The virtual {K}-theory of {Q}uot schemes of surfaces},
  Journal of Geometry and Physics \textbf{164} (2021), 104154.

\bibitem[Arb21]{Arbesfeld}
Noah Arbesfeld, \emph{{K}-theoretic {D}onaldson–{T}homas theory and the
  {H}ilbert scheme of points on a surface}, Algebraic Geometry (2021), no.~8,
  587--625.

\bibitem[BF98]{BF}
K.~Behrend and Barbara Fantechi, \emph{The intrinsic normal cone}, Inventiones
  Mathematicae \textbf{128} (1998), 45--88.

\bibitem[Boj21a]{Bojko2}
Arkadij Bojko, \emph{Wall-crossing for punctual {Q}uot-schemes},
  arXiv:2111.11102 (2021).

\bibitem[Boj21b]{boj}
\bysame, \emph{Wall-crossing for zero-dimensional sheaves and {H}ilbert schemes
  of points on {C}alabi--{Y}au 4-folds}, arXiv:2102.01056 (2021).

\bibitem[Boj22]{bojko-lag}
\bysame, \emph{Application of {L}agrange inversion to wall-crossing for
  {Q}uot-schemes on surfaces}, arXiv:2111.09868 (2022).

\bibitem[CFK09]{k-local-vir}
Ionut Ciocan-Fontanine and Mikhail Kapranov, \emph{Virtual fundamental classes
  via dg-manifolds}, Geometry and Topology \textbf{13} (2009), 1779–1804.

\bibitem[CK17]{CK1}
Yalong Cao and Martijn Kool, \emph{Zero-dimensional {D}onaldson-thomas
  invariants of {C}alabi-{Y}au 4-folds}, Advances in Mathematics \textbf{338}
  (2017), 601–648.

\bibitem[CK20]{CK2}
\bysame, \emph{Curve counting and {DT/PT} correspondence for {C}alabi-{Y}au
  4-folds}, Advances in Mathematics \textbf{375} (2020), 107371.

\bibitem[CKM22]{CKM}
Yalong Cao, Martijn Kool, and Sergej Monavari, \emph{{K}-theoretic {DT/PT}
  correspondence for toric {C}alabi–{Y}au 4-folds}, Communications in
  Mathematical Physics \textbf{396} (2022), 225–264.

\bibitem[CL14]{Cao:2014bca}
Yalong Cao and Naichung~Conan Leung, \emph{{D}onaldson-{T}homas theory for
  {C}alabi-{Y}au 4-folds}, arXiv:1407.7659 (2014).

\bibitem[Edi97]{edidin}
Dan Edidin, \emph{Equivariant intersection theory}, Inventiones Mathematicae
  \textbf{131} (1997), 595–634.

\bibitem[EG95]{Bott-residue}
Dan Edidin and William Graham, \emph{Localization in equivariant intersection
  theory and the {B}ott residue formula}, American Journal of Mathematics
  \textbf{120} (1995), 619--636.

\bibitem[EGL99]{EGL}
Geir Ellingsrud, Lothar Göttsche, and M.~Lehn, \emph{On the cobordism class of
  the {H}ilbert scheme of a surface}, J. Algebraic Geom \textbf{10} (1999),
  no.~1, 81–100.

\bibitem[Eis95]{Eisenbud}
David Eisenbud, \emph{Commutative algebra}, Springer New York, 1995.

\bibitem[FMR21]{Quot-DT}
Nadir Fasola, Sergej Monavari, and Andrea~T. Ricolfi, \emph{Higher rank
  {K}-theoretic {D}onaldson-{T}homas theory of points}, Forum of Mathematics,
  Sigma \textbf{9} (2021), e15.

\bibitem[GK22]{G_ttsche_2022}
Lothar Göttsche and Martijn Kool, \emph{Virtual {S}egre and {V}erlinde numbers
  of projective surfaces}, Journal of the London Mathematical Society
  \textbf{106} (2022), no.~3, 2562--2608.

\bibitem[GM22]{GM}
Lothar Göttsche and Anton Mellit, \emph{Refined {V}erlinde and {S}egre formula
  for {H}ilbert schemes}, arXiv:2210.01059 (2022).

\bibitem[GP97]{vir-loc}
T.~Graber and R.~Pandharipande, \emph{Localization of virtual classes},
  Inventiones Mathematicae \textbf{135} (1997), 487 – 518.

\bibitem[GSY17]{reduced}
Amin Gholampour, Artan Sheshmani, and Shing-Tung Yau, \emph{Nested {H}ilbert
  schemes on surfaces: {V}irtual fundamental class}, Advances in Mathematics
  \textbf{365} (2017), 107046.

\bibitem[HT08]{HT}
Daniel Huybrechts and Richard~P. Thomas, \emph{Deformation-obstruction theory
  for complexes via {A}tiyah and {K}odaira–{S}pencer classes}, Mathematische
  Annalen \textbf{346} (2008), 545--569.

\bibitem[Joh18]{DJon}
Drew Johnson, \emph{Universal series for {H}ilbert schemes and strange
  duality}, International Mathematics Research Notices \textbf{2020} (2018),
  no.~10, 3130--3152.

\bibitem[JOP21]{johnson}
Drew Johnson, Dragos Oprea, and Rahul Pandharipande, \emph{Rationality of
  descendent series for {H}ilbert and {Q}uot schemes of surfaces}, Selecta
  Mathematica \textbf{27} (2021), 23.

\bibitem[KL20]{KLcosectionKtheory}
Young-Hoon Kiem and Jun Li, \emph{Localizing virtual structure sheaves by
  cosections}, Int. Math. Res. Not. IMRN (2020), no.~22, 8387--8417.

\bibitem[KR]{KR-draft}
Martijn Kool and J{\o}rgen Rennemo, \emph{In preparation}.

\bibitem[KT14]{KT}
Martijn Kool and R.~Thomas, \emph{Reduced classes and curve counting on
  surfaces {I}: {T}heory}, Algebraic Geometry \textbf{1} (2014), no.~3,
  334–383.

\bibitem[Leh99]{Lehn_1999}
Manfred Lehn, \emph{{C}hern classes of tautological sheaves on {H}ilbert
  schemes of points on surfaces}, Inventiones Mathematicae \textbf{136} (1999),
  no.~1, 157--207.

\bibitem[Lim21]{lim}
Woonam Lim, \emph{Virtual $\chi_y$-genera of {Q}uot schemes on surfaces},
  Journal of the London Mathematical Society \textbf{104} (2021), no.~3,
  1300--1341.

\bibitem[LT96]{LiTian}
Jun Li and Gang Tian, \emph{Virtual moduli cycles and {G}romov-{W}itten
  invariants of algebraic varieties}, Journal of the American Mathematical
  Society \textbf{11} (1996), 119–174.

\bibitem[LYZ02]{lyz}
Kefeng Liu, Catherine Yan, and Jian Zhou, \emph{Hirzebruch $\chi_y$ genera of
  the {H}ilbert schemes of surfaces by localization formula}, Science
  China-mathematics - SCI CHINA-MATH \textbf{45} (2002), 420--431.

\bibitem[Mel18]{mel-HLV}
Anton Mellit, \emph{Integrality of {H}ausel-{L}etellier-{V}illegas kernels},
  Duke Mathematical Journal \textbf{167} (2018), no.~17, 3171--3205.

\bibitem[MNOP06a]{MNOP1}
D.~Maulik, N.~Nekrasov, A.~Okounkov, and R.~Pandharipande,
  \emph{{G}romov–{W}itten theory and {D}onaldson–{T}homas theory, i},
  Compositio Mathematica \textbf{142} (2006), no.~5, 1263--1285.

\bibitem[MNOP06b]{MNOP2}
\bysame, \emph{{G}romov–{W}itten theory and {D}onaldson–{T}homas theory,
  ii}, Compositio Mathematica \textbf{142} (2006), no.~5, 1263--1285.

\bibitem[Mon22]{monavari}
Sergej Monavari, \emph{Canonical vertex formalism in {DT} theory of toric
  {C}alabi-{Y}au 4-folds}, Journal of Geometry and Physics \textbf{174} (2022),
  104466.

\bibitem[MOP15]{MOP1}
Alina Marian, Dragos Oprea, and Rahul Pandharipande, \emph{{S}egre classes and
  {H}ilbert schemes of points}, Annales scientifiques de l'École normale
  supérieure \textbf{50} (2015), 239--267.

\bibitem[MOP17]{MOP}
\bysame, \emph{The combinatorics of {L}ehn's conjecture}, Journal of the
  Mathematical Society of Japan \textbf{71} (2017), no.~1, 299--308.

\bibitem[MOP21]{Marian_2021}
Alina Marian, Dragos Oprea, and Rahul Pandharipande, \emph{Higher rank {S}egre
  integrals over the {H}ilbert scheme of points}, Journal of the European
  Mathematical Society \textbf{24} (2021), no.~8, 2979--3015.

\bibitem[Nek20]{Mag}
Nikita Nekrasov, \emph{Magnificent four}, Advances in Theoretical and
  Mathematical Physics \textbf{24} (2020), 1171--1202.

\bibitem[NO14]{NO}
N.~Nekrasov and Andrei Okounkov, \emph{Membranes and sheaves}, Algebraic
  Geometry \textbf{3} (2014), no.~3, 320–369.

\bibitem[NP19]{magcolor}
N.~Nekrasov and Nicolò Piazzalunga, \emph{Magnificent four with colors},
  Communications in Mathematical Physics \textbf{372} (2019), 573–597.

\bibitem[OP22]{OP}
Dragos Oprea and Rahul Pandharipande, \emph{{Q}uot schemes of curves and
  surfaces: virtual classes, integrals, {E}uler characteristics}, Geometry \&
  Topology \textbf{25} (2022), 3425--3505.

\bibitem[OT20]{Oh:2020rnj}
Jeongseok Oh and Richard~P. Thomas, \emph{Counting sheaves on {C}alabi-{Y}au
  4-folds, {I}}, arXiv:2009.05542 (2020).

\bibitem[Ric20]{RICOLFI}
Andrea~T. Ricolfi, \emph{Virtual classes and virtual motives of quot schemes on
  threefolds}, Advances in Mathematics \textbf{369} (2020), 107182.

\bibitem[Ric21]{Ri}
Andrea Ricolfi, \emph{The equivariant {A}tiyah class}, Comptes Rendus.
  Mathématique \textbf{359} (2021), 257--282.

\bibitem[RS21]{HRR-vir}
Charanya Ravi and Bhamidi Sreedhar, \emph{Virtual equivariant
  {G}rothendieck-{R}iemann-{R}och formula}, Documenta Mathematica \textbf{26}
  (2021), 2061--2094.

\bibitem[Sta21]{stark}
Samuel Stark, \emph{On the {Q}uot scheme $\mathrm{Quot}^{l}(\mathscr{E})$},
  arXiv:2107.03991 (2021).

\bibitem[Tyu94]{Tyurin_1994}
A.~Tyurin, \emph{Spin polynomial invariants of smooth structures on algebraic
  surfaces}, Russian Academy of Sciences. Izvestiya Mathematics \textbf{42}
  (1994), 333.

\end{thebibliography}


\end{document}